\documentclass [12pt, draft]{article}

\usepackage{amsmath, amsfonts, amsthm, amssymb, latexsym}
\usepackage{mathtools}
\usepackage[bookmarks]{hyperref}
\usepackage[top=3.0cm, bottom=3.0cm, left=2.5cm, right=2.0cm]{geometry}
\usepackage{graphicx}
\usepackage[utf8]{inputenc}
\usepackage[allfiguresdraft]{draftfigure}
\theoremstyle{plain}
\usepackage[all]{xypic}
\usepackage{enumerate,longtable}
\newtheorem{thm}{Theorem}[section]
\newtheorem{cor}[thm]{Corollary}

\newtheorem{prop}[thm]{Proposition}
\theoremstyle{definition}
\newtheorem{defn}[thm]{Definition}
\theoremstyle{remark}
\newtheorem{rem}[thm]{Remark}

\numberwithin{equation}{section}
\usepackage{color, fancyhdr, graphicx, wrapfig}
\usepackage{indentfirst}

\newcommand{\G}{\mathcal{G}}

\newcommand{\R}{\mathbb{R}}

\newcommand{\X}{\mathfrak{X}}
\newcommand{\g}{\mathfrak{g}}
\newcommand{\Aut}{\operatorname{Aut}}

\newcommand{\End}{\operatorname{End}}
\newcommand{\ad}{\operatorname{ad}}  
 
\newcommand{\Ad}{\operatorname{Ad}}
\newcommand{\Rank}{\operatorname{rank}}
\newcommand{\const}{\operatorname{const}}

\begin{document}
	\title{\bf Foliations formed by generic coadjoint orbits of a class of 7-dimensional real solvable Lie groups}
	
	\author{{\bf Nguyen Tuyen}
		\footnote{
			%		Funded by Vingroup Joint Stock Company and supported by the Domestic Master/PhD Scholarship Programme of Vingroup Innovation Foundation (VINIF), Vingroup Big Data Institute (VINBIGDATA) under grant VINIF.2020.TS.46, and by the project SPD2019.01.37; 		Address: 
			Faculty of Mathematics and Computer Science,
			University of Science -- Vietnam National University, Ho Chi Minh City, Vietnam / Faculty of Mathematics and Computer Science Teacher Education, Dong Thap University, Dong Thap Province, Vietnam; 
			E-mail: ntmtuyen@dthu.edu.vn}, 
		{\bf Le Vu}
		\footnote{Department of Mathematics and Economic Statistics, University of Economics and Law, Vietnam National University - Ho Chi Minh City; E-mail:\,vula@uel.edu.vn}}
	%\\${}^{*}$ \scriptsize{Department of Mathematics and Economic Statistics, University of Economics and Law}
	%\\\scriptsize{}
	%\\\scriptsize{}}
	\date{}
	\maketitle
%----------additions
%\dedicatory{To my boss}
%%% ----------------------------------------------------------------------

\begin{abstract}
	In this paper, we consider exponential, connected and simply connected Lie groups which are corresponding to Lie algebras of dimension 7 such that the nilradical of them is 5-dimensional nilpotent Lie algebra $\g_{5,2}$ in Table \ref{tab1}. In particular, we give a description of the geometry of the generic orbits in the coadjoint representation of some considered Lie groups. We prove that, for each considered group, the family of the generic coadjoint orbits forms a measurable foliation in the sense of Connes. The topological classification of these foliations is also provided.\\[0.2cm]
	
	\noindent\textsl{MSC}: 53C12, 17B08, 22E27, 57R30, 17B30, 22E45 \\
	\textsl{Keywords}: Lie algebra; Lie group; K-orbit; foliation; measurable foliation.
\end{abstract}

%%% ----------------------------------------------------------------------
\maketitle
%%% ----------------------------------------------------------------------
%\tableofcontents
\section{INTRODUCTION}\label{sec1}
The study of foliations on manifolds has a long history in mathematics. Since the works of C. Ehresmann and G. Reeb \cite{ER44} in 1944, of C. Ehresmann \cite{E51} in 1951 and of G. Reeb \cite{Re52} in 1952, the foliations on manifolds have enjoyed a rapid development. Nowaday, they become the focus of a great deal of research activity (see \cite{LA74} of H. B. Lawson, Jr.).
In general, the leaf space of a foliation with quotient topology is a fairly intractable topological space. To improve upon the shortcoming, A. Connes \cite{Con82} proposed the notion of measurable foliations in 1982 and associated each such foliation with a $C^*$-algebra. During the last few decades, these concepts of A. Connes have become important tools of non-commutative differential geometry and have attracted much attention from mathematicians around the world. 

%On the other hand, A. A. Kirillov invented in 1962 
In 1962, A. A. Kirillov invented
the method of orbits and it quickly became the most important method in the theory of representations of Lie groups and Lie algebras (see \cite{K76}, Section 15). The key of Kirillov's method of orbits is generic orbits in the coadjoint representation ($K$-orbits for short) of Lie groups. Hence, the problem of describing the geometry of (generic) $K$-orbits of each Lie group is very important to study.

In 1980, for studying Kirillov's method of orbits, Do Ngoc Diep suggested considering the class of MD-groups. For any positive natural number $n$, an MD-group of dimension $n$ (for brevity, an MD$n$-group) in his terms (see \cite{D99}, Section 4.1) is a $n$-dimensional solvable real Lie group whose $K$-orbits are the orbits of zero or maximal dimension. The Lie algebra of each MD$n$-group is called an MD$n$-algebra. It is noticed that, for every MD-group $G$ (of arbitrary dimension), the family of $ K $-orbits of maximal dimension forms a measured foliation in terms of A. Connes \cite{Con82}. This foliation is called MD-foliation associated with $G$.

In recent decades, the second authors of this paper and his colleagues have combined the methods of Kirillov and Connes in studying the geometry of $ K $-orbits as well as the MD-foliation associated with several low dimension MD-groups.
From 1987 to 1993, the completely solved this problem for the class of all MD4-groups in \cite{V87, V90, Vu90, V93} and MD4-foliations. In the period 2008--2014, similar results have been proposed for MD5-groups and MD5-foliations in \cite{VT08, VH09, VHT14}. Although several partial results have been investigated, the general properties of MD-class and the complete classification of MD-algebras and MD-foliations is still open.

This paper can be considered as a continuation of \cite{VTTTT21} in which a full classification of 7-dimensional real solvable Lie algebras having a 5-dimensional nilradical was achieved. For solvable Lie algebra $\G$ of dimension 7, the nilradical $N(\G)$ are of dimension from 4 to 7 by Mubarakzyanov \cite[Theorem 5]{Mub63a}. Those with the nilradical of dimension 4, 6 and 7 are classified in \cite{Gon98, HT08, Par07}. When $\dim N(\G)=5$, there are nine cases listed in Table \ref{tab1} below \cite[Proposition 1]{Dix58}. In this table, $N(\G)$ is generated by $\{X_1, X_2, X_3, X_4, X_5\}$.
\begin{center}
	\begin{longtable}{p{.050\textwidth} p{.140\textwidth} p{.630\textwidth}}
		\caption{5-dimensional nilpotent Lie algebras}\label{tab1} \\
		\hline Case& Algebras & Non-zero Lie brackets \endfirsthead 
		\caption*{Table 1 (continued)} \\ 
		\hline Case& Algebras & Non-zero Lie brackets 	\\	
		\hline \endhead \hline
		\endlastfoot \hline
		1	&$(\g_1)^5$ & The 5-dimensional abelian Lie algebra \\  
		2&$(\g_1)^2 \oplus \g_3$ & $[X_1, X_2] = X_3$ \\  
		3&$\g_1 \oplus \g_4$ & $[X_1, X_2] = X_3$, $[X_1, X_3] = X_4$ \\ 
		4&$\g_{5,1}$ & $[X_1, X_2] = X_5$, $[X_3, X_4] = X_5$ \\
		5&	$\g_{5,2}$ & $[X_1, X_2] = X_4$, $[X_1, X_3] = X_5$ \\
		6&	$\g_{5,3}$ & $[X_1, X_2] = X_4$, $[X_1, X_4] = X_5$, $[X_2, X_3] = X_5$ \\
		7&	$\g_{5,4}$ & $[X_1, X_2] = X_3$, $[X_1, X_3] = X_4$, $[X_2, X_3] = X_5$ \\
		8&	$\g_{5,5}$ & $[X_1, X_2] = X_3$, $[X_1, X_3] = X_4$, $[X_1, X_4] = X_5$ \\
		9&	$\g_{5,6}$ & $[X_1, X_2] = X_3$, $[X_1, X_3] = X_4$, $[X_1, X_4] = X_5$, $[X_2, X_3] = X_5$ \\ \hline
	\end{longtable} 
\end{center}
The cases 1, 4, 6, 8 and 9 were classified in \cite{NW94, RW93, SK10, SW05, SW09,  WLD08}. The remaining cases are classified in \cite{VTTTT21}, thus a full classification of 7-dimensional indecomposable solvable Lie algebras were achieved.

Combining the idea of the Kirillov's method of orbits and the Connes' method, we hope that some beautiful properties of MD-groups can be generalized for a larger class of solvable Lie groups, including all Lie groups corresponding to Lie algebras classified by the authors et al. in \cite{VTTTT21}. More precisely, we would like to study the properties that are similar to those of MD-groups for the Lie groups corresponding to Lie algebras listed in \cite{VTTTT21}. Among these Lie groups, the ones of Lie algebras in Case 5 of Table \ref{tab1} will be chosen to study in this paper because they are not trivial and not too complicated. 
The main results of the paper are as follows. First, we describe maximal dimensional $ K $-orbits of the considered Lie groups which are exponential, connected and simply connected. Second, we proved that the family of all generic maximal dimensional $K$-orbits of the considered Lie groups forms measurable foliations (in the sense of Connes \cite{Con82}). Finally, we give the topological classification of these foliations and describe the Connes' C*-algebras of them. Furthermore, the method is applicable for all of the remaining Lie algebras listed in \cite{VTTTT21} and they will be studied in detail later on.

The paper is organized into four sections, including this introduction. In Section 2 we will recall some preliminary notions and properties which will be used throughout the paper. In particular, we will introduce in that section the list from \cite{VTTTT21} of 7-dimensional solvable Lie algebras having nilradical $\g_{5,2}$. Section 3 will be devoted to setting and proving the main results of the paper. We also give some concluding remarks in the last Section.

\section{PRELIMINARIES}\label{sec2}
%Subsection 2.1
\subsection{The Coadjoint Representation and K-orbits of a
	Lie Group}

Let $G$ be a Lie group, $\G$ be its Lie algebra, and $\G^*$ be the dual space to $\G$. Recall that $\left\langle F,X\right\rangle$ will denote the value of every $F \in \G^*$ at any $X \in \G$. 

%Definition 2.1
\begin{defn}(see \cite[$\S$15]{K76})
	The action
	\[\Ad: G \rightarrow \Aut(\G), \, g \mapsto Ad(g)=(L_g \circ R_{g^{-1}})_* ; \, \, \forall g \in G \]
	%	\begin{align*}
	%	\Ad: & \quad G \longrightarrow \Aut(\G)\\
	%	& \quad g \longmapsto Ad(g)=(L_g \circ R_{g^{-1}})_*: \G \longrightarrow \G, \quad \forall g \in G
	%	\end{align*}
	where $L_g, \, R_{g^{-1}}$ are left-translation, right-translation by $g, g^{-1}$ in $G$, respectively and $h_*$ means the differential (or tangent map) of any diffeomorphism $h$ of $G$. The operator $Ad$
	is called {\it the adjoint representation} of $G$ in $\G$.	
\end{defn}	

The tangent map $\ad: = \Ad_{*}$ of $\Ad$ is called {\it the adjoint representation} of $\G$. It is well-known that, $ \ad: \G \to \mathfrak{gl}(\G), \, X \mapsto \ad_X$ is the homomorphism defined as follows
\begin{equation}\label{ad_X}
\ad_X(Y):= [X,Y] \quad \mbox{for all } Y \in \G. \tag{1}
\end{equation}

%Definition 2.2
\begin{defn}(see \cite[$\S$15]{K76})
	The {\it coadjoint representation} ({\it $K$-representation} for short) of $G$ in $\G^*$
	\[K: G \rightarrow \Aut(\G^*), \, \, g \mapsto K(g) \]
	%	\begin{align*}
	%	K: & \quad G \longrightarrow \Aut(\G^*)\\
	%	& \quad g \longmapsto K(g)
	%	\end{align*}
	is defined by
	\[\left\langle K(g)(F),X\right\rangle : = \left\langle F,Ad(g^{-1}) (X)\right\rangle ; \, \, F\in \G^*, X \in \G\]
	where $\langle F, Y \rangle $ denotes the value of $F \in \G^*$ on $Y \in \G$.
\end{defn}

%Definition 2.3
\begin{defn}(see \cite[$\S$15]{K76})
	Each orbit of $G$ in the $K$-representation is called
	a {\it coadjoint orbit}  or {\it $K$-orbit} for short.
\end{defn}
\noindent For every $F \in \G^*$, the $K$-orbit containing $F$ is denoted by $\Omega_F$ and we have
\[\Omega_F=\{K(g)(F) \; | \; g \in G\}.\]
Note that the dimension of each $K$-orbit of $G$ is always even (see \cite[$\S$15]{K76}).

An important role in the study of Lie groups is played by the one-parameter subgroup $x(t)$ ($t \in \R$) that corresponds to every $X \in \G$ (see \cite[$\S$6]{K76}). We have the following definition.

%Definition 2.4
\begin{defn}(see \cite[$\S$6]{K76})
	The {\it exponential mapping} ${\rm exp}: \G \rightarrow G$ is defined by ${\rm exp} (X): = x(1)$ for any $X \in \G$. A Lie group $G$ is called {\it exponential} if and only if its exponential mapping $\exp: \G \mapsto G$ is a diffeomorphism.
\end{defn}

%Proposition 2.5
\begin{prop}[see \cite{Sa57}]\label{pro exp}
	Assume that $G$ is finite-dimensional connected and simply connected  real solvable Lie groups. Then, the following assertion are equivalent
	\begin{enumerate}
		\item [(i)] $G$ is exponential. 
		\item [(ii)] For any $X \in \G=Lie(G)$, $\ad_X$ has no nonzero imaginary eigenvalues.
	\end{enumerate}
\end{prop}

As mentioned above, for each Lie group $G$, the problem of describing the geometry of $K$-orbits of $G$ is very important to study. Here in this paper, 
we want to have a method of description in the case where the group law of $ G $ is not explicitly stated but only knows the structure of Lie algebra $\G$ of $G$. Then, the exponential mapping and its naturality is very helpful to us. 

Recall that the group $\Aut(\G)$ of all automorphisms of $\G$ is also a Lie group and its Lie algebra is exactly the algebra $\End(\G)$ of all endomorphisms on $\G$. Let $\exp_G: \G \longrightarrow G$ and $\exp: \End(\G) \longrightarrow \Aut(\G)$ be the exponential mappings of $G$ and $ \Aut(\G) $, respectively. In fact, we have commutative rectangle as follows: 
\[
\xymatrix
{
	G \ar[r]^{{\rm Ad} \hspace{20pt}} & {\rm Aut} \left( \mathcal{G} \right) \\
	\mathcal{G} \ar[r]^{{\rm ad} \hspace{20pt}} \ar[u]^{{\rm exp}_G} & {\rm End} \left( \mathcal{G} \right) \ar[u]_{{\rm exp}}
}
\]
That means $\Ad \circ \exp_G=\exp \circ \ad$ (see \cite[$\S$6]{K76}).

For each $U \in \G$, each $F \in \G^*$, we determine $F_U \in \G^*$ as follows:
\[\left\langle F_U, X \right\rangle = \left\langle F, \left( {\rm exp} \circ {\rm ad}_U \right) \left( X \right) \right\rangle; \, \, \forall X \in \G\]
and denote $\Omega_F(\G)=\left\lbrace F_U: U \in \mathcal{G} \right\rbrace.$

%Proposition 2.6
\begin{prop}[see \cite{V87,  V90, Vu90}]\label{2.1}
	If $\Omega_F$ is the $K$-orbit of $G$ passing $F$ then 
	\[
	\Omega_F \supset \Omega_F(\G).
	\]
	Furthermore, if $\exp_G$ is surjective then $ \Omega_F = \Omega_F(\G).$
\end{prop}

%Corollary 2.7
\begin{cor}[see \cite{V87, V90, Vu90}]\label{cor exp}
	If $G$ is an exponential Lie group then $\Omega_F = \Omega_F(\G)$ for every $F \in \G^*$. % Assume that $G$ is a finite-dimensional connected and simply connected real solvable Lie group and $\G=Lie(G)$. If $G$ is exponential then $ \Omega_F = \Omega_F(\G)$ for every $F \in \G^*$. 
\end{cor}

In order to define the dimension of the $K$-orbits
${\Omega}_{F}$ for each $F \in \G^*$, it is useful to
consider the skew-symmetric bilinear form $B_{F}$ on
$\G$ as~follows
\begin{equation}\label{eq-BF}
B_{F}(X, Y) := \langle F, [X, Y]\rangle; \, \forall\, X, Y \in
\G. \tag{2}
\end{equation}

Denote the stabilizer of $F$ under the $K$-representation of
$G$ in $\mathcal{{G}^{*}}$ by $G_{F}$ and ${\G}_{F}$ :=
Lie($G_{F}$). We shall need in the sequel the following result.

%Proposition 2.4
\begin{prop} [see \cite{K76}, $\S$15]\label{Prop2.4} For any element $F \in \G^*$, we have
	\begin{enumerate}
		\item[(i)] $\ker B_{F} = {\G}_{F}$.
		\item[(ii)] $\dim{\Omega}_{F} = \dim \G - \dim{\G}_{F} = \Rank B_F$.
	\end{enumerate} 
\end{prop}

%Subsection 2.2
\subsection{Foliations and Measurable Foliations}\label{Subsection2.2}

Let $V$ be an $n$-dimensional smooth manifold ($0 < n \in \mathbb{N}$). We always denote its tangent bundle by $TV$, the tangent space of $V$ at $x \in V$ by $T_xV$. 

%Definition 2.9
\begin{defn}(see \cite[Introduction]{Con82})
	A smooth subbundle $\mathcal{F}$ of $TV$ is called {\it integrable} if and only if every $x \in V$ is contained in a submanifold $W$ of $V$ such that $T_pW = \mathcal{F}_p; \forall p \in W$.
\end{defn} 

%Definition 2.10
\begin{defn}(see \cite[Introduction]{Con82})
	A {\it foliation} $(V, \mathcal{F})$ is defined by a smooth manifold $V$ and an integrable subbundle $\mathcal{F}$ of $TV$. Then, $V$ is called the {\it foliated manifold} and $\mathcal{F}$ is called the {\it subbundle defining the foliation}. The dimension of $\mathcal{F}$ is also called the {\it dimension of the foliation} $(V, \mathcal{F})$ and $n - \dim \mathcal{F}$ is called the {\it codimension of the foliation} $(V, \mathcal{F})$ in $V$. Each maximal connected submanifolds $L$ of $V$ such that $T_xL = \mathcal{F}_x (\forall x \in L)$ is called a {\it leaf} of the foliation $(V, \mathcal{F})$.
\end{defn} 

The set of leaves with the quotient topology of $V$ is denoted by $V/\mathcal{F}$ and called the {\it space of leaves} or the {\it leaf space} of $(V, \mathcal{F})$. In general, it is a fairly intractable topological space.

The partition of $V$ in leaves: 
$V=\bigcup_{a \in V/\mathcal{F}}L_{a}$
is geometrically characterized by the following local triviality: each $x \in V$ has a system of local coordinates $\{U; x^1, x^2, \dots, x^n\}$ such that $x \in U$ and for any leaf $L$ with $L \cap U \neq \emptyset$, each connected component of $L \cap U$ (which is called a plaque of the leaf $L$) is given by the equations
$x^{k+1}=c^1, x^{k+2}=c^2, \dots, x^n=c^{n-k},$
where $k=\dim \mathcal{F} < n$ and $c^1, c^2, \dots , c^{n-k}$ are suitable constants. Each such system $\{U, x^1, x^2, \dots, x^n\}$ is called a foliation chart.

A $k$-dimensional foliation can be given by a partition of $V$ in a family $\mathcal{C}$ of its $k$-dimensional submanifolds ($k \in \mathbb{N}, 0 < k < n$) if there exists some integrable $k$-dimensional subbundle $\mathcal{F}$ of $TV$ such that each $L \in \mathcal{C}$ is a maximal connected integral submanifold of $\mathcal{F}$. In this case, $\mathcal{C}$ is the family of leaves of the foliation $(V, \mathcal{F})$. Sometimes $\mathcal{C}$ is identified with $\mathcal{F}$ and we say that $(V, \mathcal{F})$ is formed by $\mathcal{C}$.

%Definition 2.11
\begin{defn}(see \cite[Preliminaries]{Aziz07})\label{ToPo}
	Two foliations $(V_1, \mathcal{F}_1)$ and $(V_2, \mathcal{F}_2)$ are said to be topologically equivalent if and only if there exists a homeomorphism $h: V_1 \rightarrow V_2$ such that for every leaf $L$ of $\mathcal{F}_1$, $h(L)$ is also a leaf of $\mathcal{F}_2$. In other words, $h$ sends leaves of $\mathcal{F}_1$ onto those of $\mathcal{F}_2$. The mapping $h$ is called a topological equivalence of considered foliations. 
\end{defn}

%Definition 2.12
\begin{defn}(see \cite[Section 1]{Con82}) A submanifold $N$ of the foliated manifold $V$ is called a {\it transversal} if and only if $T_{x}V = T_{x}N \oplus {\mathcal{F}}_{x},\, \forall x \in N$. Thus,
	$\dim N = n - \dim \mathcal{F} = codim \mathcal{F}$.	
	A Borel subset $B$ of $V$ such that $B\cap L$ is countable for any leaf L
	is called a {\it Borel transversal} to $(V, \mathcal{F})$.
\end{defn}

%Definition 2.13
\begin{defn}(see \cite[Section 1, 2]{Con82}) A {\it transverse measure} $\Lambda$ for the foliation $(V, \mathcal{F})$ is $\sigma$-additive map 
	$B \mapsto \Lambda (B)$ from the set of all Borel transversals to [0,
	+$\infty$] such that the following conditions are satisfied:
	
	(i) If $\psi$ : $B_{1} \rightarrow B_{2}$ is a Borel bijection and
	$\psi$(x) is on the leaf of any x$\in B_{1}$, then $\Lambda(B_{1}) = \Lambda(B_{2})$.
	
	(ii) $\Lambda(K)< + \infty$ if $K$ is any compact subset of a smooth transversal submanifold of $V$.
	
	By a {\it measurable foliation} we mean a foliation
	$(V, \mathcal{F})$ equipped with some transverse measure~$\Lambda$.\end{defn}

Let $(V, \mathcal{F})$ be a foliation with $\mathcal{F}$ is
oriented. Then the complement of zero section of the bundle
${\Lambda}^{k}(\mathcal{F})$ ($0 < k = \dim \mathcal{F} < n$) has two
components ${\Lambda}^{k}{(\mathcal{F})}^{-}$ and
${\Lambda}^{k}{(\mathcal{F})}^{+}$.
Let $\mu$ be a measure on $V$ and  $\{ U, x^{1}, x^{2}, ..., x^{n} \}$
be a foliation chart of $(V, \mathcal{F})$.
Then $U$ can be identified with the direct product $N \times {\Pi}$
of some smooth transversal submanifold $N$ of $V$ and some plaque
$\Pi$. The restriction of $\mu$ on $U \equiv N \times {\Pi}$ becomes
the product ${\mu}_{N} \times {\mu}_{\Pi}$ of measures ${\mu}_{N}$
and ${\mu}_{\Pi}$ respectively.
Let $X \in C^{\infty}{\bigl({\Lambda}^{k}(\mathcal{F})\bigr)}^{+}$
be a smooth $k$-vector field and ${\mu}_{X}$ be the measure on each leaf L
determined by the volume element $X$.

%Definition 2.14
\begin{defn}(see \cite[Sections 1, 2]{Con82}) The measure $\mu$ is called
	{\it $X$-invariant} if and only if ${\mu}_{X}$ is proportional to
	${\mu}_{\Pi}$ for an arbitrary foliation chart $\{ U, x^{1}, x^{2},
	..., x^{n} \}$.\end{defn}

%Definition 2.15
\begin{defn}(see \cite[Sections 1, 2]{Con82}) Let ($X, \mu$) and  ($Y, \nu$) be two pairs where $X,Y$ $ \in$  $C^{\infty}{\bigl({\Lambda}^{k}(\mathcal{F})\bigr)}^{+}$ and $\mu, \nu$	are measures on $V$ such that $\mu$ is $X$-invariant, $\nu$ is $Y$-invariant. Then ($X, \mu$ ), ($Y, \nu$ ) are called {\it equivalent} if and only if $Y = \varphi X$ and $\mu = \varphi \nu$ for some $\varphi \in C^{\infty}(V).$
\end{defn}

There is a bijection between the set of transverse measures for
$(V, \mathcal{F})$ and the set of equivalent classes of pairs ($X, \mu$), where $X \in C^{\infty}{\bigl({\Lambda}^{k}(\mathcal{F})\bigr)}^{+}$ and $\mu$ is a $X$-invariant measure on $V$. Thus, to prove that $(V, \mathcal{F})$ is measurable, we only need to choose some suitable pair ($X, \mu$) on $V$.

%Definition 2.16
%\begin{defn}[see \cite{NIZ13}]\label{Proper}
%	A leaf $L$ of a foliation $(V,F)$ is said to be {\it proper} if it is embedded submanifold of the foliated manifold $V$. A foliation is {\it proper}, if every its leaf is proper. %A leaf $L$ is called closed if $L$ is a closed subset of $M$.
%\end{defn}	

%Subsection 2.3
\subsection{$7$-dimensional Solvable Lie Algebras Having Nilradical $\g_{5,2}$}\label{ListLieAlgebras}

For convenience, we shall use notation $\G$ to replace the notation $L$ or $R$ in \cite{VTTTT21}. Specifically, we consider the set \{$\G^{\lambda}_{1}$, $ \G_{2} $, $ \G_{3} $, $ \G_4^{\lambda_1,\lambda_2} $, $\G_{5}$, $ \G_6^{\lambda} $, $ \G_{7} $, $ \G_8^{\lambda} $, $ \G_{9} $, $ \G_{10}^{\lambda} $, $ \G_{11} $, $\G_{12}^{\lambda}$,  $\G_{13}^{\lambda}$, $\G_{14}^{\lambda_1,\lambda_2}$, $\G_{15}$, $\G_{16}^{\lambda}$\} of $7$-dimensional complex or real indecomposable solvable Lie algebras having nilradical $\g_{5,2}$ which are listed by the authors et al. in Table 3 of \cite{VTTTT21}. Each $\G$ has
basis $(X_1, X_2, X_3, X_4, X_5, X, Y)$ and its nilradical is $\g_{5,2}$ in Table \ref{tab1}. %$=\s\{X_1, X_2, X_3, X_4, X_5: [X_1,X_2]=X_4, [X_1,X_3]=X_5\}$. 
The operator matrices $a_X:=(\ad_X)^T|_{\g_{5,2}}$, $a_Y:=(\ad_Y)^T|_{\g_{5,2}}$ $\in Mat(5,\R)$ and $[X,Y]$ are given in Table \ref{tab2} as follows

%List of 16 families of Lie algebras
\begin{longtable}{p{.060\textwidth} p{.500\textwidth} p{.200\textwidth}}
	\caption{Solvable Lie algebras with nilradical $\g_{5,2}$}\label{tab2} \\
	\hline No. & $(a_X, a_Y, [X,Y])$ & Condition \endfirsthead
	\caption*{Table 1: Solvable Lie algebras with nilradical $\g_{5,2}$ (continued)} \\ 
	\hline No. & $(\ad_X, \ad_Y, [X,Y])$ & Condition \\ \hline \endhead \hline
	$\G_1^\lambda$ & $(1, -1, 0, 0, 1)$, $(0, 0, 1, 0, 1)$, $\lambda X_4$ & $\lambda\in\{0,1\}$ \\ 
	$\G_2$ & $(1, 0, 0, 1, 1)$, $(0, 0, 1, 0, 1)$ &\\ 
	$\G_3$ & $(0, 1, 0, 1, 0)$, $(0, 0, 1, 0, 1)$& \\ 
	$\G^{\lambda_1,\lambda_2}_4$ & $(1, 0, \lambda_1, 1, 1+\lambda_1)$, $(0, 1, \lambda_2, 1, \lambda_2)$ &  $(\lambda_{1},\lambda_2)\neq (-1,0)$ \\ 
	$\G_5$ & $(0, 0, 1, 0, 1)$, $(1, 1, 0, 2, 1) + E_{12}$& \\ 
	$\G^{\lambda}_6$ & $(1, 1, \lambda, 2, 1+\lambda)$, $(0, 0, 1, 0, 1) + E_{12}$ & $\lambda\in\R$\\ 
	$\G_7$ & $(0, 1, 1, 1, 1)$, $(1, 1, 0, 2, 1) + E_{25}$& \\ 
	$\G^{\lambda}_8$ & $(1, 1+\lambda, \lambda, 2+\lambda, 1+\lambda)$, $(0, 1, 1, 1, 1) + E_{25}$& $\lambda\in\R$\\ 
	$\G_9$ & $(0, 0, 1, 0, 1)$, $(0, 1, 0 , 1, 0) + E_{35}$& \\ 
	$\G^{\lambda}_{10}$ & $(0 , 1 , \lambda, 1, \lambda)$, $(0, 0, 1, 0, 1) + E_{35}$ & $\lambda\in\R$ \\ 
	$\G_{11}$ & $(0, 1, 1, 1, 1)$, $(1, 0, 0, 1, 1) + E_{23} + E_{45}$& \\ 
	$\G^{\lambda}_{12}$ & $(1, \lambda, \lambda, 1+\lambda, 1+\lambda)$, $(0, 1, 1, 1, 1) + E_{23} + E_{45}$& $\lambda\in\R$\\ 
	$\G_{13}^{\lambda}$ & $(0, 1, 1, 1, 1)$, $(\lambda, S_{01}, S_{\lambda1})$ & $\lambda \geq 0$ \\ 
	$\G_{14}^{\lambda_{1},\lambda_2}$ & $(1, \lambda_1, \lambda_1, 1+\lambda_1, 1+\lambda_1)$, $(0, S_{\lambda_21}, S_{\lambda_21})$ & $\lambda_2 \geq 0$ \\ 
	$\G_{15}$ & $(0, S_{01}, S_{01})$, $(0, 1, 1, 1, 1) +E_{25} - E_{34}$& \\ 
	$\G_{16}^{\lambda}$ & $(0, S_{01}, S_{01}) + E_{25}$, $(0, 1, 1, 1, 1) + \lambda(E_{25} - E_{34})$ & $\lambda \geq 0$ \\ 
	\hline
\end{longtable}
Here, as usual, we convention that $[X,Y]$ disappears if $[X,Y]=0$; ${\rm Mat}(n,{\mathbb R})$ denotes the algebra of all real $(n \times n)$-matrices ($0 < n \in {\mathbb N}$); we denote respectively by $(a_1, \dots, a_5)$, $E_{ij}$ and $S_{ab}$  the diagonal matrix ${\rm diag}(a_1, \dots, a_5)$, the $(5 \times 5)$-matrix whose only non-zero entry is 1 in row i, column j ($1 \leq i, j \leq 5$) and $\begin{bmatrix}
a&b\\ -b&a	
\end{bmatrix}$. 

%Remark 2.17
\begin{rem}\label{remarklistLiealgebras} For all Lie algebras in Table \ref{tab2}, we have the following~notes.
	\begin{enumerate}[(i)]
		\item If $\G$ is one of $\G^{\lambda}_{1}$,  $ \G_4^{\lambda_1,\lambda_2} $ ($\lambda_1^2+\lambda_2^2\neq0$), $\G_{5}$, $ \G_6^{\lambda} $, $ \G_{7} $, $ \G_8^{\lambda} $, $\G_{11} $, $\G_{12}^{\lambda}$,  $\G_{13}^{\lambda}$ ($\lambda\neq0$), $\G_{14}^{\lambda_1,\lambda_2}$ then the first derived ideal $\G^1: = [\G, \G] \equiv\g_{5,2}$. 
		\item If $\G$ is one of the Lie algebras  $ \G_{2} $, $ \G_{3} $, $ \G_4^{00}$ (i.e. $\G_4^{\lambda_1,\lambda_2}$ with $\lambda_1=\lambda_2=0$), $ \G_{9} $, $ \G_{10}^{\lambda} $,  $\G_{13}^0$ (i.e $\G_{13}^{\lambda}$ with ${\lambda=0}$), $\G_{15}$, $\G_{16}^{\lambda}$ then $\G^1 = [\G, \G]$($\subset \g_{5,2}$) is 4-dimensional .
	\end{enumerate}
\end{rem}

%Đã kiểm tra chính tả
%Section 3 - The main results
\section{THE MAIN RESULTS}\label{sec3}
In this section, we shall continue using the notations of Subsection \ref{ListLieAlgebras} and Table \ref{tab2}. Recall that each real Lie algebra $\G$ defines a unique connected and simply connected Lie group $G$ such that $\text{Lie}(G)=\G$. Therefore, we get a collection of sixteen families of connected and simply connected (real solvable) Lie groups corresponding to the indecomposable Lie algebras given in Table \ref{tab2}. For convenience, each such Lie group is also denoted by the same indices as its Lie algebra. For example, $G^{\lambda}_1, G_2$ are the connected and simply connected real solvable Lie groups corresponding to $\G^{\lambda}_1, \G_2$, respectively.

%Subsection 3.1
\subsection{The exponential mapping of considered  Lie groups}

From now on, by $G$ we always denote one of the following groups $G^{\lambda}_{1}$, $ G_{2} $, $ G_{3} $, $ G_4^{\lambda_1,\lambda_2} $, $G_{5}$, $ G_6^{\lambda} $, $ G_{7} $, $ G_8^{\lambda} $, $ G_{9} $, $ G_{10}^{\lambda} $, $ G_{11} $, $G_{12}^{\lambda}$,  $G_{13}^{\lambda}$, $G_{14}^{\lambda_1,\lambda_2}$, $G_{15}$, $G_{16}^{\lambda}$, i.e. the Lie algebra $\G: = \text{Lie}(G)$ is one of 
%the list of 
the Lie algebras in Table \ref{tab2}.
%Subsection \ref{ListLieAlgebras}.
Let $U=x_1X_1+x_2X_2+x_3X_3+x_4X_4+x_5X_5+xX+yY \in \G$. We always identify $U$ with $(x_1,x_2,x_3,x_4,x_5,x,y) \in  \G $. We get the following~result.

%Proposition 3.1
\begin{prop}\label{Pro5}
	Notations being as above, we have
	\begin{enumerate}
		\item [(1)]	If  $G \in \{G^{\lambda}_{1}, G_{2}, G_{3}, G_4^{\lambda_1,\lambda_2}, G_{5}, G_6^{\lambda}, G_{7}, G_8^{\lambda}, G_{9}, G_{10}^{\lambda}, G_{11}, G_{12}^{\lambda}\}$ then it is exponential. %, i.e. the exponential mapping $\exp: \G \longrightarrow G$ is diffeomorphism.
		
		\item [(2)]	If  $G \in \{ G_{13}^{\lambda}, G_{14}^{\lambda_1,\lambda_2}, G_{15}, G_{16}^{\lambda}\}$ then it is not exponential. %, i.e. the exponential mapping $\exp: \G \longrightarrow G$ is not diffeomorphism.
	\end{enumerate}
\end{prop}
%Proof of Prop 3.1
\begin{proof}
For each $\G$ under consideration, we used MAPLE to find all eigenvalues of $\ad_U$ for all $U \in \G$. Table \ref{tab3} shows the results whereby all $\ad_U$ have no purely imaginary eigenvalues except for four Lie algebras $\G_{13}^{\lambda}$, $\G_{14}^{\lambda_1,\lambda_2}$, $\G_{15}$, $\G_{16}^{\lambda}$, i.e. there exists $U\in \G$ for which $\ad_U$  have purely imaginary eigenvalues. Therefore, the proposition is proved immediately from Proposition~\ref{pro exp}.
\end{proof}

\begin{longtable}{p{.060\textwidth} p{.600\textwidth} p{.245\textwidth}}
	\caption*{Table 3}\label{tab3} \\
	\hline No. & $\ad_U$, $U=\sum_{i=1}^{5}x_iX_i+xX+yY (x_i, x, y \in \R)$  & Eigenvalues of $\ad_U$ \endfirsthead 
	\caption*{Table 3 (continued)} \\ 
	\hline No. & $\ad_U$, $U=\sum_{i=1}^{5}x_iX_i+xX+yY (x_i, x, y \in \R)$ &  Eigenvalues of $\ad_U$ \\ \hline \endhead \hline
	\endlastfoot \hline
	$\G_1^{\lambda}$ & 	$\begin{bmatrix}
	x & 0 & 0 & 0 & 0 & -x_1 & 0 \\
	0 & -x & 0 & 0 & 0 & x_2 & 0 \\
	0 & 0 & y & 0 & 0 & 0 & -x_3 \\
	-x_2 & x_1 & 0 & 0 &0 & -\lambda y & \lambda x\\
	-x_3 & 0 & x_1 & 0 & x+y & -x_5 & -x_5\\
	0 & 0 & 0 & 0 & 0 & 0 & 0\\
	0 & 0 & 0 & 0 & 0 & 0 & 0
	\end{bmatrix}$ & $0$ (multiplicity 3), $y$, $x$, $-x$, $x+y$\\
	&&\\
	$\G_2$ & $\begin{bmatrix}
	x & 0 & 0 & 0 & 0 & -x_1 & 0 \\
	0 & 0 & 0 & 0 & 0 & 0 & 0 \\
	0 & 0 & y & 0 & 0 & 0 & -x_3 \\
	-x_2 & x_1 & 0 & x &0 & -x_4 &0\\
	-x_3 & 0 & x_1 & 0 & x+y & -x_5 & -x_5\\
	0 & 0 & 0 & 0 & 0 & 0 & 0\\
	0 & 0 & 0 & 0 & 0 & 0 & 0
	\end{bmatrix}$ & $y, x+y, x$ (multiplicity 2), $0$ (multiplicity 3)\\
	&&\\
	$\G_3$ & $\begin{bmatrix}
	0 & 0 & 0 & 0 & 0 & 0 & 0 \\
	0 & x & 0 & 0 & 0 & -x_2 & 0 \\
	0 & 0 & y & 0 & 0 & 0 & -x_3 \\
	-x_2 & x_1 & 0 & x & 0 & -x_4 &0\\
	-x_3 & 0 & x_1 & 0 & y & 0 & -x_5\\
	0 & 0 & 0 & 0 & 0 & 0 & 0\\
	0 & 0 & 0 & 0 & 0 & 0 & 0
	\end{bmatrix}$ & $y$ (multiplicity 2), $x$ (multiplicity 2), $0$ (multiplicity 3)\\
	&&\\
	$\G_4^{\lambda_1,\lambda_2}$ & $\begin{bmatrix}
	x & 0 & 0 & 0 & 0 & -x_1 & 0 \\
	0 & y & 0 & 0 & 0 & 0 & -x_2  \\
	0 & 0 &p & 0 & 0 & -\lambda_1x_3 & -\lambda_2x_3 \\
	-x_2 & x_1 & 0 & x+y &0 & -x_4 &-x_4\\
	-x_3 & 0 & x_1 & 0 & q & -(1+\lambda_1)x_5 & -\lambda_2x_5\\
	0 & 0 & 0 & 0 & 0 & 0 & 0\\
	0 & 0 & 0 & 0 & 0 & 0 & 0
	\end{bmatrix}$ & $0$ (multiplicity 2), $y,$ $x,$ $y+x,$ $\lambda_1x+\lambda_2y,$ $\lambda_1x+\lambda_2y+x$\\
	& where $p= \lambda_1x+\lambda_2y  , q= (1+\lambda_1)x+\lambda_2y$&\\
	&&\\
	$\G_5$ & $\begin{bmatrix}
	y & 0 & 0 & 0 & 0 & 0 & -x_1 \\
	y & y & 0 & 0 & 0 & 0 & -x_1-x_2 \\
	0 & 0 & x & 0 & 0 &-x_3 & 0 \\
	-x_2 & x_1 & 0 & 2y &0 & 0 & -2x_4 \\
	-x_3 & 0 & x_1 & 0 & x+y & -x_5 & -x_5\\
	0 & 0 & 0 & 0 & 0 & 0 & 0\\
	0 & 0 & 0 & 0 & 0 & 0 & 0
	\end{bmatrix}$ & $2y, x, x+y, 0$ (multiplicity 2), $y$ (multiplicity 2)\\
	&&\\
	$\G_6^{\lambda}$ & $\begin{bmatrix}
	x & 0 & 0 & 0 & 0 & -x_1 & 0  \\
	y & x& 0 & 0 & 0 & -x_2 & -x_1 \\
	0 & 0 & \lambda x+y & 0 & 0 & -\lambda x_3 & -x_3 \\
	-x_2 & x_1 & 0 & 2x &0 & -2x_4 & 0 \\
	-x_3 & 0 &x_1 & 0 & p & -(1+\lambda)x_5 & -x_5\\
	0 & 0 & 0 & 0 & 0 & 0 & 0\\
	0 & 0 & 0 & 0 & 0 & 0 & 0
	\end{bmatrix}$ &$\lambda x+y, 2x, \lambda x+x+y, 0$ (multiplicity 2), $x$ (multiplicity 2) \\
	&where $p=(1+\lambda)x+y$ &\\
	$\G_7$ & $ \begin{bmatrix}
	y & 0 & 0 & 0 & 0 & 0 & -x_1  \\
	0 & x+y & 0 & 0 & 0 & -x_2 & -x_2 \\
	0 & 0 & x & 0 & 0 & -x_3 & 0\\
	-x_2 & x_1 & 0 & x+2y &0 & -x_4 & -2x_4 \\
	-x_3 & y & x_1 & 0 & x+y & -x_5 & -x_2-x_5\\
	0 & 0 & 0 & 0 & 0 & 0 & 0\\
	0 & 0 & 0 & 0 & 0 & 0 & 0
	\end{bmatrix}$ &  $y$,  $x$,  $x+2y$, $0$ \;(multiplicity 2),\; $x+y$ \;(multiplicity 2)\\
	&&\\
	$\G_8^{\lambda}$ & $\begin{bmatrix}
	x& 0 & 0 & 0 & 0 & -x_1 & 0  \\
	0 & p & 0 & 0 & 0 & -\lambda x_2-x_2 & -x_2\\
	0 & 0 & \lambda x+y & 0 & 0 & -\lambda x_3 & -x_3\\
	-x_2& x_1 & 0 & q&0 & -\lambda x_4-x_4 & -x_4 \\
	-x_3& y & x_1 & 0 & p & -\lambda x_5-x_5 & -x_2-x_5\\
	0 & 0 & 0 & 0 & 0 & 0 & 0\\
	0 & 0 & 0 & 0 & 0 & 0 & 0
	\end{bmatrix}$ &$x, \lambda x+y,  \lambda x+2x+y, 0$ (multiplicity 2), $ \lambda x+x+y$ (multiplicity 2)\\
	& where $p=\lambda x+x+y, q= \lambda x+2x+y$ & \\
	&&\\
	$\G_9$ & $\begin{bmatrix}
	0 & 0 & 0 & 0 & 0 & 0 & 0 \\
	0 & y & 0 & 0 & 0 & 0 & -x_2 \\
	0 & 0 & x& 0 & 0 & -x_3 & 0 \\
	-x_2 & x_1 & 0 & y &0 & 0 &-x_4\\
	-x_3 & 0 & x_1+y & 0 & x & -x_5 & -x_3\\
	0 & 0 & 0 & 0 & 0 & 0 & 0\\
	0 & 0 & 0 & 0 & 0 & 0 & 0
	\end{bmatrix}$ & $y$ (multiplicity 2), $x$ (multiplicity 2), $0$ (multiplicity 3)\\
	&&\\
	$\G_{10}^{\lambda}$ & $\begin{bmatrix}
	0 & 0 & 0 & 0 & 0 & 0 & 0 \\
	0 & x & 0 & 0 & 0 &-x_2 & 0 \\
	0 & 0 & \lambda x+y & 0 & 0 &-\lambda x_3 & -x_3 \\
	-x_2 & x_1 & 0 & x &0 & -x_4 &0\\
	-x_3 & 0 & x_1+y & 0 & \lambda x+y & -\lambda x_5 & -x_3-x_5\\
	0 & 0 & 0 & 0 & 0 & 0 & 0\\
	0 & 0 & 0 & 0 & 0 & 0 & 0
	\end{bmatrix}$ & $\lambda x+y$ (multiplicity 2), $x$ (multiplicity 2), $0$ (multiplicity 3)\\
	&&\\
	$\G_{11}$ & $\begin{bmatrix}
	y & 0 & 0 & 0 & 0 & 0 & -x_1 \\
	0 & x & 0 & 0 & 0 & -x_2 & 0 \\
	0 & y & x & 0 & 0 & -x_3 &-x_2\\
	-x_2& x_1 & 0 & x+y &0 & -x_4 & -x_4 \\
	-x_3 & 0 & x_1 & y & x+y & -x_5 & -x_4-x_5\\
	0 & 0 & 0 & 0 & 0 & 0 & 0\\
	0 & 0 & 0 & 0 & 0 & 0 & 0
	\end{bmatrix}$ & $y, x+y$ (multiplicity 2), $0$ (multiplicity 2), $x$ (multiplicity~2)\\
	&&\\
	$\G_{12}^{\lambda}$ & $\begin{bmatrix}
	x & 0 & 0 & 0 & 0 & -x_1 & 0  \\
	0 & \lambda x+y & 0 & 0 & 0 & -\lambda x_2 & -x_2 \\
	0 & y & \lambda x+y & 0 & 0 & -\lambda x_3&-x_2-x_3\\
	-x_2 & x_1 & 0 & p &0 & -(\lambda +1)x_4 & -x_4 \\
	-x_3 & 0 & x_1 & y & p & -(\lambda+1)x_5 & -x_4-x_5\\
	0 & 0 & 0 & 0 & 0 & 0 & 0\\
	0 & 0 & 0 & 0 & 0 & 0 & 0
	\end{bmatrix}$ &$x,  (\lambda+1)x+y$  (multiplicity 2), $0$ (multiplicity 2), $\lambda x+y$ (multiplicity 2)\\
	&where $p= \lambda x+x+y$& \\
	&&\\
	$\G_{13}^{\lambda}$ & $\begin{bmatrix}
	\lambda y & 0 & 0 & 0 & 0 & 0 &- \lambda x_1  \\
	0 & x & -y & 0 & 0 & -x_2 & x_3 \\
	0 & y & x & 0 & 0 & -x_3 & -x_2\\
	-x_2& x_1 & 0 & x+\lambda y &-y & -x_4 & -\lambda x_4+x_5 \\
	-x_3 & 0 & x_1 & y & x+\lambda y & -x_5 & -x_4-\lambda x_5\\
	0 & 0 & 0 & 0 & 0 & 0 & 0\\
	0 & 0 & 0 & 0 & 0 & 0 & 0
	\end{bmatrix}$ & $0$  (multiplicity 2), $\lambda y$,  $x\pm iy$, $x+\lambda y\pm iy$\\
	&&\\
	$\G_{14}^{\lambda_1,\lambda_2}$ & $  \begin{bmatrix}
	x & 0 & 0 & 0 & 0 &  -x_1 & 0  \\
	0 & p & -y & 0 & 0 & -\lambda_1x_2 & -\lambda_2x_2+x_3 \\
	0 & y & p & 0 & 0 & -\lambda_1x_3 & -x_2-\lambda_2x_3\\
	-x_2 & x & 0 & q & -y & -\lambda_1x_4-x_4 & -\lambda_2x_4+x_5 \\
	-x_3 & 0 & x_1 & y & q & -\lambda_1x_5-x_5 & -x_4-\lambda_2x_5\\
	0 & 0 & 0 & 0 & 0 & 0 & 0\\
	0 & 0 & 0 & 0 & 0 & 0 & 0
	\end{bmatrix}$ & $0$  (multiplicity 2),  $ x, \lambda_1 x+\lambda_2 y\pm iy, \lambda_1 x+\lambda_2 y+x\pm iy$\\
	&where $p=\lambda_1x+\lambda_2y ,q=\lambda_1x+x+\lambda_2y$ & \\
	&&\\
	$\G_{15}$ & $\begin{bmatrix}
	0 & 0 & 0 & 0 & 0 & 0 & 0  \\
	0 & y & -x & 0 & 0 & x_3 &- x_2 \\
	0 & x & y & 0 & 0 & -x_2 & -x_3\\
	-x_2 & x_1 & -y & y & -x & x_5 & -x_4+x_3\\
	-x_3 & y & x_1 & x & y & -x_4& -x_2-x_5\\
	0 & 0 & 0 & 0 & 0 & 0 & 0\\
	0 & 0 & 0 & 0 & 0 & 0 & 0
	\end{bmatrix}$ & $0$ (multiplicity 3), $ y\pm ix$ (multiplicity 2)\\
	&&\\
	$\G_{16}^{\lambda}$ & $\begin{bmatrix}
	0 & 0 & 0 & 0 & 0 & 0 & 0  \\
	0 & y & -x & 0 & 0 & x_3 & -x_2 \\
	0 & x & y & 0 & 0 & -x_2 & -x_3\\
	-x_2 & x_1 & -\lambda y & y & -x & x_5 & -x_4+\lambda x_3\\
	-x_3 & \lambda y+x & x_1 & x & y & -x_2-x_4&-\lambda x_2-x_5\\
	0 & 0 & 0 & 0 & 0 & 0 & 0\\
	0 & 0 & 0 & 0 & 0 & 0 & 0
	\end{bmatrix}$ & $y\pm ix$ (multiplicity 2), $0$ (multiplicity 3)\\
	\hline
\end{longtable}

As an immediate consequence of Proposition \ref{Pro5} and Corollary \ref{cor exp}, we have the following corollary.

%Corollary 3.2
\begin{cor}\label{cor6}
	For each Lie group $G$ from the set of Lie groups in (1) of Proposition \ref{Pro5}, $\Omega_F = \Omega_F(\G)$ for every $F \in \G^*$.
\end{cor}

\subsection{The geometry of maximal dimensional $K$-orbits of some considered Lie groups}\label{sec3.2}

Now we will set up and prove the main results of the paper in this subsection and the next one. As we have listed in Table \ref{tab2}, there are 16 families of considered Lie algebras having nilradical $\g_{5,2}$. As mentioned in the first section, we would like to study all listed Lie algebras by combining Kirillov's method of orbits and Connes' method. However, since the number of Lie algebras listed in Table \ref{tab2} is quite large and the computational volume is also very big, we will first focus on only one subclass of them. In fact, we will consider connected and simply connected (real solvable) Lie groups such that they are exponential and their Lie algebras belong to the list in Table \ref{tab2} with 4-dimensional derived ideals.
Namely, in view of Item (ii) of Remark \ref{remarklistLiealgebras} and Proposition \ref{Pro5}, we will describe the geometric picture of maximal dimensional K-orbits of Lie groups $G \in \{G_{2}, G_{3}, G_4^{00}, G_{9},  G_{10}^{\lambda}\}$ (where $\lambda \in \R$)
%(where $G_4^{00} = G_4^{\lambda_1, \lambda_2}$ with $\lambda_1=\lambda_2=0$ and $G_{10}^{\lambda}$ with $\lambda \in \R$)
and study the topology of foliations formed by the generic $K$-orbits of each group from this~set. 

%Theorem 3.3 (The picture of maximal dimensional K-orbits)
\begin{thm}[\textbf{The picture of maximal dimensional K-orbits}]\label{picture K-orbit}
	Assume that $G \in \{G_{2},\\G_{3}, G_4^{00}, G_{9},  G_{10}^{\lambda}\}$. 
	Denote by $(\alpha_1,\dots,\alpha_5,\alpha,\beta)$ the coordinate vector of $F \in \G^*$ ($\equiv \R^7$) with respect to the basis $(X_1^*, \dots, X_5^*, X^*, Y^*)$ which is dual one of the fixed basis $(X_1, \dots, X_5, \\X, Y)$ of $\G$. Then, the maximal dimension of $K$-orbits of $G$ is exactly six and the picture of 6-dimensional $K$-orbits in each case of $G$ is given by the following assertions.
	
	\begin{enumerate}
		%For G = G_2		
		\item For $G = G_2$: the $K$-orbit $\Omega_F$ is 6-dimensional if and only if the coordinates of $F$ satisfy the condition $\alpha_4 \neq 0 \neq \alpha_3^2+\alpha_5^2$. Namely, we~have
		\begin{enumerate}[(i)]
			\item 	If  $\alpha_3 \alpha_4 \neq 0 = \alpha_5$ and $\alpha_1, \alpha_2, \alpha, \beta \in \R$ then $\Omega_F$ is a part of the hyperplane $\{x^*_5 = 0\}$ as follows
			\[\hspace{1.3cm}\begin{array}{r} \Omega_F=\{(x^*_1, x^*_2, x^*_3, x^*_4, x^*_5, x^*, y^*) \in \G^*:   x^*_5 = 0, \alpha_3 x^*_3 > 0, \, \alpha_4 x^*_4 >0 \}.\end{array}\]
			\item If $\alpha_4 \alpha_5 \neq 0$ and $\alpha_1, \alpha_2, \alpha_3, \alpha, \beta \in \R$ then $\Omega_F$ is a part of a hypersurface of degree two as follows 
			\[\begin{array}{r}
			\Omega_F=\{(x^*_1, x^*_2, x^*_3, x^*_4, x^*_5, x^*, y^*) \in \G^*: x^*_2-\frac{x^*_3 x^*_4}{x^*_5} =\frac{\alpha_2\alpha_5-\alpha_3\alpha_4}{\alpha_5}, \alpha_4 x^*_4 > 0,  \alpha_5 x^*_5 >0 \}.\end{array}\]
		\end{enumerate}
		
		%For G = G_4	
		\item For $G =  G_4^{00}$: the $K$-orbit $\Omega_F$ is 6-dimensional if and only if the coordinates of $F$ satisfy the condition $\alpha_5 \neq 0 \neq \alpha_2^2+\alpha_4^2$. Namely, we~have	
		\begin{enumerate}[(i)]
			\item 	If  $\alpha_2 \alpha_5 \neq 0 = \alpha_4$ and $\alpha_1, \alpha_3, \alpha, \beta \in \R$ then $\Omega_F$ is a part of the hyperplane $\{x^*_4 = 0\}$ as follows
			\[\hspace{1.3cm}\begin{array}{r} \Omega_F=\{(x^*_1, x^*_2, x^*_3, x^*_4, x^*_5, x^*, y^*) \in \G^*:   x^*_4 = 0, \, \alpha_2 x^*_2 > 0, \, \alpha_5 x^*_5 >0 \}.\end{array}\]
			\item If $\alpha_4 \alpha_5 \neq 0$ and $\alpha_1, \alpha_2, \alpha_3, \alpha, \beta \in \R$ then $\Omega_F$ is a part of a hypersurface of degree two as follows
			\[\begin{array}{r}
			\Omega_F=\{(x^*_1, x^*_2, x^*_3, x^*_4, x^*_5, x^*, y^*) \in \G^*: x^*_3-\frac{x^*_2 x^*_5}{x^*_4} = \frac{\alpha_3\alpha_4-\alpha_2\alpha_5}{\alpha_4}, \alpha_4 x^*_4 > 0, \, \alpha_5 x^*_5 >0 \}.\end{array}\]
		\end{enumerate}
		
		%For G = G_3, G_9, G_10	
		\item For $G = G_{3}, \, G = G_{9}, \, G = G_{10}^{\lambda}$: the $K$-orbit $\Omega_F$ is 6-dimensional if and only if the coordinates of $F$ satisfy the condition either $\alpha_4 = 0 \neq \alpha_2\alpha_5$ or $\alpha_4 \neq 0 \neq \alpha_3^2+\alpha_5^2$. Namely, we have
		
		\begin{enumerate}[(i)]
			\item 	If  $\alpha_4 = 0 \neq \alpha_2\alpha_5$ and $\alpha_1, \alpha_3, \alpha, \beta \in \R$ then $\Omega_F$ is a part of the hyperplane $\{x^*_4 = 0\}$ as follows
			\[\hspace{1.3cm}\begin{array}{r} \Omega_F=\{(x^*_1, x^*_2, x^*_3, x^*_4, x^*_5, x^*, y^*) \in \G^*:   x^*_4 = 0, \, \alpha_2 x^*_2 > 0, \, \alpha_5 x^*_5 >0 \}.\end{array}\]
			\item 	If  $\alpha_5 = 0 \neq \alpha_3\alpha_4$ and $\alpha_1, \alpha_2, \alpha, \beta \in \R$ then $\Omega_F$ is a part of the hyperplane $\{x^*_5 = 0\}$ as follows
			\[\hspace{1.3cm}\begin{array}{r} \Omega_F=\{(x^*_1, x^*_2, x^*_3, x^*_4, x^*_5, x^*, y^*) \in \G^*:   x^*_5 = 0, \, \alpha_3 x^*_2 > 0, \, \alpha_4 x^*_4 >0 \}.\end{array}\]
			\item If $\alpha_4 \alpha_5 \neq 0$ and $\alpha_1, \alpha_2, \alpha_3, \alpha, \beta \in \R$ then $\Omega_F$ is a part of a (transcendental or algebraic) hypersurface which is given in each case as follows
			\begin{itemize}
				\item When $G = G_3$, $\Omega_F$ is a part of the following algebraic hypersurface of degree~two
				\[\begin{array}{r}
				\Omega_F=\{(x^*_1, x^*_2, x^*_3, x^*_4, x^*_5, x^*, y^*) \in \G^*: \frac{x^*_2}{x^*_4} - \frac{x^*_3}{x^*_5}=\frac{\alpha_2}{\alpha_4} - \frac{\alpha_3}{\alpha_5}, \\\alpha_4 x^*_4 > 0,  \alpha_5 x^*_5 >0 \}.\end{array}\]
				\item When $G = G_9$, $\Omega_F$ is a part of the following transcendental hypersurface
				\[\begin{array}{r}
				\Omega_F=\{(x^*_1, x^*_2, x^*_3, x^*_4, x^*_5, x^*, y^*) \in \G^*: \frac{x^*_2}{x^*_4} - \frac{x^*_3}{x^*_5} + \ln\vert x^*_4 \vert   = \frac{\alpha_2}{\alpha_4} - \frac{\alpha_3}{\alpha_5} + \ln\vert \alpha_4 \vert,\\ \alpha_4 x^*_4 > 0, \, \alpha_5 x^*_5 >0 \}.\end{array}\]
				\item When $G = G_{10}^{\lambda}$, $\Omega_F$ is a part of the following transcendental hypersurface
				\[\begin{array}{r}
				\Omega_F=\{(x^*_1, x^*_2, x^*_3, x^*_4, x^*_5, x^*, y^*) \in \G^*: \frac{x^*_2}{x^*_4} - \frac{x^*_3}{x^*_5} + \ln \frac{\vert x^*_5 \vert}{{\vert x^*_4 \vert}^{\lambda}}  =  \frac{\alpha_2}{\alpha_4} - \frac{\alpha_3}{\alpha_5} + \ln \frac{\vert \alpha_5 \vert}{{\vert \alpha_4 \vert}^{\lambda}},\\ \alpha_4 x^*_4 > 0, \, \alpha_5 x^*_5 >0 \}.\end{array}\]
			\end{itemize}
		\end{enumerate}
	\end{enumerate}
\end{thm}

%Proof of Therem 3.3
\begin{proof}
First of all, we prove that the maximal dimension of $K$-orbits of each $G$ from the set $\{G_{2}, G_{3}, G_4^{00}, G_{9},  G_{10}^{\lambda}\}$ is exactly six. 

%Now, we show in detail when $G = G_2$. 
Assume that $F  (\alpha_1,\dots,\alpha_5,\alpha,\beta) \in \G^*$. Upon direct computation by using the definition of $B_F$ in Equality (\ref{eq-BF}), we get matrix of $B_{F}$ in Table 4.

It can be verified that 
%the maximal rank of $B_F$ is exactly six. Moreover
%\[\Rank (B_F) = 6 \Leftrightarrow \alpha_4 \neq 0 \neq {\alpha_3}^2 + {\alpha_5}^2.\]
%In the same way, we also prove that 
the maximal dimension of $K$-orbits of all groups from the set $\{G_{2}, G_{3},\\ G_4^{00}, G_{9},  G_{10}^{\lambda}\}$ is exactly six and condition of $\Rank(B_{F})=6$ as Table \ref{tab4}.

Next, we will describe all maximal dimensional $K$-orbits of $G \in \{ G_{2},G_{3}, G_4^{00}, G_{9}, G_{10}^{\lambda}\}$. From Corollary \ref{cor6}, it follows that $\Omega_F=\Omega_F(\G)$. Recall that
\[\Omega_F(\G)=\{F_U/U\in\G\} \subset \R^7,\]
where $F_U =\sum_{i=1}^{5}x^*_iX^*_i+x^*X^*+y^*Y^*=(x^*_1 , x^*_2, x^*_3, x^*_4, x^*_5, x^*,y^*)\in \G^*$ is the linear form on the Lie algebra $\G$ of $ G $ defined by 
$$\left\langle F_U, T\right\rangle =\left\langle F,\exp(\ad_U) T\right\rangle, \forall T \in \G, \forall U \in \G.$$ 
To determine $F_U$ for all 
$U =\sum_{i=1}^{5}x_iX_i+xX+yY=(x_1,x_2,x_3,x_4,x_5,x,y)\in \G$,
we have to determine $\exp(\ad_U)$ in the basic $(X_1, X_2, X_3, X_4, X_5, X, Y)$. We use MAPLE to obtain $\exp(\ad_U)$ of $ G_{2},$ $ G_3,$ $G^{00}_4,$ $G_9,$ $G^{\lambda}_{10}$ in Table \ref{tab5}. 
Now, each case is proved in detail below. 
	
		\begin{enumerate}
		\item For $G=G_2$, we use the result about of $\exp(\ad_U)$ of $G_2$ in Table \ref{tab5} and by a direct computation, we get
		\[\begin{cases}
		x^*_1=\alpha_1 e^{x}-\alpha_4 x_2e^{x} + \alpha_5\frac{(e^{x}-e^{x+y})x_3}{y}   \\
		x^*_2=\alpha_2 + \alpha_4\frac{x_1(e^{x}-1)}{x}   \\
		x^*_3=  \alpha_3 e^{y}+\alpha_5\frac{(e^{x+y}-e^{y})x_1}{x}   \\
		x^*_4= \alpha_4e^x\\
		x^*_5= \alpha_5 e^{x+y}\\
		x^*=  \alpha_1\frac{x_1(1-e^{x})}{x}  +\alpha_4A_2 + \alpha_5 B_2+ \alpha \\
		y^*= \alpha_3\frac{x_3(1-e^{y})}{y} +\alpha_5 C_2+\beta.
		\end{cases}\] 
		where 
		\[ \begin{cases}
		A_2=\frac{e^{x}[x_1x_2(x-1)-x_4x]+x_1x_2+x_4x}{x^2
		}\\
		B_2=\frac{e^{x+y}x(x_1x_3-x_5y)-e^{x}x_1x_3(x+y)-x_1x_3y-x_5xy}{xy(x+y)}\\
		C_2=\frac{e^{y}x_1x_3(x+y)-e^{x+y}y(x_1x_3+x_5x)-x_1x_3x+x_5xy}{xy(x+y)}.
		\end{cases}\]
		
		To describe all maximal dimensional $K$-orbits of $G = G_2$, we only consider 
		$$F(\alpha_1,\dots,\alpha_5,\alpha,\beta) \in \G^*_2$$
		with $\alpha_3 \alpha_4 \neq 0 = \alpha_5$ or $\alpha_4 \alpha_5 \neq 0$, the remaining parameters are arbitrary. 
		
		\begin{itemize}
			\item 	The first case, $\alpha_3 \alpha_4 \neq 0 = \alpha_5$. It is obvious that each of the coordinates $x^*_1, x^*_2, x^*, y^*$ runs over line $\R$, while $x^*_5 \equiv 0$ and $x^*_3, x^*_4 \in \R; \, \alpha_3 x^*_3 > 0, \, \alpha_4 x^*_4 >0$. For this reason, we get	 	
			$\Omega_F$ is a part of hyperplane as follows:
			\[\begin{array}{r} \Omega_F=\{(x^*_1, x^*_2, x^*_3, x^*_4, x^*_5, x^*, y^*) \in \G^*:   x^*_5 = 0, \, \alpha_3 x^*_3 > 0, \, \alpha_4 x^*_4 >0 \}.\end{array}\]
			\item The second case, $\alpha_4 \alpha_5 \neq 0$. Clearly, each of the coordinates $x^*_1, x^*, y^*$ runs over line $\R$, while $x^*_4, x^*_5 \in \R; \, \alpha_4 x^*_4 > 0, \, \alpha_5 x^*_5 >0$. By an easy computation it follows that the coordinates $x^*_2, x^*_3, x^*_4, x^*_5$ are satisfy the following equation
			\[x^*_2-\frac{x^*_3 x^*_4}{x^*_5} =\frac{\alpha_2\alpha_5-\alpha_3\alpha_4}{\alpha_5}.\]
			Thus we get	 	
			$\Omega_F$ is a part of hypersurface of degree two as follows:
			\[\begin{array}{r}
			\Omega_F=\{(x^*_1, x^*_2, x^*_3, x^*_4, x^*_5, x^*, y^*) \in \G^*: x^*_2-\frac{x^*_3 x^*_4}{x^*_5} =\frac{\alpha_2\alpha_5-\alpha_3\alpha_4}{\alpha_5}, \alpha_4 x^*_4 > 0, \alpha_5 x^*_5 >0 \}.\end{array}\] 
		\end{itemize}
		This completes the proof for $G = G_2$. 
		
		\item For $G=G_{3}$: we use the result about of $\exp(\ad_U)$ of $G_3$ in Table \ref{tab5} and by a direct computation, we get
		\[\begin{cases}
		x^*_1= \alpha_1 +\alpha_4\frac{x_2(1-e^{x})}{x}+\alpha_5\frac{x_3 (1-e^{y})}{y}    \\
		x^*_2= \alpha_2 e^{x}  +\alpha_4 x_1e^{x}  \\
		x^*_3= \alpha_3 e^{y} +\alpha_5 x_1e^{y}  \\
		x^*_4= \alpha_4e^x\\
		x^*_5= \alpha_5 e^y\\
		x^*= \alpha_2\frac{x_2(1-e^{x})}{x} +\alpha_4 A_3+\alpha \\
		y^*= \alpha_3 \frac{x_3(1-e^{y})}{y}+\alpha_5 B_3+\beta.
		\end{cases}\] 
		where 
		\[ \begin{cases}
		A_3=-\frac{e^{x}(x_1x_2x-x_1x_2+x_4x)+x_1x_2-x_4x}{x^2}\\
		B_3=-\frac{e^{y}(x_1x_3y -x_1x_3+x_5y)+x_1x_3-x_5y}{y^2}.
		\end{cases}\]
		
		To describe all maximal dimensional $K$-orbits of $G_3$, we only consider $F(\alpha_1,\dots,\alpha_5,\alpha,\beta)\\ \in \G^*_3$ with $\alpha_2 \alpha_5 \neq 0 = \alpha_4$ or $\alpha_3 \alpha_4 \neq 0 = \alpha_5$ or $\alpha_4 \alpha_5 \neq 0$, the remaining ones are arbitrary. 
		
		\begin{itemize}
			\item The first case, $\alpha_4 = 0 \neq \alpha_2\alpha_5$. It is obvious that each of the coordinates $x^*_1, x^*_3, x^*, y^*$ runs over line $\R$, while $x^*_4 \equiv 0$ and $x^*_2, x^*_5 \in \R; \, \alpha_2 x^*_2 > 0, \, \alpha_5 x^*_5 >0$. Therefore we get	$\Omega_F$ is a part of hyperplane as follows:
			\[\begin{array}{r} \Omega_F=\{(x^*_1, x^*_2, x^*_3, x^*_4, x^*_5, x^*, y^*) \in \G^*:   x^*_4 = 0, \, \alpha_2 x^*_2 > 0, \, \alpha_5 x^*_5 >0 \}.\end{array}\]
			\item 	The second case, $\alpha_3 \alpha_4 \neq 0 = \alpha_5$. It is easily seen that each of the coordinates $x^*_1, x^*_2, x^*, y^*$ runs over line $\R$, while $x^*_5 \equiv 0$ and $x^*_3, x^*_4 \in \R; \, \alpha_3 x^*_3 > 0, \, \alpha_4 x^*_4 >0$. Therefore, we get	
			$\Omega_F$ is a part of hyperplane as~follows:
			\[\begin{array}{r} \Omega_F=\{(x^*_1, x^*_2, x^*_3, x^*_4, x^*_5, x^*, y^*) \in \G^*:   x^*_5 = 0,\, \alpha_3 x^*_3 > 0, \, \alpha_4 x^*_4 >0 \}.\end{array}\]
			\item The third case, $\alpha_4 \alpha_5 \neq 0$. Clearly, each of the coordinates $x^*_1$, $x^*$, $y^*$ runs over line $\R$, while $x^*_4, x^*_5 \in \R; \, \alpha_4 x^*_4 > 0, \, \alpha_5 x^*_5 >0$. By an easy computation it follows that the coordinates $x^*_2, x^*_3, x^*_4, x^*_5$ are satisfy the following equation
			\[\frac{x^*_2}{x^*_4}-\frac{x^*_3 }{x^*_5} =\frac{\alpha_2}{\alpha_4}-\frac{\alpha_3}{\alpha_5}.\]
			Hence we get	 	
			$\Omega_F$ is a part of algebraic hypersurface of degree two as~follows:
			\[\begin{array}{r}
			\Omega_F=\{(x^*_1, x^*_2, x^*_3, x^*_4, x^*_5, x^*, y^*) \in \G^*: \frac{x^*_2}{x^*_4}-\frac{x^*_3 }{x^*_5} =\frac{\alpha_2}{\alpha_4}-\frac{\alpha_3}{\alpha_5}, \alpha_4 x^*_4 > 0, \, \alpha_5 x^*_5 >0 \}.\end{array}\] 
		\end{itemize}
		This completes the proof for case $G = G_3$.
		
		\item For $G=G^{00}_4$: we use the result about of $\exp(\ad_U)$ of $G^{00}_4$ in Table \ref{tab5} and by a direct computation, we get
		\[\begin{cases}
		x^*_1=\alpha_1 e^{x}+\alpha_4 \frac{x_2(e^{x}-e^{x+y})}{y}- \alpha_5 x_3e^x \\
		x^*_2= \alpha_2 e^{y} +\alpha_4 \frac{x_1(e^{x+y}-e^{y})}{x} \\
		x^*_3= \alpha_3  + \alpha_5 \frac{x_1(e^x-1)}{x} \\
		x^*_4= \alpha_4 e^{x+y} \\
		x^*_5= \alpha_5 e^{x}\\
		x^*=\alpha_1\frac{x_1(1-e^{x})}{x}+\alpha_4 A_4+\alpha_5 B_4+\alpha  \\
		y^* = \alpha_2 \frac{x_2(1-e^{y})}{y}+\alpha_4 C_4+\beta
		\end{cases}\] 
		where 
		\[ \begin{cases}
		A_4=\frac{e^{x+y}x(x_1x_2-x_4y)-e^{x}x_1x_2(x+y)+y(x_1x_2+x_4x)}{xy(x+y)}\\
		B_4=\frac{e^xx_1x_3x-e^xx_1x_3-e^xx_5x+x_1x_3+x_5x}{x^2} \\
		C_4=\frac{e^{y}x_1x_2(x+y)-e^{x+y}y(x_1x_2+x_4x)+x(yx_4-x_1x_2)}{xy(x+y)}.	 
		\end{cases}\]
		
		To describe all maximal dimensional $K$-orbits of $G^{00}_4$, we only consider $F(\alpha_1,\dots,\alpha_5,\alpha,\beta)\\ \in {\G^{00}_4}^*$ with $\alpha_2\alpha_5 \neq 0 =\alpha_4$ or $\alpha_4\alpha_5\neq0$, the remaining ones are arbitrary. 
		
		\begin{itemize}
			\item 	The first case, $ \alpha_2\alpha_5 \neq0=\alpha_4$. It is obvious that each of the coordinates $x^*_1, x^*_3, x^*, y^*$ runs over line $\R$, while $x^*_4 \equiv 0$ and $x^*_2, x^*_5 \in \R; \, \alpha_2 x^*_2 > 0, \, \alpha_5 x^*_5 >0$. Hence we get	 	
			$\Omega_F$ is a part of hyperplane as follows:
			\[\begin{array}{r} \Omega_F=\{(x^*_1, x^*_2, x^*_3, x^*_4, x^*_5, x^*, y^*) \in \G^*:   x^*_4 = 0,  \alpha_2 x^*_2 > 0, \, \alpha_5 x^*_5 >0 \}.\end{array}\]
			\item The second case, $\alpha_4 \alpha_5 \neq 0$. Clearly, each of the coordinates $x^*_1, x^*, y^*$ runs over line $\R$, while $x^*_4, x^*_5 \in \R; \, \alpha_4 x^*_4 > 0, \, \alpha_5 x^*_5 >0$. By an easy computation it follows that the coordinates $x^*_2, x^*_3, x^*_4$ and $x^*_5$ are satisfy the following equation
			\[x^*_3-\frac{x^*_2 x^*_5}{x^*_4} = \frac{\alpha_3\alpha_4-\alpha_2\alpha_5}{\alpha_4}.\]
			Therefore we get	 	
			$\Omega_F$ is a part of hypersurface of degree two as follows:
			\[\begin{array}{r}
			\Omega_F=\{(x^*_1, x^*_2, x^*_3, x^*_4, x^*_5, x^*, y^*) \in \G^*:x^*_3-\frac{x^*_2 x^*_5}{x^*_4} = \frac{\alpha_3\alpha_4-\alpha_2\alpha_5}{\alpha_4}, \alpha_4 x^*_4 > 0, \, \alpha_5 x^*_5 >0 \}.\end{array}\] 
		\end{itemize}
		This completes the proof for case $G = G^{00}_4$.
		
		\item For $G=G_{9}$: we use the result about of $\exp(\ad_U)$ of $G_9$ in Table \ref{tab5} and by a direct computation, we get
		\[\begin{cases}
		x^*_1= \alpha_1 + \alpha_4 \frac{x_2(1-e^{y})}{y}+  \alpha_5 \frac{x_3(1-e^{x})}{x} \\
		x^*_2= \alpha_2  e^{ y} + \alpha_4 x_1 e^{y} \\
		x^*_3= \alpha_3 e^{x} + \alpha_5(x_1+y)e^{x} \\
		x^*_4= \alpha_4e^{y}  \\
		x^*_5= \alpha_5 e^{x} \\
		x^*=  \alpha_3 \frac{ x_3(1-e^{x})}{x}+\alpha_5A_9+\alpha\\
		y^* =  \alpha_2\frac{x_2(1-e^{ y})}{y} +\alpha_4 B_9+\alpha_5\frac{x_3(1-e^{x})}{x}
		+\beta.
		\end{cases}\] 
		where 
		\[ \hspace{1cm}\begin{cases}
		A_9=-\dfrac{e^{x}(x_1x_3x+x_3xy-x_1x_3-x_3y+x_5x)+x_1x_3+x_3y-x_5x}{x^2}\\
		B_9=-\dfrac{e^{y}(x_1x_2y-x_1x_2+x_4y)+x_1x_2-x_4y}{y^2}.
		\end{cases}\]
		
		To describe all maximal dimensional $K$-orbits of $G_9$, we only consider $F(\alpha_1,\dots,\alpha_5,\alpha,\beta)\\ \in \G^*_9$ with $\alpha_4 = 0 \neq \alpha_2\alpha_5$ or $\alpha_3 \alpha_4 \neq 0 = \alpha_5$ or $\alpha_4 \alpha_5 \neq 0$, the remaining ones are arbitrary. 
		
		\begin{itemize}
			\item The first case, $\alpha_4 = 0 \neq \alpha_2\alpha_5$. It is obvious that each of the coordinates $x^*_1, x^*_3, x^*, y^*$ runs over line $\R$, while $x^*_4 \equiv 0$ and $x^*_2, x^*_5 \in \R; \, \alpha_2 x^*_2 > 0, \, \alpha_5 x^*_5 >0$. Therefore we get	 	
			$\Omega_F$ is a part of hyperplane as follows:
			\[\begin{array}{r} \Omega_F=\{(x^*_1, x^*_2, x^*_3, x^*_4, x^*_5, x^*, y^*) \in \G^*:   x^*_4 = 0, \alpha_2 x^*_2 > 0, \, \alpha_5 x^*_5 >0 \}.\end{array}\]
			\item 	The second case, $\alpha_3 \alpha_4 \neq 0 = \alpha_5$. Obviously, each of the coordinates $x^*_1, x^*_2, x^*, y^*$ runs over line $\R$, while $x^*_5 \equiv 0$ and $x^*_3, x^*_4 \in \R; \, \alpha_3 x^*_3 > 0, \, \alpha_4 x^*_4 >0$. For this reason, we get	 	
			$\Omega_F$ is a part of hyperplane as follows:
			\[\begin{array}{r} \Omega_F=\{(x^*_1, x^*_2, x^*_3, x^*_4, x^*_5, x^*, y^*) \in \G^*:   x^*_5 = 0, \alpha_3 x^*_3 > 0, \, \alpha_4 x^*_4 >0 \}.\end{array}\]
			\item The third case, $\alpha_4 \alpha_5 \neq 0$. Clearly, each of the coordinates $x^*_1$, $x^*$, $y^*$ runs over line $\R$, while $x^*_4, x^*_5 \in \R; \, \alpha_4 x^*_4 > 0, \, \alpha_5 x^*_5 >0$. By an easy computation it follows that the coordinates $x^*_2, x^*_3, x^*_4, x^*_5$ are satisfy the following equation
			\[\frac{x^*_2}{x^*_4} - \frac{x^*_3}{x^*_5} + \ln\vert x^*_4 \vert =  \frac{\alpha_2}{\alpha_4} - \frac{\alpha_3}{\alpha_5} + \ln\vert \alpha_4 \vert.\]
			From this we get	 	
			$\Omega_F$ is a part of transcendental hypersurface as follows:
			\[\begin{array}{r}
			\Omega_F=\{(x^*_1, x^*_2, x^*_3, x^*_4, x^*_5, x^*, y^*) \in \G^*: \frac{x^*_2}{x^*_4} - \frac{x^*_3}{x^*_5} + \ln\vert x^*_4 \vert =  \frac{\alpha_2}{\alpha_4} - \frac{\alpha_3}{\alpha_5} + \ln\vert \alpha_4 \vert,\\ \alpha_4 x^*_4 > 0, \, \alpha_5 x^*_5 >0 \}.\end{array}\] 
		\end{itemize}
		This completes the proof for case $G = G_9$.
		\item For $G=G^{\lambda}_{10}$: We use the result about of $\exp(\ad_U)$ of $G^{\lambda}_{10}$ in Table \ref{tab5}  and by a direct computation, we get
		\[\begin{cases}
		x^*_1= \alpha_1 + \alpha_4 \frac{x_2(1-e^{x})}{x}+ \alpha_5 D_{10} \\
		x^*_2= \alpha_2  e^{ x}+ \alpha_4 x_1 e^{x} \\
		x^*_3=\alpha_3e^{\lambda x+y}+\alpha_5(x_1+y)e^{\lambda x+y} \\
		x^*_4= \alpha_4 e^{x}  \\
		x^*_5= \alpha_5 e^{\lambda x+y} \\
		x^*=  \alpha_2\frac{ x_2(1-e^{x})}{x}+\alpha_3 \lambda D_{10}+\alpha_4 A_{10}+\alpha_5 B_{10}+\alpha\\
		y^* =  \alpha_3D_{10} +\alpha_5 C_{10}+\beta.
		\end{cases}\] 
		where 
		\[	\hspace{1cm}	 \begin{cases}
		\begin{array}{l}
		A_{10}=-\frac{1}{x^2}[e^{x}(x_1x_2x-x_1x_2+x_4x)+x_1x_2-x_4x]\end{array}\\
		\begin{array}{rr}
		B_{10}=-\frac{\lambda}{(\lambda x+y)^2}[e^{\lambda x+y}( x_3y^2+ x_1x_3y+\lambda x_3xy+\lambda x_1x_3x+ x_5y&\\- x_3y- x_1x_3+\lambda x_5x)-\lambda x_5y+\lambda x_3y+\lambda x_1x_3-\lambda x_5x]&\end{array}\\
		\begin{array}{rr}
		C_{10}=-\frac{1}{(\lambda x+y)^2} \{e^{\lambda x+y}[x_3y(x_1+y)+\lambda x (x_3y+x_1x_2+x_3+x_5)&\\-x_1x_3+  + x_5y+]+x_1x_3-\lambda x_3x-x_5y-\lambda x_5x\}&
		\end{array}\\
		\begin{array}{l}
		D_{10}=\frac{x_3(1-e^{\lambda x+y})}{\lambda x+y}
		\end{array}.\\
		\end{cases}\]
		
		To describe all maximal dimensional $K$-orbits of $G^{\lambda}_{10}$, we only consider $F(\alpha_1,\dots,\alpha_5,\alpha,\beta)\\ \in \G^{\lambda*}_{10}$ with $\alpha_4 = 0 \neq \alpha_2\alpha_5$ or $\alpha_3 \alpha_4 \neq 0 = \alpha_5$ or $\alpha_4 \alpha_5 \neq 0$, the remaining ones are arbitrary. 
		
		\begin{itemize}
			\item The first case, $\alpha_4 = 0 \neq \alpha_2\alpha_5$. It is easily seen that each of the coordinates $x^*_1, x^*_3, x^*, y^*$ runs over line $\R$, while $x^*_4 \equiv 0$ and $x^*_2, x^*_5 \in \R; \, \alpha_2 x^*_2 > 0, \, \alpha_5 x^*_5 >0$. Therefore we get	 	
			$\Omega_F$ is a part of hyperplane as follows:
			\[\begin{array}{r} \Omega_F=\{(x^*_1, x^*_2, x^*_3, x^*_4, x^*_5, x^*, y^*) \in \G^*:   x^*_4 = 0, \alpha_2 x^*_2 > 0, \, \alpha_5 x^*_5 >0 \}.\end{array}\]
			\item 	The second case, $\alpha_3 \alpha_4 \neq 0 = \alpha_5$. It is obvious that each of the coordinates $x^*_1, x^*_2, x^*, y^*$ runs over line $\R$, while $x^*_5 \equiv 0$ and $x^*_3, x^*_4 \in \R; \, \alpha_3 x^*_3 > 0, \, \alpha_4 x^*_4 >0$. For this reason, we get	 	
			$\Omega_F$ is a part of hyperplane as follows:
			\[\begin{array}{r} \Omega_F=\{(x^*_1, x^*_2, x^*_3, x^*_4, x^*_5, x^*, y^*) \in \G^*:   x^*_5 = 0,  \alpha_3 x^*_3 > 0, \, \alpha_4 x^*_4 >0 \}.\end{array}\]
			\item The third case, $\alpha_4 \alpha_5 \neq 0$. Clearly, each of the coordinates $x^*_1$, $x^*$, $y^*$ runs over line $\R$, while $x^*_4, x^*_5 \in \R; \, \alpha_4 x^*_4 > 0, \, \alpha_5 x^*_5 >0$. By an easy computation it follows that the coordinates $x^*_2, x^*_3, x^*_4, x^*_5$ are satisfy the following equation
			\[\frac{x^*_2}{x^*_4} - \frac{x^*_3}{x^*_5} + \ln \frac{\vert x^*_5 \vert}{{\vert x^*_4 \vert}^{\lambda}}  =  \frac{\alpha_2}{\alpha_4} - \frac{\alpha_3}{\alpha_5} + \ln \frac{\vert \alpha_5 \vert}{{\vert \alpha_4 \vert}^{\lambda}}.\]
			Thus we get	 	
			$\Omega_F$ is a part of transcendental hypersurface as follows:
			\[\begin{array}{r}
			\Omega_F=\{(x^*_1, x^*_2, x^*_3, x^*_4, x^*_5, x^*, y^*) \in \G^*: \frac{x^*_2}{x^*_4} - \frac{x^*_3}{x^*_5} + \ln \frac{\vert x^*_5 \vert}{{\vert x^*_4 \vert}^{\lambda}}  = \frac{\alpha_2}{\alpha_4} - \frac{\alpha_3}{\alpha_5} + \ln \frac{\vert \alpha_5 \vert}{{\vert \alpha_4 \vert}^{\lambda}},\\ \alpha_4 x^*_4 > 0, \, \alpha_5 x^*_5 >0 \}.\end{array}\] 
		\end{itemize}
This completes the proof for $G = G^{\lambda}_{10}$.	
	\end{enumerate}
\end{proof}	
		
	{\small \begin{longtable}{p{.050\textwidth} p{.500\textwidth} p{.355\textwidth}}
			\caption*{Table 4}\label{tab4} \\
			\hline No. & $B_{F}$, $F(\alpha_1,\dots,\alpha_5,\alpha,\beta) \in \G^*$ & Condition of $\Rank(B_{F})=6$ \endfirsthead 
			\caption*{Table 4 (continued)} \\ 
			\hline No. & $B_{F}$, $F(\alpha_1,\dots,\alpha_5,\alpha,\beta) \in \G^*$ &  Condition of $\Rank(B_{F})=6$ \\ \hline \endhead \hline
			\endlastfoot \hline
			$\G_2$ & $\begin{bmatrix}
			0 & \alpha_4 & \alpha_5 & 0 & 0 & - \alpha_1 & 0 \\
			- \alpha_4 & 0 & 0 & 0 & 0 &0 & 0 \\
			- \alpha_5 & 0 & 0 & 0 & 0 & 0 & - \alpha_3 \\
			0 & 0 & 0 & 0 & 0 & - \alpha_4 & 0\\
			0 & 0 & 0 & 0 & 0 & - \alpha_5 & - \alpha_5\\
			\alpha_1 & 0 & 0 & \alpha_4 & \alpha_5 & 0 & 0\\
			0 & 0 & \alpha_3 & 0 & \alpha_5 & 0 & 0
			\end{bmatrix}$ &  $ \alpha_4 \neq 0 \neq {\alpha_3}^2 + {\alpha_5}^2 $\\
			$\G_3$ & $\begin{bmatrix}
			0 & a_4 & a_5 & 0 & 0 & 0 & 0 \\
			-a_4 & 0 & 0 & 0 & 0 & -a_2 & 0 \\
			-a_5 & 0 & 0 & 0 & 0 & 0 & -a_3 \\
			0 & 0 & 0 & 0 &0 & -a_4 & 0\\
			0 & 0 & 0 & 0 & 0 & 0 & -a_5\\
			0 & a_2 & 0 & a_4 & 0 & 0 & 0\\
			0 & 0 & a_3 & 0 & a_5 & 0 & 0
			\end{bmatrix}  $ & $\alpha_4 = 0 \neq \alpha_2\alpha_5$ or $\alpha_4 \neq 0  \neq\alpha_3^2+\alpha_5^2$\\
			&&\\
			$\G_4^{00}$ & $\begin{bmatrix}
			0 & a_4 & a_5 & 0 & 0 & -a_1 & 0 \\
			-a_4 & 0 & 0 & 0 & 0 & 0 & -a_2 \\
			-a_5 & 0 & 0 & 0 & 0 & 0 & 0 \\
			0 & 0 & 0 & 0 &0 & -a_4 & -a_4\\
			0 & 0 & 0 & 0 & 0 & -a_5 & 0\\
			a_1 & 0 & 0 & a_4 & a_5 & 0 & 0\\
			0 & a_2 & 0 & a_4 & 0 & 0 & 0
			\end{bmatrix} $ & $\alpha_5 \neq 0 \neq \alpha_2^2+\alpha_4^2$\\
			&&\\
			$\G_9$ & $ \begin{bmatrix}
			0 & a_4 & a_5 & 0 & 0 & 0 & 0 \\
			-a_4 & 0 & 0 & 0 & 0 & 0 & -a_2\\
			-a_5 & 0 & 0 & 0 & 0 & -a_3 & -a_5 \\
			0 & 0 & 0 & 0 &0 & 0 & -a_4\\
			0 & 0 & 0 & 0 & 0 & -a_5 & 0\\
			0 & 0 & a_3 & 0 & a_5 & 0 & 0\\
			0 & a_2 & a_5 & a_4 & 0 & 0 & 0
			\end{bmatrix}$ & $\alpha_4 = 0 \neq \alpha_2\alpha_5$ or $\alpha_4 \neq 0   \neq\alpha_3^2+\alpha_5^2$ \\
			&&\\
			$\G_{10}^{\lambda}$ & $\begin{bmatrix}
			0 & a_4 & a_5 & 0 & 0 & 0 & 0 \\
			-a_4 & 0 & 0 & 0 & 0 & -a_2 & 0\\
			-a_5 & 0 & 0 & 0 & 0 & -\lambda a_3 & -a_3-a_5 \\
			0 & 0 & 0 & 0 &0 & -a_4 & 0\\
			0 & 0 & 0 & 0 & 0 & -\lambda a_5 & -a_5\\
			0 & a_2 & \lambda a_3 & a_4 & \lambda a_5 & 0 & 0\\
			0 & 0 & a_3+a_5 & 0 & a_5 & 0 & 0
			\end{bmatrix} $ & $\alpha_4 = 0 \neq \alpha_2\alpha_5$ or $\alpha_4 \neq 0 \neq  \alpha_3^2+\alpha_5^2$ \\
			&&\\
			\hline
	\end{longtable}}
	
	%Table 2
	{\small \begin{longtable}{p{.050\textwidth}  p{.860\textwidth}}
			\caption*{Table 5}\label{tab5} \\
			\hline No. & $\exp(\ad_U)$, $U=\sum_{i=1}^{5}x_iX_i+xX+yY (x_i, x, y \in \R)$  \endfirsthead 
			\caption*{Table 5 (continued)} \\ 
			\hline No. & $\exp(\ad_U)$, $U=\sum_{i=1}^{5}x_iX_i+xX+yY (x_i, x, y \in \R)$   \\ \hline \endhead \hline
			\endlastfoot \hline
			$\G_2$ & $ \begin{bmatrix}
			e^{x} & 0 & 0 & 0 & 0 & \frac{x_1(1-e^{x})}{x} & 0 \\
			0 & 1 & 0 & 0 & 0 &0 & 0 \\
			0 & 0 & e^{y} & 0 & 0 & 0 & \frac{x_3(1-e^{y})}{y} \\
			-x_2e^{x} & \frac{x_1(e^{x}-1)}{x} & 0 & e^{x} &0 & A_2 & 0\\
			\frac{x_3e^{x}(1-e^{y})}{y} & 0 & \frac{x_1(e^{x}-1)e^{y}}{x} & 0 & e^{x+y} & B_2 & C_2\\
			0 & 0 & 0 & 0 & 0 & 1 & 0\\
			0 & 0 & 0 & 0 & 0 & 0 & 1
			\end{bmatrix} $ \\
			&\hspace{5.0cm} where $  \begin{cases}
			A_2=\frac{e^{x}[x_1x_2(x-1)-x_4x]+x_1x_2+x_4x}{x^2
			}\\
			B_2=\frac{e^{x+y}x(x_1x_3-x_5y)-e^{x}x_1x_3(x+y)-x_1x_3y-x_5xy}{xy(x+y)}\\
			C_2=\frac{e^{y}x_1x_3(x+y)-e^{x+y}y(x_1x_3+x_5x)-x_1x_3x+x_5xy}{xy(x+y)}
			\end{cases}$ \\
			$\G_3$ & $\begin{bmatrix}
			1 & 0 & 0 & 0 & 0 & 0 & 0 \\
			0 & e^{x} & 0 & 0 & 0 & \frac{x_2(1-e^{x})}{x}  & 0 \\
			0 & 0 & e^{y} & 0 & 0 & 0 & \frac{x_3(1-e^{y})}{y} \\
			\frac{x_2(1-e^{x})}{x} & x_1 e^{x} & 0 & e^{x} &0 & A_3 & 0\\
			\frac{x_3 (1-e^{y})}{y} & 0 & x_1e^{y} & 0 & e^{y} & 0 & B_3\\
			0 & 0 & 0 & 0 & 0 & 1 & 0\\
			0 & 0 & 0 & 0 & 0 & 0 & 1
			\end{bmatrix}$  \\
			&  \hspace{6.5cm} where  $\begin{cases}
			A_3=-\frac{e^{x}(x_1x_2x-x_1x_2+x_4x)+x_1x_2-x_4x}{x^2}\\
			B_3=-\frac{e^{y}(x_1x_3y -x_1x_3+x_5y)+x_1x_3-x_5y}{x^2}
			\end{cases}$ \\
			& \\
			$\G_4^{00}$ & $\begin{bmatrix}
			e^{x} & 0 & 0 & 0 & 0 & \frac{x_1(1-e^{x})}{x} & 0 \\
			0 & e^{y} & 0 & 0 & 0 & 0  & \frac{x_2(1-e^{y})}{y} \\
			0 & 0 & 1 & 0 & 0 & 0& 0 \\
			\frac{x_2(e^{x}-e^{x+y})}{y} & \frac{x_1(e^{x+y}-e^{y})}{x} & 0 & e^{x+y} &0 & A_4 & C_4\\
			-x_3e^x & 0 & \frac{x_1(e^x-1)}{x}& 0 & e^x & B_4 & 0\\
			0 & 0 & 0 & 0 & 0 & 1 & 0\\
			0 & 0 & 0 & 0 & 0 & 0 & 1
			\end{bmatrix}$  \\
			& \hspace{4.7cm} where $\begin{cases}
			A_4=\frac{e^{x+y}x(x_1x_2-x_4y)-e^{x}x_1x_2(x+y)+y(x_1x_2+x_4x)}{xy(x+y)}\\
			B_4=\frac{e^xx_1x_3x-e^xx_1x_3-e^xx_5x+x_1x_3+x_5x}{x^2} \\
			C_4=\frac{e^{y}x_1x_2(x+y)-e^{x+y}y(x_1x_2+x_4x)+x(yx_4-x_1x_2)}{xy(x+y)}
			\end{cases}$ \\
			&\\
			$\G_9$ & $\begin{bmatrix}
			1 & 0 & 0 & 0 & 0 & 0 & 0 \\
			0 & e^{y} & 0 & 0 & 0 &0 & \frac{x_2(1-e^{y})}{y} \\
			0 & 0 & e^{x} & 0 & 0 &  \frac{x_3(1-e^{x})}{x}&0 \\
			\frac{x_2(1-e^{y})}{y} & x_1 e^{y} &0& e^{y} & 0 & 0 & B_9 \\
			\frac{x_3(1-e^{x})}{x} & 0 & (x_1+y)e^{x} & 0 & e^{x} & A_9 & \frac{x_3(1-e^{x})}{x} \\
			0 & 0 & 0 & 0 & 0 & 1 & 0\\
			0 & 0 & 0 & 0 & 0 & 0 & 1
			\end{bmatrix}$\\
			& \hspace{4.0cm}  where $  \begin{cases}
			A_9=-\frac{e^{x}(x_1x_3x+x_3xy-x_1x_3-x_3y+x_5x)+x_1x_3+x_3y-x_5x}{x^2}\\
			B_9=-\frac{e^{y}(x_1x_2y-x_1x_2+x_4y)+x_1x_2-x_4y}{y^2}
			\end{cases}  $ \\
			&\\
			$\G_{10}^{\lambda}$ & $\begin{bmatrix}
			1 & 0 & 0 & 0 & 0 & 0 & 0 \\
			0 & e^{x} & 0 & 0 & 0 &\frac{x_2(1-e^{x})}{x} & 0 \\
			0 & 0 & e^{\lambda x+y} & 0 & 0 &  \lambda D_{10}&D_{10} \\
			\frac{x_2(1-e^{x})}{x} & x_1 e^{x} &0& e^{x} & 0  & A_{10} & 0 \\
			D_{10} & 0 & (x_1+y)e^{\lambda x+y} & 0 & e^{\lambda x+y} & B_{10} & C_{10}\\
			0 & 0 & 0 & 0 & 0 & 1 & 0\\
			0 & 0 & 0 & 0 & 0 & 0 & 1
			\end{bmatrix}$  \\
			& \hspace{1.3cm}where $\begin{cases}
			\begin{array}{l}
			A_{10}=-\frac{1}{x^2}[e^{x}(x_1x_2x-x_1x_2+x_4x)+x_1x_2-x_4x]\end{array}\\
			\begin{array}{rr}
			B_{10}=-\frac{\lambda}{(\lambda x+y)^2}[e^{\lambda x+y}( x_3y^2+ x_1x_3y+\lambda x_3xy+\lambda x_1x_3x+ x_5y&\\- x_3y- x_1x_3+\lambda x_5x)-\lambda x_5y+\lambda x_3y+\lambda x_1x_3-\lambda x_5x]&\end{array}\\
			\begin{array}{rr}
			C_{10}=-\frac{1}{(\lambda x+y)^2} \{e^{\lambda x+y}[x_3y(x_1+y)+\lambda x (x_3y+x_1x_2+x_3+x_5)&\\-x_1x_3 + x_5y]+x_1x_3-\lambda x_3x-x_5y-\lambda x_5x\}&
			\end{array}\\
			\begin{array}{l}
			D_{10}=\frac{x_3(1-e^{\lambda x+y})}{\lambda x+y}
			\end{array}\\
			\end{cases}$    \\
			\hline
	\end{longtable}}

%Remark 3.4
\begin{rem}[{\bf Geometrical characteristics of maximal $K$-orbits}]\label{RemarkOrbits} By looking at the picture of $K$-orbits of maximal dimension of considered Lie groups, we have some geometrical characteristics as follows.
	\begin{enumerate}[(i)]
		\item Because $G$ is exponential for any group $G \in \{G_{2}, G_{3}, G_4^{00}, G_{9},  G_{10}^{\lambda}\}$, all $K$-orbits of $G$ are connected, simply connected submanifolds of $\G^*$. Moreover, they are all homeomorphic to Euclidean spaces (see \cite[$\S$15]{K76}). In particular, all 6-dimensional $K$-orbits of $G$ are homeomorphic to $\R^6$. It can be easily verified by using the picture of $K$-orbits in Theorem \ref{picture K-orbit} above.
		\item For each group $G \in \{G_{2}, G_{3}, G_4^{00}, G_{9},  G_{10}^{\lambda}\}$, there are exactly two types of maximal dimensional $K$-orbits. Each of the $K$-orbits of the first type is a part of a certain hyperplane. The second type contains the 6-dimensional $K$-orbits which are parts of certain (transcendental or algebraic) hypersurfaces in $\G^*$. 
		\item Each group $G \in \{G_{2}, G_{3}, G_4^{00}, G_{9},  G_{10}^{\lambda}\}$ has exactly four 6-dimensional $K$-orbits of the first type. These $K$-orbits are four parts of only one certain hyperplane in $\G^*$ and they are formed when ``cutting" this hyperplane by two other ones. 
		For example, in the case of $G = G_2$, four $K$-orbits of the first type are exactly 
		\[\{x^*_5 = 0, x^*_3 > 0, x^*_4 > 0\}, \, \, \{x^*_5 = 0, x^*_3 < 0, x^*_4 > 0\},\]
		\[\{x^*_5 = 0, x^*_3 < 0, x^*_4 < 0\}, \, \, \{x^*_5 = 0, x^*_3 > 0, x^*_4 < 0\}.\]
		There are four connected components of $\{x^*_5 = 0\} \setminus \bigl(\{x^*_3 = 0\} \cup \{x^*_4 = 0\}\bigr)$.

		\noindent A completely similar situation occurs for all remaining groups. 
		\item For each group $G \in \{G_{2}, G_{3}, G_4^{00}, G_{9},  G_{10}^{\lambda}\}$, there is an infinite family of $K$-orbits in the second type. One such $K$-orbit is always a part of a (transcendental or algebraic) hypersurface in $\G^*$. Each such hypersurface also generates four $K$-orbits of the second type by ``cutting" this hypersurface by two certain hyperplanes. 
		For example, in the case of $G = G_2$, ``cutting" a (algebraic) hypersurface of the~form $\{x^*_2-\frac{x^*_3 x^*_4}{x^*_5} = c\}$
		(for some constant $c \in \R$) by two hyperplanes $\{x^*_4 = 0\}$ and $\{x^*_5 = 0\}$ we obtain four $K$-orbits of the second type as follows: 
		\[\hspace{1cm}\{x^*_2-\frac{x^*_3 x^*_4}{x^*_5} = c; x^*_4 > 0, x^*_5 > 0\}, \{x^*_2-\frac{x^*_3 x^*_4}{x^*_5} = c; x^*_4 < 0, x^*_5 > 0\},\]
		\[\hspace{1cm}\{x^*_2-\frac{x^*_3 x^*_4}{x^*_5} = c; x^*_4 < 0, x^*_5 < 0\}, \{x^*_2-\frac{x^*_3 x^*_4}{x^*_5} = c; x^*_4 > 0, x^*_5 < 0\}.\]		
		They are also four connected components of 
		\[\{x^*_2-\frac{x^*_3 x^*_4}{x^*_5} = c\} \setminus \bigl(\{x^*_3 = 0\} \cup \{x^*_4 = 0\}\bigr).\]
		Now we consider a $K$-orbit $\{x^*_2-\frac{x^*_3 x^*_4}{x^*_5} = c; x^*_4 > 0, x^*_5 > 0\}$. To visually illustrate the geometry, we 
project this $K$-orbit onto the 4-dimensional plane $\{x^*_1 = x^* = y^* = 0\}$, then by considering its cross-section with the hyperplane $x^*_3 = a$ (constant), we get the image of this $K$-orbit in Figure 1 below. 
		
	\begin{figure} [!htp]
			\centering 
			\includegraphics[draft=false,scale=0.70]{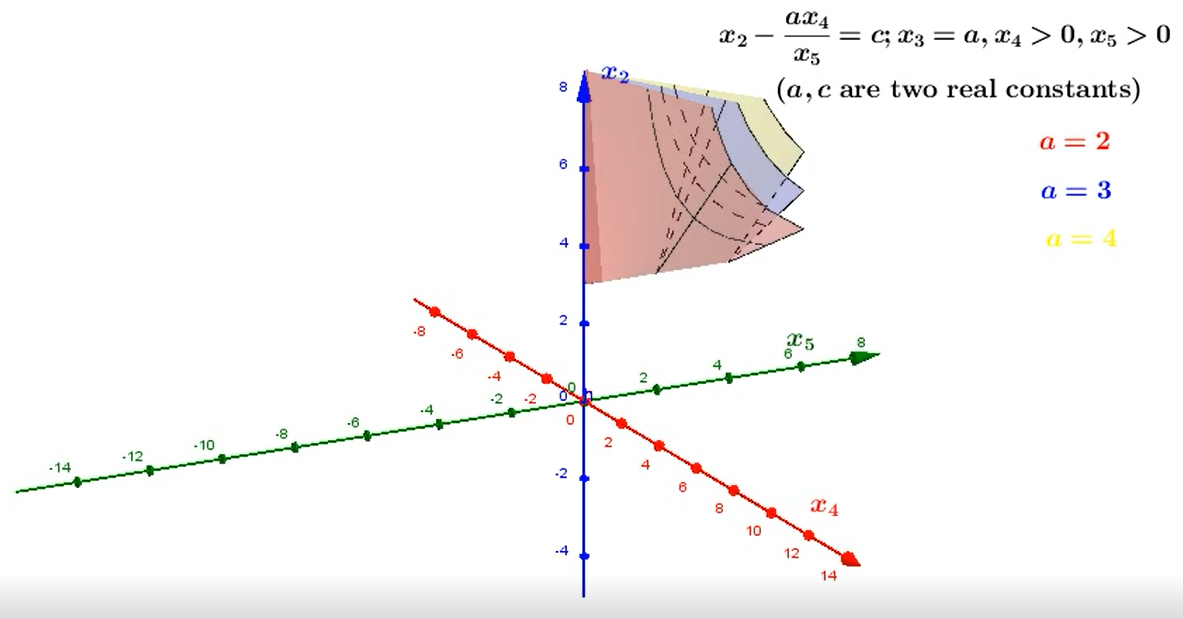}
			\centerline{Figure 1: Geometric illustration of $K$-orbit}
		\end{figure} 
\noindent The situation is also exactly the same for the rest of the groups.
		\item Now we consider the group $G_2$. By letting $\alpha_5 \rightarrow 0$, of course we have
		\[F(\alpha_1,\alpha_2,\alpha_3, \alpha_4, \alpha_5,\alpha,\beta) \rightarrow F_0 (\alpha_1,\alpha_2,\alpha_3, \alpha_4,0,\alpha,\beta).\] 
		But then, it is clear that the orbit $\Omega_F$ of $F$ does not ``converge" to the one $\Omega_{F_0}$  of $F_0$. In other words, four $K$-orbits of the first type seem to play a role ``singularity" and should be excluded from the family of generic $K$-orbits. 
		
\noindent For all remaining groups, we also have a completely similar situation.
		
		\item For any group $G$ from the set $\{G_{2}, G_{3}, G_4^{00}, G_{9},  G_{10}^{\lambda}\}$, it is worth noticing that the family $\mathcal{F}_G$ of all $K$-orbits which are parts of hypersurfaces (i.e. the orbits of second type) forms a cover of the same open set 
		\[ V: = \{(x^*_1, \dots, x^*_5, x^*, y^*) \in \G^* : x^*_4 x^*_5 \neq 0\} \subset \G^*.\]
		Moreover, $V$ can be identified with $\R^3 \times (\R \setminus \{0\}) \times (\R \setminus \{0\}) \times \R^2$ which is open in $\R^7$. That means
		\begin{equation}\label{FoliatedManifold}
		\bigcup\{\Omega/\Omega\in \mathcal{F}_G\} = V = V_{++} \sqcup V_{-+} \sqcup V_{- -} \sqcup V_{+-} \tag{3}
		\end{equation}
		where
		\begin{equation}\label{SS1}
		V_{++} : = \R^3 \times \R_{+} \times \R_{+} \times \R^2, \, \, V_{-+} : = \R^3 \times \R_{-} \times \R_{+} \times \R^2 \tag{4}
		\end{equation}
		\begin{equation}\label{SS2}
		V_{--} : = \R^3 \times \R_{-} \times \R_{-} \times \R^2, \, \, V_{+-+} : = \R^3 \times \R_{+} \times \R_{-} \times \R^2 \tag{5}
		\end{equation}
		and $\R_{+}$ and $\R_{-}$ are, as usual, the sets of all real numbers which are positive and negative, respectively. 
		
		Obviously, $V$ is open submanifold of $\G^*$ (with natural differential structure) and each $K$-orbit $\Omega$ from $\mathcal{F}_G$ is a 6-dimensional submanifold of $V \subset \G^*$.		
		In the next subsection, we will prove that $\mathcal{F}_G$ forms a measurable foliation on the open submanifold $V$.
	\end{enumerate}
\end{rem}

%Subsection 3.3
\subsection{Foliations Formed by the Generic $\boldsymbol{K}$-orbits of Considered Lie \\Groups}

We will now establish in this subsection the remaining new results of the paper on the foliations formed by the (generic) 6-dimensional $K$-orbits of the second type of each Lie group $G \in \{G_{2}, G_{3},  G_4^{00}, G_{9}, G_{10}^{\lambda}\}$. Recall that $V$ is the open submanifold of $\G^*$ which is given by Equality (\ref{FoliatedManifold}) and $\mathcal{F}_G$ is the family of all 6-dimensional $K$-orbits of $G$ such that each of them is a part of a certain hypersurface as mentioned in Item (vi) of Remark \ref{RemarkOrbits}.

%Theorem 3.5
\begin{thm}\label{FormedFoliation}
	For any group $G \in \{G_{2}, G_{3},  G_4^{00}, G_{9}, G_{10}^{\lambda}\}$, the family $\mathcal{F}_G$ forms a measurable foliation on the open manifold $V$ (in sense of Connes \cite{Con82}) and it is called the foliation associated with $G$.
\end{thm}

%Proof of Theorem 3.5
\begin{proof}
	The proof is analogous to the case of MD-groups in \cite{V87, V90, VT08, VH09, VHT14}. For any group $G$ in the considered set, we first need build a suitable differential system $S_G$ of rank 6 on the manifold $V$ such that each K-orbit $\Omega$ from ${\mathcal{F}}_{G}$ is a maximal
	connected integrable submanifold corresponding to this system. As the next	step, we have to show that the Lebegues measure is invariant for some smooth polyvector field $\X$ of degree 6 such that it generates the system $S_G$.
\begin{enumerate}
	\item For $G = G_2$.  For any $F(\alpha_1,\alpha_2, \alpha_3, \alpha_4, ,\alpha_5,\alpha,\beta) \in \G^*$,
	by Theorem \ref{picture K-orbit}, the $K$-orbits $\Omega_F$ belongs to $\mathcal{F}_G$ if and only if $\alpha_4\alpha_5 \neq 0$. Furthermore we have 
	\[\begin{array}{r}
	\Omega_F=\{(x^*_1, x^*_2, x^*_3, x^*_4, x^*_5, x^*, y^*) \in \G^*: x^*_2-\frac{x^*_3 x^*_4}{x^*_5} =\frac{\alpha_2\alpha_5-\alpha_3\alpha_4}{\alpha_5}, \alpha_4 x^*_4 > 0, \, \alpha_5 x^*_5 >0 \}.\end{array}\]
	
	On the open submanifold $V (\subset \G^*)$ with natural differential structure), we consider the following differential system $S_G = \{\X_1, \X_2, \X_3, \X_4, \X_5, \X_6\}$ 
	\[S_{G}: \begin{cases}\begin{array}{lllllll}
	\X_1:=\frac{\partial}{\partial x^*_1}&&&&&&\\
	\X_2:=&\frac{x^*_4}{x^*_5}\frac{\partial}{\partial x^*_2}&+\frac{\partial}{\partial x^*_3}&&&&\\
	\X_3:=&\frac{x^*_3}{x^*_5}\frac{\partial}{\partial x^*_2}&&+\frac{\partial}{\partial x^*_4}&&&\\
	\X_4:=&-\frac{x^*_3x^*_4}{x^{*2}_5}\frac{\partial}{\partial x^*_2}&&&+\frac{\partial}{\partial x^*_5}&&\\
	\X_5:=&&&&&\frac{\partial}{\partial x^*}&\\
	\X_6:=&&&&&&\frac{\partial}{\partial y^*}.\end{array}
	\end{cases}\]
%	The definite matrix of $S_{G}$ is
%	\[M=\begin{bmatrix}
%	1 & & & & &  \\
%	&\frac{x^*_4}{x^*_5}&\frac{x^*_3}{x^*_5}&-\frac{x^*_3x^*_4}{x^{*2}_5}& & \\
%	& 1& & & & & \\
%	& & 1 & & & \\
%	& & & 1& &\\
%	& & & & 1&  \\
%	& & & & & 1\\
%	\end{bmatrix}.\]
	Obviously, $\Rank(S_G)=6$. Furthermore, $\X_i$ is smooth all over $V$, $i=\overline{1,6}$. Now, we will show that $S_{G}$ produces $\mathcal{F}_{G}$, i.e. show that each K-orbit $\Omega$ from ${\mathcal{F}}_{G}$ is a maximal connected integrable submanifold corresponding to $S_{G}$.
	
	Firstly, we consider three vector fields  $\X_1, \X_5, \X_6$. Obviously, the flows (i.e. one-parameter subgroups) of $\X_1, \X_5$ and $\X_6$ are determined as follows
	\begin{align*}
	\theta^{\X_1}_{x^*_1-\alpha_1}: F \mapsto \theta^{\X_1}_{x^*_1-\alpha_1}(F):=(x^*_1, \alpha_2,\alpha_3, \alpha_4, \alpha_5, \alpha, \beta)\\
	\theta^{\X_5}_{x^*-\alpha}: F \mapsto \theta^{\X_5}_{x^*-\alpha}(F):=(\alpha_1, \alpha_2,\alpha_3, \alpha_4, \alpha_5, x^*, \beta)\\
	\theta^{\X_6}_{y^*-\beta}: F \mapsto \theta^{\X_6}_{y^*-\beta}(F):=(\alpha_1, \alpha_2,\alpha_3, \alpha_4, \alpha_5, \alpha, y^*).
	\end{align*}
	
	Next, we consider $\X_2:=\frac{x^*_4}{x^*_5}\frac{\partial}{\partial x^*_2}+\frac{\partial}{\partial x^*_3}$. Assume that 
	\[\varphi: t \mapsto \varphi(t)=\bigl(x^*_1(t), x^*_2(t), x^*_3(t), x^*_4(t), x^*_5(t), x^*(t), y^*(t)\bigr)\]
	is an integral curve of $\X_2$ passing through $F=\varphi(0)$, where $t \in (-\epsilon, \epsilon) \subset \R$ for some positive real number $\epsilon$. Hence, we have
	\begin{align}\label{XG2}
	\varphi'(t)={\X_2}_{\varphi(t)} \Leftrightarrow & \sum_{i=1}^{5}x^{*'}_i(t)\frac{\partial}{\partial x^*_i}+x^{*'}(t)\frac{\partial}{\partial x^*}+y^{*'}(t)\frac{\partial}{\partial y^*} = \frac{x^*_4}{x^*_5}\frac{\partial}{\partial x^*_2}+\frac{\partial}{\partial x^*_3} \nonumber\\
	\Leftrightarrow &  \begin{cases}
	x^{*'}_1(t)= x^{*'}_4(t)= x^{*'}_5(t)=x^{*'}(t)= y^{*'}(t)=0\\
	x^{*'}_2(t)=\frac{x^{*}_4(t)}{x^{*}_5(t)}\\
	x^{*'}_3(t)=1
	\end{cases} \nonumber\\
	\Leftrightarrow &  \begin{cases}
	x^*_1,  x^*_4, x^*_5, x^*, y^* \mbox{ are constant functions}\\
	x^{*}_2(t)=\frac{x^{*}_4}{x^{*}_5}t + \const\\ x^{*}_3(t)=t+\const. \tag{6}
	\end{cases}
	\end{align}
	Combining with condition $F=\varphi(0)$, Equation (\ref{XG2}) gives us
	\[x^*_1=\alpha_1, x^*_2=\alpha_2+\frac{\alpha_4}{\alpha_5}t, x^*_3=\alpha_3+t, x^*_4=\alpha_4, x^*_5=\alpha_5, x^*=\alpha, y^*=\beta.\]
	Therefore, the flow %[ see \cite{Lee13}, p. 209-211] 
	of $\X_2$ is
	\begin{align*}
	\theta^{\X_2}_{x^*_3-\alpha_3}: F \mapsto \theta^{\X_2}_{x^*_3-\alpha_3}(F):=&(\alpha_1, \alpha_2+(x^*_3-\alpha_3)\frac{\alpha_4}{\alpha_5},x^*_3, \alpha_4, \alpha_5, \alpha, \beta)\\
	=&(\alpha_1, \frac{x^*_3\alpha_4}{\alpha_5}+\alpha_2-\frac{\alpha_3\alpha_4}{\alpha_5},x^*_3, \alpha_4, \alpha_5, \alpha, \beta)
	\end{align*}
	Similarly, the flows of $\X_3$ and $\X_4$ are determined as follows
	\begin{align*}
	&\theta^{\X_3}_{x^*_4-\alpha_4}: F \mapsto \theta^{\X_3}_{x^*_4-\alpha_4}(F):=(\alpha_1, \frac{\alpha_3x^*_4}{\alpha_5}+\alpha_2-\frac{\alpha_3\alpha_4}{\alpha_5},\alpha_3, x^*_4, \alpha_5, \alpha, \beta)\\
	&\theta^{\X_4}_{x^*_5-\alpha_5}: F \mapsto \theta^{\X_4}_{x^*_5-\alpha_5}(F):=(\alpha_1, \frac{\alpha_3\alpha_4}{x^*_5}+\alpha_2-\frac{\alpha_3\alpha_4}{\alpha_5},\alpha_3, \alpha_4, x^*_5, \alpha, \beta).
	\end{align*}
	Finally, by setting $\theta=\theta^{\X_6}_{y^*-\beta}\circ\theta^{\X_5}_{x^*-\alpha}\circ\theta^{\X_4}_{x^*_5-\alpha_5}\circ\theta^{\X_3}_{x^*_4-\alpha_4}\circ\theta^{\X_2}_{x^*_3-\alpha_3}\circ\theta^{\X_1}_{x^*_1-\alpha_1}$, we~have
	\begin{align*}
	\theta(F)=&\theta^{\X_6}_{y^*-\beta}\circ\theta^{\X_5}_{x^*-\alpha}\circ\theta^{\X_4}_{x^*_5-\alpha_5}\circ\theta^{\X_3}_{x^*_4-\alpha_4}\circ\theta^{\X_2}_{x^*_3-\alpha_3}\circ\theta^{\X_1}_{x^*_1-\alpha_1}(F)\\
	=&(x^*_1, \frac{x^*_3x^*_4}{x^*_5}+\alpha_2-\frac{\alpha_3\alpha_4}{\alpha_5}, x^*_3, x^*_4, x^*_5, x^*, y^*)\\
	=&(x^*_1, x^*_2, x^*_3, x^*_4, x^*_5, x^*, y^*)
	\end{align*}
	where $  x^*_1, x^*, y^* \in \R$ and %By direct calculation, we get 
	$$x^*_2=\frac{x^*_3x^*_4}{x^*_5} +\alpha_2-\frac{\alpha_3\alpha_4}{\alpha_5}\Leftrightarrow x^*_2-\frac{x^*_3x^*_4}{x^*_5}=\frac{\alpha_2\alpha_5-\alpha_3\alpha_4}{\alpha_5}.$$
	Hence $\{\theta(F):\,x^*_1, x^*_3, x^*_4, x^*_5, x^*, y^* \in \R;\, \alpha_4x^*_4>0; \,\alpha_5x^*_5>0 \}$ $\equiv\Omega_F$, i.e. $\Omega_F$ is a maximal connected integrable submanifold corresponding to $S_{G}$. Therefore $S_{G}$ generates $\mathcal{F}_{G}$. In other words, $(V,\mathcal{F}_{G})$ is a 6-dimensional foliation for the case $G = G_2$.
	
	We now turn to the second step of the proof. Namely, we have to show that the foliation $(V,\mathcal{F}_{G})$ is measurable in the sense of Connes. As mentioned in Subsection \ref{Subsection2.2} of Preliminaries, to prove that $(V, \mathcal{F}_G)$ is measurable, we only need to choose some suitable pair ($\X, \mu$) on $V$ where $\X$ is some smooth 6-vector field defined on $V$, $\mu$ is some measure on $V$ such that $\X$ generates $S_G$ and $\mu$ is $\X$-invariant. 
	Namely, we choose $\mu$ to be exactly the Lebegues measure on $V$ and set $\X: = \X_1 \wedge \X_2 \wedge \X_3 \wedge \X_4 \wedge \X_5 \wedge \X_6$.
%	\[\X: = \X_1 \wedge \X_2 \wedge \X_3 \wedge \X_4 \wedge \X_5 \wedge \X_6.\] 
	Clearly, $\X$ is smooth, non-zero everywhere on $V$ and it is exactly a polyvector field of degree 6. Moreover $\X$ generates $S_G$. That is, when choosing on $(V, \mathcal{F}_G)$ a suitable orientation, then $\X \in C^{\infty}{\bigl({\Lambda}^{6}(\mathcal{F})\bigr)}^{+}$. It is obvious that the invariance of the Lebegues measure $\mu$ with respect to $\X$ is equivalent to the invariance of $\mu$ for the $K$-representation that is restricted to the foliated submanifold $V$ in $\G^*$. For any $U(x_1, x_2, x_3, x_4, x_5, x, y) \in \G$, direct computation show that Jacobi's determinant
	$J_U$ of differential mapping $K\bigl(\exp_G(U)\bigr)$ is a constant which depends only on U but do not depend on the coordinates of any point which moves in each $K$-orbit $\Omega \in \mathcal{F}_G$. In other words, the Lebegues measure $\mu$ is $\X$-invariant. The proof is complete for the case $G = G_2$.
	
	\item For $G = G_3$.  For any $F(\alpha_1,\alpha_2, \alpha_3, \alpha_4, ,\alpha_5,\alpha,\beta) \in \G^*_3$, by Theorem \ref{picture K-orbit}, the $K$-orbits $\Omega_F$ belongs to $\mathcal{F}_{G_3}$ if and only if $\alpha_4\alpha_5\neq0$. Moreover we have 
\[\begin{array}{r}
\Omega_F=\{(x^*_1, x^*_2, x^*_3, x^*_4, x^*_5, x^*, y^*) \in \G^*_3: \frac{x^*_2}{x^*_4} - \frac{x^*_3}{x^*_5}=\frac{\alpha_2}{\alpha_4} - \frac{\alpha_3}{\alpha_5},\alpha_4 x^*_4 > 0, \, \alpha_5 x^*_5 >0 \}.\end{array}\]

On the open subset $V$ ($\subset\G^*_3$) with natural differential structure, we consider the following differential system $S_{G_3} = \{\X_1, \X_2, \X_3, \X_4, \X_5, \X_6\}$ 
\[\hspace{1cm}S_{G_3}:\begin{cases}\begin{array}{lllllll}
\X_1:=\frac{\partial}{\partial x^*_1}& & & & & &\\
\X_2:=&  x^*_2\frac{\partial}{\partial  x^*_2}&&+x^*_4\frac{\partial}{\partial  x^*_4}& & &\\
\X_3:=& &x^*_3\frac{\partial}{\partial  x^*_3}& &+x^*_5\frac{\partial}{\partial  x^*_5} & &\\
\X_4:=&x^*_4\frac{\partial}{\partial  x^*_2} &+x^*_5\frac{\partial}{\partial  x^*_3}& & & &  \\
\X_5:= & & & & &\frac{\partial}{\partial x^*}& \\
\X_6:=& & & & & &\frac{\partial}{\partial y^*}.
\end{array}
\end{cases}\]
The definite matrix of $S_{G_3}$ is
\[M=\begin{bmatrix}
1 & & & & & \\
&x^*_2 &&x^*_4& & \\
& & x^*_3&x^*_5& & & \\
&x^*_4 &  & & &\\
& & x^*_5& & & \\
& & & & 1& \\
& & & & & 1\\
\end{bmatrix}.\]
Obviously, $\Rank(S_{G_3}) = \Rank(M)=6$. Furthermore, $\X_i$ is smooth all over $V$, $i=\overline{1,6}$. Now, we will show that $S_{{G_3}}$ produces $\mathcal{F}_{{G_3}}$, i.e. show that each K-orbit $\Omega$ from ${\mathcal{F}}_{{G_3}}$ is a maximal connected integrable submanifold corresponding to~$S_{{G_3}}$.

Firstly, we consider three vector fields $\X_1, \X_5, \X_6$. Obviously, the flows of $\X_1, \X_5$ and $\X_6$ are determined as follows
\begin{align*}
\theta^{\X_1}_{x^*_1-\alpha_1}: F \mapsto \theta^{\X_1}_{x^*_1-\alpha_1}(F):=(x^*_1, \alpha_2,\alpha_3, \alpha_4, \alpha_5, \alpha, \beta)\\
\theta^{\X_5}_{x^*-\alpha}: F \mapsto \theta^{\X_5}_{x^*-\alpha}(F):=(\alpha_1, \alpha_2,\alpha_3, \alpha_4, \alpha_5, x^*, \beta)\\
\theta^{\X_6}_{y^*-\beta}: F \mapsto \theta^{\X_6}_{y^*-\beta}(F):=(\alpha_1, \alpha_2,\alpha_3, \alpha_4, \alpha_5, \alpha, y^*).
\end{align*}

Next, we consider $\X_2:=x^*_2 \frac{\partial}{\partial  x^*_2}+x^*_4\frac{\partial}{\partial  x^*_4}$. Assume that 

$$\varphi: x \mapsto \varphi(x)=(x^*_1(x), x^*_2(x), x^*_3(x), x^*_4(x), x^*_5(x), x^*(x), y^*(x))$$
be an integral curve of $\X_2$ from $F=\varphi(0)$, where $x \in (-\epsilon, \epsilon) \subset \R$ for some positive real number $\epsilon$. Then we have
\begin{align}\label{G3X2}
\hspace{0.5cm}\varphi'(x)={\X_2}_{\varphi(x)} \Leftrightarrow & \sum_{i=1}^{5}x^{*'}_i(x)\frac{\partial}{\partial x^*_i}+x^{*'}(x)\frac{\partial}{\partial x^*}+y^{*'}(x)\frac{\partial}{\partial y^*} =x^*_4(x)\frac{\partial}{\partial  x^*_2} +x^*_5(x)\frac{\partial}{\partial  x^*_3} \nonumber\\
\Leftrightarrow &  \begin{cases}
x^{*'}_1(x)= x^{*'}_3(x)= x^{*'}_5(x)=x^{*'}(x)= y^{*'}(x)=0\\
x^{*'}_2(x)=x^{*}_2(x)\\
x^{*'}_4(x)=x^*_4(x).\tag{7}
\end{cases}
\end{align}
Combining with condition $F=\varphi(0)$, Equation \eqref{G3X2} gives us
\begin{equation}\label{X2G3}
x^*_1=\alpha_1,\,x^*_2= \alpha_2e^x,\, x^*_3=\alpha_3,\, x^*_4=\alpha_4e^x, \,x^*_5=\alpha_5,\, x^*=\alpha,\, y^*=\beta.\tag{8}
\end{equation}
Similarly, we consider $\X_3:=x^*_3\frac{\partial}{\partial  x^*_3} +x^*_5\frac{\partial}{\partial  x^*_5}$. Assume that

$$\varphi: y \mapsto \varphi(y)=(x^*_1(y), x^*_2(y), x^*_3(y), x^*_4(y), x^*_5(y), x^*(y), y^*(y))$$
be an integral curve of $\X_3$ from $F=\varphi(0)$, where $y \in (-\epsilon, \epsilon) \subset \R$ for some positive real number $\epsilon$. Then we have
\begin{align}\label{G3X3}
\hspace{1cm}\varphi'(y)={\X_3}_{\varphi(y)}
\Leftrightarrow & \sum_{i=1}^{5}x^{*'}_i(y)\frac{\partial}{\partial x^*_i}+x^{*'}(y)\frac{\partial}{\partial x^*}+y^{*'}(y)\frac{\partial}{\partial y^*}=x^*_3(y)\frac{\partial}{\partial  x^*_3} +x^*_5(y)\frac{\partial}{\partial  x^*_5} \nonumber\\
\Leftrightarrow &  \begin{cases}
x^{*'}_1(y)= x^{*'}_2(y)=x^{*'}_4(y)= x^{*'}(y)= y^{*'}(y)=0\\
x^{*'}_3(y)=x^*_3(y)\\
x^{*'}_5(y)=x^*_5(y). \tag{9}
\end{cases}
\end{align}
Combining with condition $F=\varphi(0)$, Equation \eqref{G3X3} gives us
\begin{equation}\label{X3G3}
\hspace{0.5cm}x^*_1=\alpha_1,\,x^*_2= \alpha_2,\, x^*_3=\alpha_3e^y, \,x^*_4=\alpha_4,\, x^*_5=\alpha_5e^y,\,x^*=\alpha,\,y^*=\beta.\tag{10}
\end{equation}
Similarly, we consider $\X_4:=x^*_4 \frac{\partial}{\partial  x^*_2} +x^*_5\frac{\partial}{\partial  x^*_3}$. Assume that
$$\varphi: x_1 \mapsto \varphi(x_1)=(x^*_1(x_1), x^*_2(x_1), x^*_3(x_1), x^*_4(x_1), x^*_5(x_1), x^*(x_1), y^*(x_1))$$
be an integral curve of $\X_4$ from $F=\varphi(0)$, where $x_1 \in (-\epsilon, \epsilon) \subset \R$ for some positive real number $\epsilon$. Then we have
\begin{align}\label{G3X4}
\varphi'(x_1)={\X_4}_{\varphi(x_1)}
\Leftrightarrow & \sum_{i=1}^{5}x^{*'}_i(x_1)\frac{\partial}{\partial x^*_i}+x^{*'}(x_1)\frac{\partial}{\partial x^*}+y^{*'}(x_1)\frac{\partial}{\partial y^*}=x^*_4(x_1) \frac{\partial}{\partial  x^*_2} +x^*_5(x_1)\frac{\partial}{\partial  x^*_3} \nonumber\\
\Leftrightarrow &  \begin{cases}
x^{*'}_1(x_1)= x^{*'}_4(x_1)=
x^{*'}_5(x_1)= x^{*'}(x_1)= y^{*'}(x_1)=0\\
x^{*'}_2(x_1)=x^*_4(x_1)\\
x^{*'}_3(x_1)=x^*_5(x_1). \tag{11}
\end{cases} 
\end{align}
Combining with condition $F=\varphi(0)$, Equation \eqref{G3X4} gives us
\begin{align}\label{X4G3}
&x^*_1=\alpha_1,\, x^*_2= \alpha_2+\alpha_4x_1,\, x^*_3=\alpha_3+\alpha_5x_1,\nonumber \\
&x^*_4=\alpha_4,\,	 x^*_5=\alpha_5,\, x^*=\alpha,\, y^*=\beta.\tag{12}
\end{align}
According to (\ref{X2G3}), (\ref{X3G3}), (\ref{X4G3}), the flows of $\X_2, \X_3, \X_4$ are
\begin{align*}
&\theta^{\X_2}_{x}: F \mapsto \theta^{\X_2}_{x}(F):=(\alpha_1, \alpha_2e^x,\alpha_3, \alpha_4e^x, \alpha_5, \alpha, \beta)\\
&\theta^{\X_3}_{y}: F \mapsto \theta^{\X_3}_{y}(F):=(\alpha_1, \alpha_2,\alpha_3e^y, \alpha_4, \alpha_5e^y, \alpha, \beta)\\
&\theta^{\X_4}_{x_1}: F \mapsto \theta^{\X_4}_{x_1}(F):=(\alpha_1, \alpha_2+\alpha_4x_1,\alpha_3+\alpha_5x_1, \alpha_4, \alpha_5, \alpha, \beta).
\end{align*}
Finally, by setting $\theta=\theta^{\X_6}_{y^*-\beta}\circ\theta^{\X_5}_{x^*-\alpha}\circ\theta^{\X_4}_{x_1}\circ\theta^{\X_3}_{y}\circ\theta^{\X_2}_{x}\circ\theta^{\X_1}_{x^*_1-\alpha_1}$, we have

\begin{align*}
\theta(F)=&\theta^{\X_6}_{y^*-\beta}\circ\theta^{\X_5}_{x^*-\alpha}\circ\theta^{\X_4}_{x_1}\circ\theta^{\X_3}_{y}\circ\theta^{\X_2}_{x}\circ\theta^{\X_1}_{x^*_1-\alpha_1}(F)\\
=&(x^*_1, x^*_2, x^*_3, x^*_4, x^*_5, x^*, y^*)
\end{align*}
where $x^*_1, x^*, y^* \in \R$ and
\[x^*_2= \alpha_2e^x+\alpha_4e^xx_1, x^*_3= \alpha_3e^y+\alpha_5e^yx_1, x^*_4= \alpha_4e^x, x^*_5=\alpha_5e^y.\]
By direct calculation, we get 
\[x^*_2=(\frac{x^*_3}{x^*_5}+\frac{\alpha_2}{\alpha_4}-\frac{\alpha_3}{\alpha_5})x^*_4 \Leftrightarrow  \frac{x^*_2}{x^*_4} - \frac{x^*_3}{x^*_5}=\frac{\alpha_2}{\alpha_4} - \frac{\alpha_3}{\alpha_5}.\]
Hence $\{\theta(F) : x^*_1, x_1, x, y, x^*, y^* \in \R; \alpha_4x^*_4>0; \alpha_5x^*_5>0 \}\equiv\Omega_F$, i.e. $\Omega_F$ is a maximal connected integrable submanifold corresponding to $S_{G_3}$. Therefore $S_{G_3}$ generate $\mathcal{F}_{G_3}$. In other words, $(V,\mathcal{F}_{G_3})$ is a 6-dimensional foliation.

We now turn to the second step of the proof. Namely, we have to show that the foliation $(V,\mathcal{F}_{G_3})$ is measurable in the sense of Connes. To prove that $(V, \mathcal{F}_{G_3})$ is measurable, we only need to choose some suitable pair ($\X, \mu$) on $V$ where $\X$ is some smooth 6-vector field defined on $V$, $\mu$ is some measure on $V$ such that $\X$ generates $S_{G_3}$ and $\mu$ is $\X$-invariant. 
Namely, we choose $\mu$ to be exactly the Lebegues measure on $V$ and set $\X: = \X_1 \wedge \X_2 \wedge \X_3 \wedge \X_4 \wedge \X_5 \wedge \X_6$.
Clearly, $\X$ is smooth, non-zero everywhere on $V$ and it is exactly a polyvector field of degree 6. Moreover $\X$ generates $S_{G_3}$. That is, when choosing on $(V, \mathcal{F}_{G_3})$ a suitable orientation, then $\X \in C^{\infty}{\bigl({\Lambda}^{6}(\mathcal{F})\bigr)}^{+}$. It is obvious that the invariance of the Lebegues measure $\mu$ with respect to $\X$ is equivalent to the invariance of $\mu$ for the $K$-representation that is restricted to the foliated submanifold $V$ in $\G_3^*$. For any $U(x_1, x_2, x_3, x_4, x_5, x, y) \in \G_3$, direct computation show that Jacobi's determinant
$J_U$ of differential mapping $K\bigl(\exp_{G_3}(U)\bigr)$ is a constant which depends only on U but do not depend on the coordinates of any point which moves in each $K$-orbit $\Omega \in \mathcal{F}_{G_3}$. In other words, the Lebegues measure $\mu$ is $\X$-invariant. The proof is complete for the case $G = G_3$.

\item For $G = G^{00}_4$.  For any $F(\alpha_1,\alpha_2, \alpha_3, \alpha_4, ,\alpha_5,\alpha,\beta) \in \G^{00*}_4$, by Theorem \ref{picture K-orbit}, the $K$-orbits $\Omega_F$ belongs to $\mathcal{F}_{G^{00}_4}$ if and only if $\alpha_4\alpha_5 \neq 0$. Moreover we have 
\[\hspace{1cm}\begin{array}{r}
\Omega_F=\{(x^*_1, x^*_2, x^*_3, x^*_4, x^*_5, x^*, y^*) \in \G^{00*}_4: x^*_3-\frac{x^*_2 x^*_5}{x^*_4} = \frac{\alpha_3\alpha_4-\alpha_2\alpha_5}{\alpha_4},\\ \alpha_4 x^*_4 > 0, \, \alpha_5 x^*_5 >0 \}.\end{array}\]

Now, on the open subset $V$ ( $\subset\G^{00*}_4$) with natural differential structure, we consider the following differential system $S_{G^{00}_4} = \{\X_1, \X_2, \X_3, \X_4, \X_5, \X_6\}$ 
\[\hspace{1cm}S_{G^{00}_4}:  \begin{cases}\begin{array}{lllllll}
\X_1:=\frac{\partial}{\partial x^*_1}& & & & & &\\
\X_2:=& \frac{\partial}{\partial  x^*_2} &+\frac{x^*_5}{x^*_4}\frac{\partial}{\partial  x^*_3}& & & &  \\
\X_3:=& &-\frac{x^*_2x^*_5}{x^{*2}_4}\frac{\partial}{\partial  x^*_3}& &+\frac{\partial}{\partial  x^*_4} & &\\
\X_4:=& & \frac{x^*_2}{x^*_4}\frac{\partial}{\partial  x^*_3}&&+\frac{\partial}{\partial  x^*_5}& & \\
\X_5:= & & & & &\frac{\partial}{\partial x^*}& \\
\X_6:=& & & & & &\frac{\partial}{\partial y^*}.
\end{array}
\end{cases}\]
The definite matrix of $S_{G^{00}_4}$ is
\[M=\begin{bmatrix}
1 & & & & &\\
&1& && & \\
&\frac{x^*_5}{x^*_4} &-\frac{x^*_2x^*_5}{x^*_4} & \frac{x^*_2}{x^*_4}& & & \\
&  &1 & & &\\
& & &1 & &\\
& & & & 1&\\
& & & & &1\\
\end{bmatrix}.\]
Obviously, $\Rank(S_{G^{00}_4}) = \Rank(M)=6$. Furthermore, $\X_i$ is smooth all over $V$, $i=\overline{1,6}$. Now, we will show that $S_{G^{00}_4}$ produces $\mathcal{F}_{G^{00}_4}$, i.e. show that each K-orbit $\Omega$ from ${\mathcal{F}}_{G^{00}_4}$ is a maximal connected integrable submanifold corresponding to~$S_{G^{00}_4}$.

Firstly, we consider three vector fields $\X_1, \X_5, \X_6$. Obviously, the flows of $\X_1, \X_5$ and $\X_6$ are determined as follows
\begin{align*}
\theta^{\X_1}_{x^*_1-\alpha_1}: F \mapsto \theta^{\X_1}_{x^*_1-\alpha_1}(F):=(x^*_1, \alpha_2,\alpha_3, \alpha_4, \alpha_5, \alpha, \beta)\\
\theta^{\X_5}_{x^*-\alpha}: F \mapsto \theta^{\X_5}_{x^*-\alpha}(F):=(\alpha_1, \alpha_2,\alpha_3, \alpha_4, \alpha_5, x^*, \beta)\\
\theta^{\X_6}_{y^*-\beta}: F \mapsto \theta^{\X_6}_{y^*-\beta}(F):=(\alpha_1, \alpha_2,\alpha_3, \alpha_4, \alpha_5, \alpha, y^*).
\end{align*}

Next, we consider $\X_2:=\frac{\partial}{\partial  x^*_2} +\frac{x^*_5}{x^*_4}\frac{\partial}{\partial  x^*_3}$.  Assume that 

$$\varphi: t \mapsto \varphi(t)=(x^*_1(t), x^*_2(t), x^*_3(t), x^*_4(t), x^*_5(t), x^*(t), y^*(t))$$
be an integral curve of $\X_2$ from $F=\varphi(0)$, where $t \in (-\epsilon, \epsilon) \subset \R$ for some positive real number $\epsilon$. Then we get
\begin{align}\label{G4X2}
\varphi'(t)={\X_2}_{\varphi(t)}
\Leftrightarrow & \sum_{i=1}^{5}x^{*'}_i(t)\frac{\partial}{\partial x^*_i}+x^{*'}(t)\frac{\partial}{\partial x^*}+y^{*'}(t)\frac{\partial}{\partial y^*}=\frac{\partial}{\partial  x^*_2} +\frac{x^*_5(t)}{x^*_4(t)}\frac{\partial}{\partial  x^*_3} \nonumber\\
\Leftrightarrow &  \begin{cases}
x^{*'}_1(t)= x^{*'}_4(t)= x^{*'}_5(t)=x^{*'}(t)= y^{*'}(t)=0\\
x^{*'}_2(t)=1\\
x^{*'}_3(t)=\frac{x^{*}_5(t)}{x^{*}_4(t)}
\end{cases} \nonumber\\
\Leftrightarrow &  \begin{cases}
x^*_1,  x^*_4, x^*_5, x^*, y^* \mbox{ are constant functions}\\
x^{*}_2(t)=t + \const\\ x^{*}_3(t)=\frac{x^{*}_5}{x^{*}_4}t+\const. \tag{13}
\end{cases}
\end{align}
Combining with condition $F=\varphi(0)$, Equation \eqref{G4X2} gives us
\[\hspace{1cm}x^*_1=\alpha_1, x^*_2=\alpha_2+t, x^*_3=\alpha_3+\frac{\alpha_5}{\alpha_4}t, x^*_4=\alpha_4, x^*_5=\alpha_5, x^*=\alpha, y^*=\beta.\]
Therefore, the flow of $\X_2$ is
\begin{align*}
\hspace{1cm}\theta^{\X_2}_{x^*_2-\alpha_2}: F \mapsto \theta^{\X_2}_{x^*_2-\alpha_2}(F):=&(\alpha_1, x^*_2,\alpha_3+(x^*_2-\alpha_2)\frac{\alpha_5}{\alpha_4}, \alpha_4, \alpha_5, \alpha, \beta)\\
=&(\alpha_1, x^*_2,\alpha_3+\frac{x^*_2\alpha_5}{\alpha_4}-\frac{\alpha_2\alpha_5}{\alpha_4}, \alpha_4, \alpha_5, \alpha, \beta).
\end{align*}

Similarly,  we consider $\X_3:=-\frac{x^*_2x^*_5}{x^{*2}_4}\frac{\partial}{\partial  x^*_3}+\frac{\partial}{\partial  x^*_4}$. Assume that
$$\varphi: u \mapsto \varphi(u)=(x^*_1(u), x^*_2(u), x^*_3(u), x^*_4(u), x^*_5(u), x^*(u), y^*(u))$$
be an integral curve of $\X_3$ from $F=\varphi(0)$, where $u \in (-\epsilon, \epsilon) \subset \R$ for some positive real number $\epsilon$. Then we have
\begin{align}\label{X3G4}
\hspace{1cm}\varphi'(u)={\X_4}_{\varphi(u)}
\Leftrightarrow & \sum_{i=1}^{5}x^{*'}_i(u)\frac{\partial}{\partial x^*_i}+x^{*'}(u)\frac{\partial}{\partial x^*}+y^{*'}(u)\frac{\partial}{\partial y^*}=\notag\\
&\hspace{3.0cm}=-\frac{x^*_2(u)x^*_5(u)}{x^{*2}_4(u)}\frac{\partial}{\partial  x^*_3}+\frac{\partial}{\partial  x^*_4} \nonumber\\
\Leftrightarrow &  \begin{cases}
x^{*'}_1(u)= x^{*'}_2(u)= x^{*'}_5(u)=x^{*'}(u)= y^{*'}(u)=0\\
x^{*'}_3(u)=-\frac{x^*_2(u)x^*_5(u)}{x^{*2}_4(u)}=-\frac{x^*_2(u)x^*_5(u)}{(u+\const_1)^2}\\
x^{*'}_4(u)=1 \Rightarrow x^*_4(u)=u+\const_1
\end{cases} \nonumber\\
\Leftrightarrow &  \begin{cases}
x^*_1,  x^*_2, x^*_5, x^*, y^* \mbox{ are constant functions}\\
x^{*}_3(u)=\frac{x^*_2x^*_5}{u+\const}_1 + \const_2\\ x^{*}_4(u)=u+\const_1. \tag{14}
\end{cases}
\end{align}
Combining with condition $F=\varphi(0)$, we have 
\[x^*_2(0)=\alpha_2,\, x^*_4(0)=0+\mbox{const}_1=\alpha_4,\, x^*_5(0)=\alpha_5,\]
\[x^*_3(0)=\frac{\alpha_2\alpha_5}{0+\alpha_4}+\mbox{const}_2=\alpha_3.\]
Hence, we get $const_2=\alpha_3-\frac{\alpha_2\alpha_5}{\alpha_4}$. Equation \eqref{X3G4} gives us
\[x^*_1=\alpha_1,x^*_2=\alpha_2, x^*_3=\frac{\alpha_2\alpha_5}{x^*_4}+\alpha_3-\frac{\alpha_2\alpha_5}{\alpha_4}, \]
\[ x^*_4=\alpha_4+u, x^*_5=\alpha_5, x^*=\alpha, y^*=\beta.\]
Therefore, the flow of $\X_3$ is
\[\hspace{1cm}\theta^{\X_3}_{x^*_4-\alpha_4}: F \mapsto \theta^{\X_3}_{x^*_4-\alpha_4}(F):=(\alpha_1,\alpha_2, \frac{\alpha_2\alpha_5}{x^*_4}+\alpha_3-\frac{\alpha_2\alpha_5}{\alpha_4},  x^*_4,\alpha_5, \alpha, \beta).\]
Next, we consider $\X_4:=\frac{x^*_2}{x^*_4}\frac{\partial}{\partial  x^*_3}+\frac{\partial}{\partial  x^*_5}$. Assume that
$$\varphi: s \mapsto \varphi(s)=(x^*_1(s), x^*_2(s), x^*_3(s), x^*_4(s), x^*_5(s), x^*(s), y^*(s))$$
be an integral curve of $\X_4$ from $F=\varphi(0)$, where $s \in (-\epsilon, \epsilon) \subset \R$ for some positive real number $\epsilon$. Then we have		
\begin{align}\label{X4G4}
\hspace{1cm}\varphi'(s)={\X_4}_{\varphi(s)}\nonumber
\Leftrightarrow & \sum_{i=1}^{5}x^{*'}_i(s)\frac{\partial}{\partial x^*_i}+x^{*'}(s)\frac{\partial}{\partial x^*}+y^{*'}(s)\frac{\partial}{\partial y^*}=\frac{x^*_2(s)}{x^*_4(s)}\frac{\partial}{\partial  x^*_3}+\frac{\partial}{\partial  x^*_5}\nonumber\\
\Leftrightarrow &  \begin{cases}
x^{*'}_1(s)= x^{*'}_2(s)= x^{*'}_4(s)=x^{*'}(s)= y^{*'}(s)=0\\
x^{*'}_3(s)=\frac{x^{*}_2(s)}{x^{*}_4(s)}\\
x^{*'}_5(s)=1
\end{cases} \nonumber\\
\Leftrightarrow &  \begin{cases}
x^*_1,  x^*_2, x^*_4, x^*, y^* \mbox{ are constant functions}\\
x^{*}_3(s)=\frac{x^{*}_2}{x^{*}_4}s + \const\\ x^{*}_5(s)=s+\const. \tag{15}
\end{cases}
\end{align}

Combining with condition $F=\varphi(0)$, Equation \eqref{X4G4} gives us
\begin{align*}
&x^*_1=\alpha_1,\; x^*_2=\alpha_2,\; x^*_3=\alpha_3+\frac{\alpha_2}{\alpha_4}s, \;x^*_4=\alpha_4,\; x^*_5=\alpha_5+s, \;x^*=\alpha, y^*=\beta.
\end{align*}
Hence, the flow of $\X_4$ is
\begin{align*}
\hspace{1cm}\theta^{\X_4}_{x^*_5-\alpha_5}: F \mapsto \theta^{\X_4}_{x^*_5-\alpha_5}(F):=(\alpha_1, \alpha_2, \frac{\alpha_2x^*_5}{\alpha_4}+\alpha_3-\frac{\alpha_2\alpha_5}{\alpha_4}, \alpha_4, x^*_5, \alpha, \beta).
\end{align*}
Finally, we set $\theta=\theta^{\X_6}_{y^*-\beta}\circ\theta^{\X_5}_{x^*-\alpha}\circ\theta^{\X_4}_{x^*_5-\alpha_5}\circ\theta^{\X_3}_{x^*_4-\alpha_4}\circ\theta^{\X_2}_{x^*_2-\alpha_2}\circ\theta^{\X_1}_{x^*_1-\alpha_1}$. From here~infer
\begin{align*}
\theta(F)=&\theta^{\X_6}_{y^*-\beta}\circ\theta^{\X_5}_{x^*-\alpha}\circ\theta^{\X_4}_{x^*_5-\alpha_5}\circ\theta^{\X_3}_{x^*_4-\alpha_4}\circ\theta^{\X_2}_{x^*_2-\alpha_2}\circ\theta^{\X_1}_{x^*_1-\alpha_1}(F)\\
%		=&(x^*_1,x^*_2, \alpha_3+\frac{x^*_2x^*_5}{x^*_4}-\frac{\alpha_2\alpha_5}{\alpha_4},x^*_4, x^*_5, x^*, y^*)\\
=&(x^*_1, x^*_2, x^*_3, x^*_4, x^*_5, x^*, y^*)
\end{align*}
where $  x^*_1, x^*, y^* \in \R$ and  
$$x^*_3=\alpha_3+\frac{x^*_2x^*_5}{x^*_4}-\frac{\alpha_2\alpha_5}{\alpha_4} \Leftrightarrow x^*_3-\frac{x^*_2 x^*_5}{x^*_4} = \frac{\alpha_3\alpha_4-\alpha_2\alpha_5}{\alpha_4}.$$
Therefore, $\{\theta(F) : x^*_1, x^*_2, x^*_4, x^*_5, x^*, y^* \in \R; \alpha_4x^*_4>0; \alpha_5x^*_5>0 \}\equiv\Omega_F$, i.e. $\Omega_F$ is a maximal connected integrable submanifold corresponding to $S_{G^{00}_4}$. Thus $S_{G^{00}_4}$ generate $\mathcal{F}_{G^{00}_4}$. In other words, $(V,\mathcal{F}_{G^{00}_4})$ is a 6-dimensional foliation.

We now turn to the second step of the proof. Namely, we have to show that the foliation $(V,\mathcal{F}_{G^{00}_4})$ is measurable in the sense of Connes. To prove that $(V, \mathcal{F}_{G^{00}_4})$ is measurable, we only need to choose some suitable pair ($\X, \mu$) on $V$ where $\X$ is some smooth 6-vector field defined on $V$, $\mu$ is some measure on $V$ such that $\X$ generates $S_{G^{00}_4}$ and $\mu$ is $\X$-invariant. 
Namely, we choose $\mu$ to be exactly the Lebegues measure on $V$ and set $\X: = \X_1 \wedge \X_2 \wedge \X_3 \wedge \X_4 \wedge \X_5 \wedge \X_6$. 
Clearly, $\X$ is smooth, non-zero everywhere on $V$ and it is exactly a polyvector field of degree 6. Moreover $\X$ generates $S_{G^{00}_4}$. That is, when choosing on $(V, \mathcal{F}_{G^{00}_4})$ a suitable orientation, then $\X \in C^{\infty}{\bigl({\Lambda}^{6}(\mathcal{F})\bigr)}^{+}$. It is obvious that the invariance of the Lebegues measure $\mu$ with respect to $\X$ is equivalent to the invariance of $\mu$ for the $K$-representation that is restricted to the foliated submanifold $V$ in $\G^{00*}_4$. For any $U(x_1, x_2, x_3, x_4, x_5, x, y) \in \G^{00}_4$, direct computation show that Jacobi's determinant
$J_U$ of differential mapping $K\bigl(\exp_{G^{00}_4}(U)\bigr)$ is a constant which depends only on U but do not depend on the coordinates of any point which moves in each $K$-orbit $\Omega \in \mathcal{F}_{G^{00}_4}$. In other words, the Lebegues measure $\mu$ is $\X$-invariant. The proof is complete for the case $G = G^{00}_4$.

\item For $G = G_9$.  For any $F(\alpha_1,\alpha_2, \alpha_3, \alpha_4, ,\alpha_5,\alpha,\beta) \in \G^*_9$, by Theorem \ref{picture K-orbit}, the $K$-orbits $\Omega_F$ belongs to $\mathcal{F}_{G_9}$ if and only if $\alpha_4\alpha_5\neq0$. Moreover we have 
\[\hspace{1cm}\begin{array}{r}
\Omega_F=\{(x^*_1, x^*_2, x^*_3, x^*_4, x^*_5, x^*, y^*) \in \G^*_9: \frac{x^*_2}{x^*_4} - \frac{x^*_3}{x^*_5} + \ln\vert x^*_4 \vert =  \frac{\alpha_2}{\alpha_4} - \frac{\alpha_3}{\alpha_5} + \ln\vert \alpha_4 \vert, \\ \alpha_4 x^*_4 > 0, \, \alpha_5 x^*_5 >0 \}.\end{array}\]

Now, on the open subset $V$ ($\subset\G^*_9$) with natural differential structure, we consider the differential system $S_{G_9} = \{\X_1, \X_2, \X_3, \X_4, \X_5, \X_6\}$ defined as~follows
\[\hspace{1cm}S_{G_9}:\begin{cases}\begin{array}{lllllll}
\X_1:=\frac{\partial}{\partial x^*_1}& & & & & &\\
\X_2:=& & x^*_3\frac{\partial}{\partial  x^*_3}&+x^*_5\frac{\partial}{\partial  x^*_5}& & &\\
\X_3:=&x^*_2\frac{\partial}{\partial  x^*_2} &+x^*_5\frac{\partial}{\partial  x^*_3}& +x^*_4\frac{\partial}{\partial  x^*_4} && &\\
\X_4:=&x^*_4 \frac{\partial}{\partial  x^*_2} &+x^*_5\frac{\partial}{\partial  x^*_3}& & & &  \\
\X_5:= & & & & &\frac{\partial}{\partial x^*}& \\
\X_6:=& & & & & &\frac{\partial}{\partial y^*}.
\end{array}
\end{cases} \]
The definite matrix of $S_{G_9}$ is
\[M=\begin{bmatrix}
1 & & & & &  \\
& &x^*_2&x^*_4& & \\
& x^*_3& x^*_5&x^*_5 & & & \\
& & x^*_4 & & & \\
&x^*_5 & && &  \\
& & & & 1& \\
& & & & &  1\\
\end{bmatrix}.\]
Clearly, $\Rank(S_{G_9}) = \Rank(M)=6$. Furthermore, $\X_i$ is smooth all over $V$, $i=\overline{1,6}$. Now, we will show that $S_{{G_9}}$ produces $\mathcal{F}_{{G_9}}$, i.e. show that each K-orbit $\Omega$ from ${\mathcal{F}}_{{G_9}}$ is a maximal connected integrable submanifold corresponding to~$S_{{G_9}}$.

Firstly, we consider three vector fields $\X_1, \X_5, \X_6$. Obviously, the flows of $\X_1, \X_5$ and $\X_6$ are determined as follows
\begin{align*}
\theta^{\X_1}_{x^*_1-\alpha_1}: F \mapsto \theta^{\X_1}_{x^*_1-\alpha_1}(F):=(x^*_1, \alpha_2,\alpha_3, \alpha_4, \alpha_5, \alpha, \beta)\\
\theta^{\X_5}_{x^*-\alpha}: F \mapsto \theta^{\X_5}_{x^*-\alpha}(F):=(\alpha_1, \alpha_2,\alpha_3, \alpha_4, \alpha_5, x^*, \beta)\\
\theta^{\X_6}_{y^*-\beta}: F \mapsto \theta^{\X_6}_{y^*-\beta}(F):=(\alpha_1, \alpha_2,\alpha_3, \alpha_4, \alpha_5, \alpha, y^*).
\end{align*}

Next, we consider $\X_2:=x^*_3\frac{\partial}{\partial  x^*_3}+x^*_5\frac{\partial}{\partial  x^*_5}$. Assume that 

$$\varphi: x \mapsto \varphi(x)=(x^*_1(x), x^*_2(x), x^*_3(x), x^*_4(x), x^*_5(x), x^*(x), y^*(x))$$
be an integral curve of $\X_2$ from $F=\varphi(0)$, where $x \in (-\epsilon, \epsilon) \subset \R$ for some positive real number $\epsilon$. Then we have
\begin{align}\label{G9X2}
\hspace{0.5cm}	\varphi'(x)={\X_2}_{\varphi(x)}
\Leftrightarrow & \sum_{i=1}^{5}x^{*'}_i(x)\frac{\partial}{\partial x^*_i}+x^{*'}(x)\frac{\partial}{\partial x^*}+y^{*'}(x)\frac{\partial}{\partial y^*}=x^*_3(x)\frac{\partial}{\partial  x^*_3}+x^*_5(x)\frac{\partial}{\partial  x^*_5}\nonumber\\
\Leftrightarrow &  \begin{cases}
x^{*'}_1(x)= x^{*'}_2(x)= x^{*'}_4(x)=x^{*'}(x)= y^{*'}(x)=0\\
x^{*'}_3(x)=x^*_3(x)\\
x^{*'}_5(x)=x^*_5(x).\tag{16}
\end{cases}
\end{align}
Combining with condition $F=\varphi(0)$, Equation \eqref{G9X2} gives us
\begin{equation}\label{X2}
\hspace{0.5cm}x^*_1=\alpha_1,x^*_2= \alpha_2, x^*_3=\alpha_3e^x, x^*_4=\alpha_4, x^*_5=\alpha_5e^x, x^*=\alpha, y^*=\beta.\tag{17}
\end{equation}
Similarly, we consider $\X_3:=x^*_2\frac{\partial}{\partial  x^*_2} +x^*_5\frac{\partial}{\partial  x^*_3} +x^*_4\frac{\partial}{\partial  x^*_4}$. Assume that
$$\varphi: y \mapsto \varphi(y)=(x^*_1(y), x^*_2(y), x^*_3(y), x^*_4(y), x^*_5(y), x^*(y), y^*(y))$$
be an integral curve of $\X_3$ from $F=\varphi(0)$, where $y \in (-\epsilon, \epsilon) \subset \R$ for some positive real number $\epsilon$. Then we have
\begin{align}\label{G9X3}
\varphi'(y)={\X_3}_{\varphi(y)} \Leftrightarrow & \sum_{i=1}^{5}x^{*'}_i(y)\frac{\partial}{\partial x^*_i}+x^{*'}(y)\frac{\partial}{\partial x^*}+y^{*'}(y)\frac{\partial}{\partial y^*}=\nonumber\\ &\hspace{1.0cm}=x^*_2(y)\frac{\partial}{\partial  x^*_2} +x^*_5(y)\frac{\partial}{\partial  x^*_3} +x^*_4(y)\frac{\partial}{\partial  x^*_4}\nonumber\\
\Leftrightarrow &  \begin{cases}
x^{*'}_1(y)= x^{*'}_5(y)= x^{*'}(y)= y^{*'}(y)=0\\
x^{*'}_2(y)=x^*_2(y)\\
x^{*'}_3(y)=x^*_5(y)\\
x^{*'}_4(y)=x^*_4(y).\tag{18}
\end{cases}
\end{align}
Combining with condition $F=\varphi(0)$, Equation \eqref{G9X3} gives us
\begin{align}\label{X3}
&x^*_1=\alpha_1,\; x^*_2= \alpha_2e^y,\;  x^*_3=\alpha_5y+\alpha_3,\nonumber\\
&x^*_4=\alpha_4e^y,\; x^*_5=\alpha_5,\; x^*=\alpha, y^*=\beta. \tag{19}
\end{align}
Similarly, we consider $\X_4:=x^*_4 \frac{\partial}{\partial  x^*_2} +x^*_5\frac{\partial}{\partial  x^*_3}$. Assume that
$$\varphi: x_1 \mapsto \varphi(x_1)=(x^*_1(x_1), x^*_2(x_1), x^*_3(x_1), x^*_4(x_1), x^*_5(x_1), x^*(x_1), y^*(x_1))$$
be an integral curve of $\X_4$ from $F=\varphi(0)$, where $x_1 \in (-\epsilon, \epsilon) \subset \R$ for some positive real number $\epsilon$. Then we have
\begin{align}\label{G9X4}
\varphi'(x_1)={\X_4}_{\varphi(x_1)} \Leftrightarrow & \sum_{i=1}^{5}x^{*'}_i(x_1)\frac{\partial}{\partial x^*_i}+x^{*'}(x_1)\frac{\partial}{\partial x^*}+y^{*'}(x_1)\frac{\partial}{\partial y^*}=x^*_4(x_1) \frac{\partial}{\partial  x^*_2} +x^*_5(x_1)\frac{\partial}{\partial  x^*_3}\nonumber\\
\Leftrightarrow &  \begin{cases}
x^{*'}_1(x_1)= x^{*'}_4(x_1)=x^{*'}_5(x_1)= x^{*'}(x_1)= y^{*'}(x_1)=0\\
x^{*'}_2(x_1)=x^*_4(x_1)\\
x^{*'}_3(x_1)=x^*_5(x_1).\tag{20}
\end{cases} 
\end{align}
Combining with condition $F=\varphi(0)$, Equation \eqref{G9X4} gives us
\begin{align}\label{X4}
&x^*_1=\alpha_1,\; x^*_2= \alpha_2 +\alpha_4x_1 ,\; x^*_3=\alpha_3+\alpha_5x_1, \nonumber\\
&x^*_4=\alpha_4 , \; x^*_5=\alpha_5, \; x^*=\alpha,\; y^*=\beta.\tag{21}
\end{align}
According to (\ref{X2}), (\ref{X3}), (\ref{X4}), the flows of $\X_2, \X_3, \X_4$ are determined as follows
\begin{align*}
&\theta^{\X_2}_{x}: F \mapsto \theta^{\X_2}_{x}(F):=(\alpha_1, \alpha_2,\alpha_3e^x, \alpha_4, \alpha_5e^x, \alpha, \beta),\\
&\theta^{\X_3}_{y}: F \mapsto \theta^{\X_3}_{y}(F):=(\alpha_1, \alpha_2e^y,\alpha_3+\alpha_5y, \alpha_4e^y, \alpha_5, \alpha, \beta),\\
&\theta^{\X_4}_{x_1}: F \mapsto \theta^{\X_4}_{x_1}(F):=(\alpha_1, \alpha_2+\alpha_4x_1,\alpha_3+\alpha_5x_1, \alpha_4, \alpha_5, \alpha, \beta).
\end{align*}
Finally, we set $\theta=\theta^{\X_6}_{y^*-\beta}\circ\theta^{\X_5}_{x^*-\alpha}\circ\theta^{\X_4}_{x_1}\circ\theta^{\X_3}_{y}\circ\theta^{\X_2}_{x}\circ\theta^{\X_1}_{x^*_1-\alpha_1}$. From here~infer
\begin{align*}
\theta(F)=&\theta^{\X_6}_{y^*-\beta}\circ\theta^{\X_5}_{x^*-\alpha}\circ\theta^{\X_4}_{x_1}\circ\theta^{\X_3}_{y}\circ\theta^{\X_2}_{x}\circ\theta^{\X_1}_{x^*_1-\alpha_1}(F)\\
%		=&(x^*_1, \alpha_2e^y+\alpha_4e^yx_1, \alpha_3e^x+\alpha_5e^xy+\alpha_5e^xx_1, \alpha_4e^y, \alpha_5e^x, x^*, y^*)\\
%		=& (x^*_1, \alpha_2e^y+\alpha_4e^yx_1, \alpha_3e^x+\alpha_5(y+x_1)e^x, \alpha_4e^y, \alpha_5e^x, x^*, y^*)\\
=&(x^*_1, x^*_2, x^*_3, x^*_4, x^*_5, x^*, y^*)
\end{align*}
where 
\begin{align*}
&x^*_2= \alpha_2e^y+\alpha_4e^yx_1,\; x^*_3= \alpha_3e^x+\alpha_5(y+x_1)e^x, \\
&x^*_4= \alpha_4e^y,\; x^*_5= \alpha_5e^x, \; x^*_1, x^*,y^* \in \R.
\end{align*} 
By direct calculation, we get 
\[\frac{x^*_2}{x^*_4} - \frac{x^*_3}{x^*_5} + \ln\vert x^*_4 \vert = \frac{\alpha_2}{\alpha_4} - \frac{\alpha_3}{\alpha_5} + \ln\vert \alpha_4 \vert.\]
Hence $\{\theta(F) : x^*_1, x_1, x, y, x^*, y^* \in \R; \alpha_4x^*_4>0; \alpha_5x^*_5>0 \}\equiv\Omega_F$, i.e. $\Omega_F$ is a maximal connected integrable submanifold corresponding to $S_{G_9}$. Thus $S_{G_9}$ generate $\mathcal{F}_{G_9}$. In other words, $(V,\mathcal{F}_{G_9})$ is a 6-dimensional foliation.

We now turn to the second step of the proof. Namely, we have to show that the foliation $(V,\mathcal{F}_{G_9})$ is measurable in the sense of Connes. To prove that $(V, \mathcal{F}_{G_9})$ is measurable, we only need to choose some suitable pair ($\X, \mu$) on $V$ where $\X$ is some smooth 6-vector field defined on $V$, $\mu$ is some measure on $V$ such that $\X$ generates $S_{G_9}$ and $\mu$ is $\X$-invariant. 
Namely, we choose $\mu$ to be exactly the Lebegues measure on $V$ and set $\X: = \X_1 \wedge \X_2 \wedge \X_3 \wedge \X_4 \wedge \X_5 \wedge \X_6$.
Clearly, $\X$ is smooth, non-zero everywhere on $V$ and it is exactly a polyvector field of degree 6. Moreover $\X$ generates $S_{G_9}$. That is, when choosing on $(V, \mathcal{F}_{G_9})$ a suitable orientation, then $\X \in C^{\infty}{\bigl({\Lambda}^{6}(\mathcal{F})\bigr)}^{+}$. It is obvious that the invariance of the Lebegues measure $\mu$ with respect to $\X$ is equivalent to the invariance of $\mu$ for the $K$-representation that is restricted to the foliated submanifold $V$ in $\G_9^*$. For any $U(x_1, x_2, x_3, x_4, x_5, x, y) \in \G_9$, direct computation show that Jacobi's determinant
$J_U$ of differential mapping $K\bigl(\exp_{G_9}(U)\bigr)$ is a constant which depends only on U but do not depend on the coordinates of any point which moves in each $K$-orbit $\Omega \in \mathcal{F}_{G_9}$. In other words, the Lebegues measure $\mu$ is $\X$-invariant. The proof is complete for the case $G = G_9$. 

\item For $G = G^{\lambda}_{10}$.  For any $F(\alpha_1,\alpha_2, \alpha_3, \alpha_4, ,\alpha_5,\alpha,\beta) \in {\G^{\lambda}_{10}}^*$, by Theorem \ref{picture K-orbit}, the $K$-orbits $\Omega_F$ belongs to $\mathcal{F}_{G^{\lambda}_{10}}$ if and only if $\alpha_4\alpha_5\neq0$. Moreover we have 
\[\begin{array}{r}
\Omega_F=\{(x^*_1, x^*_2, x^*_3, x^*_4, x^*_5, x^*, y^*) \in \G^{\lambda*}_{10}: \frac{x^*_2}{x^*_4} - \frac{x^*_3}{x^*_5} + \ln \frac{\vert x^*_5 \vert}{{\vert x^*_4 \vert}^{\lambda}}   =\frac{\alpha_2}{\alpha_4} - \frac{\alpha_3}{\alpha_5} + \ln \frac{\vert \alpha_5 \vert}{{\vert \alpha_4 \vert}^{\lambda}},\\ \alpha_4 x^*_4 > 0, \, \alpha_5 x^*_5 >0 \}.\end{array}\]

On the open subset $V$ ($\subset\G^{\lambda *}_{10}$) with natural differential structure, we consider the following differential system $S_{G^{\lambda}_{10}} = \{\X_1, \X_2, \X_3, \X_4, \X_5, \X_6\}$ 
\[\hspace{1cm}\begin{cases}\begin{array}{lllllll}
\X_1:=\frac{\partial}{\partial x^*_1}& & & & & &\\
\X_2:=& & (x^*_3+x^*_5)\frac{\partial}{\partial  x^*_3}& &+x^*_5\frac{\partial}{\partial  x^*_5} & &\\
\X_3:=&x^*_2\frac{\partial}{\partial  x^*_2} &+\lambda x^*_3\frac{\partial}{\partial  x^*_3}& +x^*_4\frac{\partial}{\partial  x^*_4} &+\lambda x^*_5\frac{\partial}{\partial  x^*_5} & &\\
\X_4:=&x^*_4 \frac{\partial}{\partial  x^*_2} &+x^*_5\frac{\partial}{\partial  x^*_3}& & & &  \\
\X_5:= & & & & &\frac{\partial}{\partial x^*}& \\
\X_6:=& & & & & &\frac{\partial}{\partial y^*}.
\end{array}
\end{cases}  \]
The definite matrix of $S_{G^{\lambda}_{10}}$ is
\[M=\begin{bmatrix}
1 & & & & & \\
& & x^*_2& x^*_4& & \\
& x^*_3+x^*_5&\lambda x^*_3 &x^*_5 & & &  \\
& & x^*_4 & & & & \\
&x^*_5 &\lambda x^*_5 & & &  \\
& & &  & 1& \\
& & & & &  1\\
\end{bmatrix}.\]
Obviously, $\Rank(S_{G^{\lambda}_{10}}) = \Rank(M)=6$. Furthermore, $\X_i$ is smooth all over $V$, $i=\overline{1,6}$. Now, we will show that $S_{G^{\lambda}_{10}}$ produces $\mathcal{F}_{G^{\lambda}_{10}}$, i.e. show that each K-orbit $\Omega$ from ${\mathcal{F}}_{G^{\lambda}_{10}}$ is a maximal connected integrable submanifold corresponding to~$S_{G^{\lambda}_{10}}$.

Firstly, we consider three vector fields  $\X_1, \X_5, \X_6$. Obviously, the flows of $\X_1, \X_5$ and $\X_6$ are
\begin{align*}
\theta^{\X_1}_{x^*_1-\alpha_1}: F \mapsto \theta^{\X_1}_{x^*_1-\alpha_1}(F):=(x^*_1, \alpha_2,\alpha_3, \alpha_4, \alpha_5, \alpha, \beta)\\
\theta^{\X_5}_{x^*-\alpha}: F \mapsto \theta^{\X_5}_{x^*-\alpha}(F):=(\alpha_1, \alpha_2,\alpha_3, \alpha_4, \alpha_5, x^*, \beta)\\
\theta^{\X_6}_{y^*-\beta}: F \mapsto \theta^{\X_6}_{y^*-\beta}(F):=(\alpha_1, \alpha_2,\alpha_3, \alpha_4, \alpha_5, \alpha, y^*).
\end{align*}

Next, we consider $\X_2:= (x^*_3+x^*_5)\frac{\partial}{\partial  x^*_3} +x^*_5\frac{\partial}{\partial  x^*_5}$. Assume that 

$$\varphi: x \mapsto \varphi(x)=(x^*_1(x), x^*_2(x), x^*_3(x), x^*_4(x), x^*_5(x), x^*(x), y^*(x))$$
be an integral curve of $\X_2$ from $F=\varphi(0)$, where $x \in (-\epsilon, \epsilon) \subset \R$ for some positive real number $\epsilon$. Then we have
\begin{align}\label{G10X2}
\hspace{0.5cm}	\varphi'(x)={\X_2}_{\varphi(x)} \Leftrightarrow & \sum_{i=1}^{5}x^{*'}_i(x)\frac{\partial}{\partial x^*_i}+x^{*'}(x)\frac{\partial}{\partial x^*}+y^{*'}(x)\frac{\partial}{\partial y^*}= \nonumber\\
&\hspace{1.5cm}=(x^*_3(x)+x^*_5(x))\frac{\partial}{\partial  x^*_3} +x^*_5(x)\frac{\partial}{\partial  x^*_5} \nonumber\\
\Leftrightarrow &  \begin{cases}
x^{*'}_1(x)= x^{*'}_2(x)= x^{*'}_4(x)=x^{*'}(x)= y^{*'}(x)=0\\
x^{*'}_3(x)=x^*_3(x)+x^*_5(x)\\
x^{*'}_5(x)=x^*_5(x). \tag{22}
\end{cases} 
\end{align}
Combining with condition $F=\varphi(0)$, Equation \eqref{G10X2} gives us
\begin{align}\label{X2G10}
& x^*_1=\alpha_1, \; x^*_2= \alpha_2,\; x^*_3=\alpha_3e^x+\alpha_5e^x x, \nonumber\\
&x^*_4=\alpha_4,\; x^*_5=\alpha_5e^x,\; x^*=\alpha,\; y^*=\beta.\tag{23}
\end{align}
Similarly, we consider $\X_3:=x^*_2\frac{\partial}{\partial  x^*_2} +\lambda x^*_3\frac{\partial}{\partial  x^*_3} +x^*_4\frac{\partial}{\partial  x^*_4} +\lambda x^*_5\frac{\partial}{\partial  x^*_5}$

$\varphi: y \mapsto \varphi(y)=(x^*_1(y), x^*_2(y), x^*_3(y), x^*_4(y), x^*_5(y), x^*(y), y^*(y))$ be an integral curve of $\X_3$ from $F=\varphi(0)$, where $y \in (-\epsilon, \epsilon) \subset \R$ for some positive real number $\epsilon$. Then we have
\begin{align}\label{G10X3}
\hspace{0.5cm}\varphi'(y)={\X_3}_{\varphi(y)}\nonumber \Leftrightarrow & \sum_{i=1}^{5}x^{*'}_i(y)\frac{\partial}{\partial x^*_i}+x^{*'}(y)\frac{\partial}{\partial x^*}+y^{*'}(y)\frac{\partial}{\partial y^*}= \nonumber\\
&=x^*_2(y)\frac{\partial}{\partial  x^*_2} +\lambda x^*_3(y)\frac{\partial}{\partial  x^*_3} +x^*_4(y)\frac{\partial}{\partial  x^*_4} +\lambda x^*_5(y)\frac{\partial}{\partial  x^*_5}\nonumber\\
\Leftrightarrow &  \begin{cases}
x^{*'}_1(y)= x^{*'}(y)= y^{*'}(y)=0\\
x^{*'}_2(y)=x^*_2(y)\\
x^{*'}_3(y)=\lambda x^*_3(y)\\
x^{*'}_4(y)=x^*_4(y)\\
x^{*'}_5(y)=\lambda x^*_5(y). \tag{24}
\end{cases} 
\end{align}
Combining with condition $F=\varphi(0)$, Equation \eqref{G10X3} gives us
\begin{align}\label{X3G10}
&x^*_1=\alpha_1, \;x^*_2= \alpha_2e^y,\; x^*_3=\alpha_3e^{\lambda y},\; x^*_4=\alpha_4e^y,\nonumber\\
& x^*_5=\alpha_5e^{\lambda y}, \;x^*=\alpha, \;y^*=\beta. \tag{25}
\end{align}
Similarly, we consider $\X_4:=x^*_4\frac{\partial}{\partial  x^*_2} + x^*_5\frac{\partial}{\partial  x^*_3}$

$\varphi: x_1 \mapsto \varphi(x_1)=(x^*_1(x_1), x^*_2(x_1), x^*_3(x_1), x^*_4(x_1), x^*_5(x_1), x^*(x_1), y^*(x_1))$ be an integral curve of $\X_4$ from $F=\varphi(0)$, where $x_1 \in (-\epsilon, \epsilon) \subset \R$ for some positive real number $\epsilon$. Then we have
\begin{align}\label{G10X4}
\varphi'(x_1)={\X_4}_{\varphi(x_1)} \Leftrightarrow & \sum_{i=1}^{5}x^{*'}_i(x_1)\frac{\partial}{\partial x^*_i}+x^{*'}(x_1)\frac{\partial}{\partial x^*}+y^{*'}(x_1)\frac{\partial}{\partial y^*} =x^*_4(x_1)\frac{\partial}{\partial  x^*_2} + x^*_5(x_1)\frac{\partial}{\partial  x^*_3}\nonumber\\
\Leftrightarrow &  \begin{cases}
x^{*'}_1(x_1)=x^{*'}_4(x_1)=x^{*'}_5(x_1)= x^{*'}(x_1)= y^{*'}(x_1)=0\\
x^{*'}_2(x_1)=x^*_4(x_1)\\
x^{*'}_3(x_1)=x^*_5(x_1). \tag{26}
\end{cases} 
\end{align}
Combining with condition $F=\varphi(0)$, Equation \eqref{G10X4} gives us
\begin{align}\label{X4G10}
&x^*_1=\alpha_1,\; x^*_2= \alpha_2+\alpha_4x_1,\;  x^*_3=\alpha_3+\alpha_5x_1,\nonumber\\
& x^*_4=\alpha_4,\; x^*_5=\alpha_5,\; x^*=\alpha,\; y^*=\beta.\tag{27}
\end{align}
By (\ref{X2G10}), (\ref{X3G10}), (\ref{X4G10}), the flows of $\X_2, \X_3$ and  $\X_4$ are
\begin{align*}
&\theta^{\X_2}_{x}: F \mapsto \theta^{\X_2}_{x}(F):=(\alpha_1, \alpha_2,\alpha_3e^x+\alpha_5e^x x, \alpha_4, \alpha_5e^x, \alpha, \beta)\\
&\theta^{\X_3}_{y}: F \mapsto \theta^{\X_3}_{y}(F):=(\alpha_1, \alpha_2e^y,\alpha_3e^{\lambda y}, \alpha_4e^y, \alpha_5e^{\lambda y}, \alpha, \beta)\\
&\theta^{\X_4}_{x_1}: F \mapsto \theta^{\X_4}_{x_1}(F):=(\alpha_1, \alpha_2+\alpha_4x_1,\alpha_3+\alpha_5x_1, \alpha_4, \alpha_5, \alpha, \beta).
\end{align*}
Finally, we set $\theta=\theta^{\X_6}_{y^*-\beta}\circ\theta^{\X_5}_{x^*-\alpha}\circ\theta^{\X_4}_{x_1}\circ\theta^{\X_3}_{y}\circ\theta^{\X_2}_{x}\circ\theta^{\X_1}_{x^*_1-\alpha_1}$. From here~infer
\begin{align*}
\theta(F)=&\theta^{\X_6}_{y^*-\beta}\circ\theta^{\X_5}_{x^*-\alpha}\circ\theta^{\X_4}_{x_1}\circ\theta^{\X_3}_{y}\circ\theta^{\X_2}_{x}\circ\theta^{\X_1}_{x^*_1-\alpha_1}(F)\\
=&(x^*_1, x^*_2, x^*_3, x^*_4, x^*_5, x^*, y^*)
\end{align*}
where 
\begin{align*}
&x^*_2= \alpha_2e^y+\alpha_4e^yx_1, \; x^*_3= \alpha_3e^{\lambda y+x}+\alpha_5(x+x_1)e^{\lambda y+x},\\
&x^*_4= \alpha_4e^y,\;  x^*_5= \alpha_5e^{\lambda y+x}, \; x^*_1, x^*, y^* \in \R.
\end{align*}
By direct calculation, we get 
\[\frac{x^*_2}{x^*_4} - \frac{x^*_3}{x^*_5} + \ln \frac{\vert x^*_5 \vert}{{\vert x^*_4 \vert}^{\lambda}}  =  \frac{\alpha_2}{\alpha_4} - \frac{\alpha_3}{\alpha_5} + \ln \frac{\vert \alpha_5 \vert}{{\vert \alpha_4 \vert}^{\lambda}}.\]
Hence $\{\theta(F): x^*_1, x_1, x, y, x^*, y^* \in \R; \alpha_4x^*_4>0, \alpha_5x^*_5>0 \}\equiv\Omega_F$, i.e. $\Omega_F$ is a maximal connected integrable submanifold corresponding to $S_{G^{\lambda}_{10}}$. Thus $S_{G^{\lambda}_{10}}$ generate $\mathcal{F}_{G^{\lambda}_{10}}$. In other words, $(V,\mathcal{F}_{G^{\lambda}_{10}})$ is a 6-dimensional foliation.

We now turn to the second step of the proof. Namely, we have to show that the foliation $(V,\mathcal{F}_{G^{\lambda}_{10}})$ is measurable in the sense of Connes. To prove that $(V, \mathcal{F}_{G^{\lambda}_{10}})$ is measurable, we only need to choose some suitable pair ($\X, \mu$) on $V$ where $\X$ is some smooth 6-vector field defined on $V$, $\mu$ is some measure on $V$ such that $\X$ generates $S_{G^{\lambda}_{10}}$ and $\mu$ is $\X$-invariant. 
Namely, we choose $\mu$ to be exactly the Lebegues measure on $V$ and set $\X: = \X_1 \wedge \X_2 \wedge \X_3 \wedge \X_4 \wedge \X_5 \wedge \X_6$.
Clearly, $\X$ is smooth, non-zero everywhere on $V$ and it is exactly a polyvector field of degree 6. Moreover $\X$ generates $S_{G^{\lambda}_{10}}$. That is, when choosing on $(V, \mathcal{F}_{G^{\lambda}_{10}})$ a suitable orientation, then $\X \in C^{\infty}{\bigl({\Lambda}^{6}(\mathcal{F})\bigr)}^{+}$. It is obvious that the invariance of the Lebegues measure $\mu$ with respect to $\X$ is equivalent to the invariance of $\mu$ for the $K$-representation that is restricted to the foliated submanifold $V$ in $\G^{\lambda*}_{10}$. For any $U(x_1, x_2, x_3, x_4, x_5, x, y) \in \G^{\lambda}_{10}$, direct computation show that Jacobi's determinant
$J_U$ of differential mapping $K\bigl(\exp_{G^{\lambda}_{10}}(U)\bigr)$ is a constant which depends only on U but do not depend on the coordinates of any point which moves in each $K$-orbit $\Omega \in \mathcal{F}_{G^{\lambda}_{10}}$. In other words, the Lebegues measure $\mu$ is $\X$-invariant. The proof is complete for the case $G = G^{\lambda}_{10}$. 
\end{enumerate}	
The proof of the Theorem \ref{FormedFoliation} is complete.
 \end{proof}

We emphasize that, for any group $G \in \{G_{2}, G_{3},  G_4^{00}, G_{9}, G_{10}^{\lambda}\}$, the foliation $(V, \mathcal{F}_G)$ is established by maximal dimensional $K$-orbits of $G$ just like MD-foliations in \cite{V90, V93, VH09, VHT14}. Therefore, we give the following definition.
\vskip 0.5cm

%Definition 3.6
\begin{defn} For any $G \in \{G_{2}, G_{3},  G_4^{00}, G_{9}, G_{10}^{\lambda}\}$, the foliation $(V, \mathcal{F}_G)$ is called the generalized MD-foliation ($GMD$-foliation for short) associated to~$G$.
\end{defn}

%Remark 3.7
\begin{rem}\label{RemarkFoliation} It should be noted that for the set $\{G_{2}, G_{3},  G_4^{00}, G_{9}, G_{10}^{\lambda}\}$ above, by Theorem \ref{FormedFoliation}, we get exactly five families of measurable $GMD$-foliations on the same foliated manifold $V$ which is defined by Equality (\ref{FoliatedManifold}) in Item (vi) of Remark \ref{RemarkOrbits}. Namely, we have the set of $GMD$-foliations as follows
	\[\{(V,\mathcal{F}_{G_2}), (V,\mathcal{F}_{G_3}), (V,\mathcal{F}_{G_4^{00}}), (V,\mathcal{F}_{G_9}), (V,\mathcal{F}_{G_{10}^{\lambda}})\}.\]
	The topological classification of this set is given in the following theorem.
\end{rem}

%Theorem 3.8
\begin{thm}[Description and classification of topological types of $\boldsymbol{GMD}$-foliations]\label{TopoClass}
	The topology of  $GMD$-foliations has the following properties
	\begin{enumerate}[(1)]
		\item \label{TopoTypes} All $GMD$-foliations are topological equivalent. In other words, there exist only one topological type  families of $GMD$-foliations. We denote this type by $\mathcal{F}_{GMD}$.
	
		\item \label{Type} All $GMD$-foliations are trivial fibration with connected fibers on the discrete union of four copies of the real line $\R$. 
	\end{enumerate}
\end{thm}

%Proof of Theorem 3.8
\begin{proof} 	
	We will now prove two assertions of the theorem. 
	\begin{enumerate}[(a)]
		\item Firstly, we prove Assertion (\ref{TopoTypes}). %of Theorem \ref{TopoClass} 
		By Definition \ref{ToPo}, two foliations on the same foliated manifold $V$ are said to be topologically equivalent if there exists a homeomorphism $h$ of $V$ which sends leaves of the first foliation onto those of the second one. \\
		Let $h_1, h_2, h_3, h: V \rightarrow V$ be four maps defined as follows:
		\begin{align*}
		\hspace{1cm}&h_1(v) := (x^*_1,x^*_3, x^*_2, x^*_5, x^*_4, x^*, y^*),\\
		&h_2(v) := (x^*_1,x^*_2 - x^*_4 \ln \vert x^*_4 \vert, x^*_3, x^*_4, x^*_5, x^*, y^*),\\
		&h_3(v) := (x^*_1,x^*_2 + \lambda x^*_4 \ln \vert x^*_4 \vert, x^*_3 + x^*_5 \ln \vert x^*_5 \vert, x^*_4, x^*_5, x^*, y^*),\\
		&h(v) := (x^*_1,x^*_2 x^* _4, x^*_3, x^* _4, \frac{x^* _5}{x^*_4}, x^*, y^*)
		\end{align*}
%		and
%		\[h(x^*_1,x^*_2, x^*_3, x^*_4, x^*_5, x^*, y^*) := (x^*_1,x^*_2 x^* _4, x^*_3, x^* _4, \frac{x^* _5}{x^*_4}, x^*, y^*)\]
		for any $v: = (x^*_1,x^*_2, x^*_3, x^*_4, x^*_5, x^*, y^*) \in V$. It is clear that the considered maps are homeomorphic. 
		
%		Let $v, \widetilde{v}$ be two arbitrary points in $V$ as follows
%		\[v = (x^*_1,x^*_2, x^*_3, x^*_4, x^*_5, x^*, y^*), \, \widetilde{v} = (\widetilde{x}^*_1, \widetilde{x}^*_2, \widetilde{x}^*_3, \widetilde{x}^*_4, \widetilde{x}^*_5, \widetilde{x}^*, \widetilde{y}^*).\]		
		\begin{itemize}
		\item Firstly, we take an arbitrary leaf $L$ of $(V, \mathcal{F}_{G_2})$. Without loss of generality, assume $L \subset V_{++} \subset V$ (see Equalities (\ref{SS1}) and (\ref{SS2}) in Item (vi) of Remark \ref{RemarkOrbits}), i.e. $L$ is determined as follows
		\[L  = \{v \in V: x^*_2-\frac{x^*_3 x^*_4}{x^*_5} = c; \, \, x^*_4 > 0, \, x^*_5 > 0\}\]
		where $c$ is some constant, $c \in \R$. Then for $(V, \mathcal{F}_{G_4^{00}})$, we consider the leaf $\widetilde{L}$ ($ \subset V_{++} \subset V$) determined as follows
		\[\widetilde{L} = \{\widetilde{v} \in V: \widetilde{x}^*_3-\frac{\widetilde{x}^*_2 \widetilde{x}^*_5}{\widetilde{x}^*_4} =c; \, \, \widetilde{x}^*_4>0,\, \widetilde{x}^*_5>0\}. \]
		For any $\widetilde{v} = (\widetilde{x}^*_1, \widetilde{x}^*_2, \widetilde{x}^*_3, \widetilde{x}^*_4, \widetilde{x}^*_5, \widetilde{x}^*, \widetilde{y}^*) \in V$, looking back at the homeomorphism $h_1: V \to V$, it is plain that $h_1(v) = \widetilde{v}$ is equivalent to
		 \begin{align} \notag \widetilde{v} =
		 	  & \, \,(x^*_1,x^*_3, x^*_2, x^*_5, x^*_4, x^*, y^*)\\		
		 	\notag \Leftrightarrow &  \, \, x^*_1 = \widetilde{x}^*_1, x^*_2 = \widetilde{x}^*_3, x^*_3 = \widetilde{x}^*_2, x^*_4 = \widetilde{x}^*_5, x^*_5 = \widetilde{x}^*_4, x^* = \widetilde{x}^*, y^* = \widetilde{y}^*
%		 	\notag \Leftrightarrow &  \, \, h_1(v) =  \widetilde{v} \in \widetilde{L}
		 	%		 	\notag \Leftrightarrow &  ad_{X}\circ ad_{Y} = ad_{Y}\circ
		 	%		 	ad_{X}. 
		 \end{align}	
		 Therefore
		 \begin{align} \notag v \in L \Leftrightarrow & \,\,  x^*_2-\frac{x^*_3 x^*_4}{x^*_5} = c; \, \, x^*_4 > 0, \, x^*_5 > 0\\		
		 	\notag \Leftrightarrow & \,\, \widetilde{x}^*_3-\frac{\widetilde{x}^*_2 \widetilde{x}^*_5}{\widetilde{x}^*_4} =c; \, \, \widetilde{x}^*_4>0, \, \widetilde{x}^*_5>0\\
		 	\notag \Leftrightarrow &  \, \, h_1(v) = \widetilde{v} \in \widetilde{L}
%		 	\notag \Leftrightarrow &  ad_{X}\circ ad_{Y} = ad_{Y}\circ
%		 	ad_{X}. 
	 	\end{align}		 
		i.e. $h_1(L) = \widetilde{L}$ for $L \subset V_{++}$. Similarly, $h_1(L) = \widetilde{L}$ for $L$ in $V_{-+}$, $V_{--}$, $V_{+-}$ and $h_1$ sends the leaves of $(V, \mathcal{F}_{G_2})$ onto those of $(V, \mathcal{F}_{G_4^{00}})$.

		\item Next, we take an arbitrary leaf $L$ of $(V, \mathcal{F}_{G_3})$. Without loss of generality, assume $L \subset V_{++} \subset V$, i.e. $L$ is determined as follows
		\[L  = \{v \in V: \frac{x^*_2}{x^*_4} - \frac{x^*_3}{x^*_5}=c; \, \, x^*_4 > 0, \, x^*_5 > 0\}\]
		where $c$ is some constant, $c \in \R$. Then for $(V, \mathcal{F}_{G_9})$, we consider the leaf $\widetilde{L}$ ($ \subset V_{++} \subset V$) determined as follows
		\[\widetilde{L} = \{\widetilde{v} \in V: \frac{\widetilde{x}^*_2}{\widetilde{x}^*_4} - \frac{\widetilde{x}^*_3}{\widetilde{x}^*_5}+\ln \vert \widetilde{x}^*_4 \vert=c; \, \, \widetilde{x}^*_4>0,\, \widetilde{x}^*_5>0\}. \]
		For any $\widetilde{v} = (\widetilde{x}^*_1, \widetilde{x}^*_2, \widetilde{x}^*_3, \widetilde{x}^*_4, \widetilde{x}^*_5, \widetilde{x}^*, \widetilde{y}^*) \in V$, looking back at the homeomorphism $h_2: V \to V$, it is plain that $h_2(v) = \widetilde{v}$ is equivalent to
		\begin{align} \notag \widetilde{v} =
		& \, \,(x^*_1,x^*_2 - x^*_4 \ln \vert x^*_4 \vert, x^*_3, x^*_4, x^*_5, x^*, y^*)\\		
		\notag \Leftrightarrow &  \, \, x^*_1 = \widetilde{x}^*_1, x^*_2 = \widetilde{x}^*_2 + \widetilde{x}^*_4 \ln \vert \widetilde{x}^*_4 \vert, x^*_3 = \widetilde{x}^*_3, x^*_4 = \widetilde{x}^*_4, x^*_5 = \widetilde{x}^*_5, x^* = \widetilde{x}^*, y^* = \widetilde{y}^*
		%		 	\notag \Leftrightarrow &  \, \, h_1(v) =  \widetilde{v} \in \widetilde{L}
		%		 	\notag \Leftrightarrow &  ad_{X}\circ ad_{Y} = ad_{Y}\circ
		%		 	ad_{X}. 
		\end{align}	
		Therefore
		\begin{align} \notag v \in L \Leftrightarrow & \,\,  \frac{x^*_2}{x^*_4} - \frac{x^*_3}{x^*_5}=c; \, \, x^*_4 > 0, \, x^*_5 > 0\\		
		\notag \Leftrightarrow & \,\, \frac{\widetilde{x}^*_2}{\widetilde{x}^*_4} - \frac{\widetilde{x}^*_3}{\widetilde{x}^*_5}+\ln \vert \widetilde{x}^*_4 \vert=c; \, \, \widetilde{x}^*_4>0, \, \widetilde{x}^*_5>0\\
		\notag \Leftrightarrow &  \, \, h_2(v) = \widetilde{v} \in \widetilde{L}
		%		 	\notag \Leftrightarrow &  ad_{X}\circ ad_{Y} = ad_{Y}\circ
		%		 	ad_{X}. 
		\end{align}		 
		i.e. $h_2(L) = \widetilde{L}$ for $L \subset V_{++}$. Similarly, $h_2(L) = \widetilde{L}$ for $L$ in $V_{-+}$, $V_{--}$, $V_{+-}$. Thus, $h_2$ sends the leaves of $(V, \mathcal{F}_{G_3})$ onto those of $(V, \mathcal{F}_{G_9})$.
		
		\item  Similarly, we take an arbitrary leaf $L$ of $(V, \mathcal{F}_{G_3})$. Without loss of generality, assume $L \subset V_{++} \subset V$, i.e. $L$ is determined as follows
		\[L  = \{v \in V: \frac{x^*_2}{x^*_4} - \frac{x^*_3}{x^*_5}=c\}; \, \, x^*_4 > 0, \, x^*_5 > 0\}\]
		where $c$ is some constant, $c \in \R$. Then for $(V, \mathcal{F}_{G^{\lambda}_{10}})$, we consider the leaf $\widetilde{L}$ ($ \subset V_{++} \subset V$) determined as follows
		\[\widetilde{L} = \{\widetilde{v} \in V: \frac{\widetilde{x}^*_2}{\widetilde{x}^*_4} - \frac{\widetilde{x}^*_3}{\widetilde{x}^*_5}+ \ln \frac{\vert \widetilde{x}^*_5 \vert}{{\vert \widetilde{x}^*_4 \vert}^{\lambda}} =c; \, \, \widetilde{x}^*_4>0,\, \widetilde{x}^*_5>0\}. \]
		For any $\widetilde{v} = (\widetilde{x}^*_1, \widetilde{x}^*_2, \widetilde{x}^*_3, \widetilde{x}^*_4, \widetilde{x}^*_5, \widetilde{x}^*, \widetilde{y}^*) \in V$, looking back at the homeomorphism $h_3: V \to V$, it is plain that $h_3(v) = \widetilde{v}$ is equivalent to
		\begin{align*}  \widetilde{v} =
		& \;(x^*_1,x^*_2 + \lambda x^*_4 \ln \vert x^*_4 \vert, x^*_3 + x^*_5 \ln \vert x^*_5 \vert, x^*_4, x^*_5, x^*, y^*)\\		
		 \Leftrightarrow &  \; x^*_1 = \widetilde{x}^*_1, x^*_2 = \widetilde{x}^*_2 - \lambda \widetilde{x}^*_4 \ln \vert \widetilde{x}^*_4 \vert, x^*_3 = \widetilde{x}^*_3-  \widetilde{x}^*_5\ln \vert \widetilde{x}^*_5 \vert,\\
	& \; x^*_4 = \widetilde{x}^*_4, x^*_5 = \widetilde{x}^*_5, x^* = \widetilde{x}^*, y^* = \widetilde{y}^*.
			\end{align*}	
		For this reason
		\begin{align} \notag v \in L \Leftrightarrow & \,\,  \frac{x^*_2}{x^*_4} - \frac{x^*_3}{x^*_5}=c; \, \, x^*_4 > 0, \, x^*_5 > 0\\		
		\notag \Leftrightarrow & \,\, \frac{\widetilde{x}^*_2}{\widetilde{x}^*_4} - \frac{\widetilde{x}^*_3}{\widetilde{x}^*_5}+ \ln \frac{\vert \widetilde{x}^*_5 \vert}{{\vert \widetilde{x}^*_4 \vert}^{\lambda}} =c; \, \, \widetilde{x}^*_4>0, \, \widetilde{x}^*_5>0\\
		\notag \Leftrightarrow &  \, \, h_3(v) = \widetilde{v} \in \widetilde{L}
		%		 	\notag \Leftrightarrow &  ad_{X}\circ ad_{Y} = ad_{Y}\circ
		%		 	ad_{X}. 
		\end{align}		 
		i.e. $h_3(L) = \widetilde{L}$ for $L \subset V_{++}$. Similarly, $h_3(L) = \widetilde{L}$ for $L$ in $V_{-+}$, $V_{--}$, $V_{+-}$. Thus, $h_3$ sends the leaves of $(V, \mathcal{F}_{G_3})$ onto those of $(V, \mathcal{F}_{G^{\lambda}_{10}})$.
		
		\item  Finally, we take an arbitrary leaf $L$ of $(V, \mathcal{F}_{G_2})$. Without loss of generality, assume $L \subset V_{++} \subset V$, i.e. $L$ is determined as follows
		\[L  = \{v \in V: x^*_2-\frac{x^*_3 x^*_4}{x^*_5} = c; \, \, x^*_4 > 0, \, x^*_5 > 0\}\]
		where $c$ is some constant, $c \in \R$. Then for $(V, \mathcal{F}_{G_3})$, we consider the leaf $\widetilde{L}$ ($ \subset V_{++} \subset V$) determined as follows
		\[\widetilde{L} = \{\widetilde{v} \in V: \frac{\widetilde{x}^*_2}{\widetilde{x}^*_4} - \frac{\widetilde{x}^*_3}{\widetilde{x}^*_5}=c; \, \, \widetilde{x}^*_4>0,\, \widetilde{x}^*_5>0\}. \]
		For any $\widetilde{v} = (\widetilde{x}^*_1, \widetilde{x}^*_2, \widetilde{x}^*_3, \widetilde{x}^*_4, \widetilde{x}^*_5, \widetilde{x}^*, \widetilde{y}^*) \in V$, looking back at the homeomorphism $h: V \to V$, it is plain that $h(v) = \widetilde{v}$ is equivalent to
		\begin{align} \notag \widetilde{v} =
		& \, \,(x^*_1,x^*_2x^*_4, x^*_3, x^*_4, \frac{x^*_5}{x^*_4}, x^*, y^*)\\		
		\notag \Leftrightarrow &  \, \, x^*_1 = \widetilde{x}^*_1, x^*_2 = \frac{\widetilde{x}^*_2 }{\widetilde{x}^*_4}, x^*_3 = \widetilde{x}^*_3, x^*_4 = \widetilde{x}^*_4, x^*_5 = \widetilde{x}^*_4 \widetilde{x}^*_5, x^* = \widetilde{x}^*, y^* = \widetilde{y}^*.
		\end{align}	
		Hence
		\begin{align} \notag v \in L \Leftrightarrow & \,\,  x^*_2-\frac{x^*_3 x^*_4}{x^*_5} = c; \, \, x^*_4 > 0, \, x^*_5 > 0\\		
		\notag \Leftrightarrow & \,\, \frac{\widetilde{x}^*_2}{\widetilde{x}^*_4} - \frac{\widetilde{x}^*_3}{\widetilde{x}^*_5}=c; \, \, \widetilde{x}^*_4>0, \, \widetilde{x}^*_5>0\\
		\notag \Leftrightarrow &  \, \, h(v) = \widetilde{v} \in \widetilde{L}
		\end{align}		 
		i.e. $h(L) = \widetilde{L}$ for $L \subset V_{++}$. Similarly, $h(L) = \widetilde{L}$ for $L$ in $V_{-+}$, $V_{--}$, $V_{+-}$. Therefore, $h$ sends the leaves of $(V, \mathcal{F}_{G_2})$ onto those of $(V, \mathcal{F}_{G_3})$.
	\end{itemize}
	Thus, two foliations $(V, \mathcal{F}_{G_2}), (V, \mathcal{F}_{G_4^{00}})$ are topological equivalent, three foliations $(V, \mathcal{F}_{G_3})$, $(V, \mathcal{F}_{G_9})$, $(V, \mathcal{F}_{G_{10}^{\lambda}})$ also define the same topological type and two foliations $(V, \mathcal{F}_{G_2}),$ $ (V, \mathcal{F}_{G_3})$ are also topological equivalent.  It follows that all GMD-foliations define the same topological type and Assertion \ref{TopoTypes} is proven.
		
		\item Now we prove Assertion (\ref{Type}) of Theorem \ref{TopoClass}. Since all GMD-foliations are topological equivalent, we only need to prove this assertion for GMD-foliation $(V, \mathcal{F}_{G_2})$.
		
		Let us recall that a subset $W$ in the foliated manifold $V$ of any foliation $(V, \mathcal{F})$ is said to be saturated (with respect to the foliation) if (and only if) every leaf $L$ such that $L \cap W \neq \varnothing$ then $L \subset W$. For any  saturated submanifold of $V$, the family of all leaves $L$ of $(V, \mathcal{F})$ such that $L \subset W$ forms a new foliation on $W$ denoted by $(W, \mathcal{F}_W)$ and it is called the restriction or the subfoliation of $(V, \mathcal{F})$ on $W$. 
		
		We look back at Equality (\ref{FoliatedManifold}) in Item (vi) of Remark \ref{RemarkOrbits}. This equality represents the foliated manifold $V$ of all GMD-foliations as the discrete union of four open subsets $V_{++}, \, V_{-+}, \, V_{--}, \, V_{+-}$ that are given by Equalities (\ref{SS1}) and (\ref{SS2}) in Item (vi) of Remark \ref{RemarkOrbits}. Evidently, these subsets are four connected components of $V$ which are saturated with respect to all GMD-foliations. 
		
		Next, we consider the GMD-foliation $(V, \mathcal{F}_{G_2})$. For convenience, we will denote its restrictions on the open saturated submanifolds $V_{++},$  $V_{-+},$ $V_{--},$ $V_{+-}$ of the foliated manifold $V$ by $\mathcal{F}_{++}, \, \mathcal{F}_{-+}, \mathcal{F}_{--}, \mathcal{F}_{+-}$, respectively. 
		
		Let $p: V \rightarrow \R$ be the map that defines as follows
		\[p(x^*_1,x^*_2, x^*_3, x^*_4, x^*_5, x^*, y^*) := x^*_2 - \frac{x^*_3x^*_4}{x^*_5}\]
		for any $(x^*_1,x^*_2, x^*_3, x^*_4, x^*_5, x^*, y^*) \in V$. It can be verified that $p_{++} : = h\vert _{V_{++}}$ is a submersion and $p_{++}: V_{++} \rightarrow \R$ is a fibration on the real line $\R$ with connected (and simply connected) fibers. Moreover, $\mathcal{F}_{++}$ comes from this fibration. Similarly, foliations $\mathcal{F}_{-+}, \mathcal{F}_{--}, \mathcal{F}_{+-}$ are also come from fibrations on $\R$ with connected (and simply connected) fibers. Therefore, $(V, \mathcal{F}_{G_2})$ comes from a fibration with connected fibers on $\R \sqcup \R \sqcup \R \sqcup \R$ and so is the  type $\mathcal{F}_{GMD}$. The proof is~complete.
	\end{enumerate}
\end{proof}
As an immediate consequence of Theorem \ref{TopoClass} and the result of Connes in \cite[Section 5, p. 16]{Con82}, we have the following result.

%Corollary 3.9
\begin{cor}\label{Cor39}
	The Connes' $C^*$-algebras of $GMD$-foliations are determined as follows
	\[C^*(\mathcal{F}_{GMD}) \cong \bigl(C_0(\R) \oplus C_0(\R) \oplus C_0(\R) \oplus C_0(\R)\bigr) \otimes \mathcal{K}. \]
	where $C_0(\R)$ is the algebra of continuous complex-valued functions defined on $\R$ vanishing at infinity and $\mathcal{K}$ denotes the $C^*$-algebra of compact operators on an (infinite dimensional) separable Hibert space.
\end{cor}

%Section 4
\section{CONCLUSION}
%\subsection*{Concluding remark}
We have considered exponential, connected and simply connected Lie groups corresponding to Lie algebras of dimension 7 with the nilradical $\g_{5,2}$ and 4-dimensional derived ideals. The main results of the paper are as follows:
First, we give a description of the picture of maximal dimensional $ K $-orbits of considered Lie groups as well as the geometrical characteristics of this maximal dimensional $ K $-orbits (see Theorem \ref{picture K-orbit} and Remark \ref{RemarkOrbits});	
Second, we proved that the family of all generic maximal dimensional $ K $-orbits of the considered Lie groups forms measurable foliations (in the sense of Connes). We call them  GMD-foliations;	
Third, the topological classification of all  GMD-foliations is given (see Theorem \ref{TopoClass}) and their Connes' C*-algebras are described (see Corollary \ref{Cor39}).
Furthermore, the proposed method can be applied to other Lie groups corresponding to the remaining Lie algebras listed in Subsection \ref{ListLieAlgebras}.
\vspace{1cm}
% ------------------------------------------------------------------------

\noindent\textbf{Funding} Tuyen T. M. Nguyen was funded by Vingroup Joint Stock Company and supported by the Domestic Master/ PhD Scholarship Programme of Vingroup Innovation Foundation (VINIF), Vingroup Big Data Institute (VINBIGDATA), VINIF.2020.TS.46, and by the project SPD2019.01.37.

% ------------------------------------------------------------------------
\end{document}